\documentclass[11pt,a4paper,reqno]{amsart}
\usepackage[english]{babel}
\usepackage[T1]{fontenc}
\usepackage{verbatim}
\usepackage{palatino}
\usepackage{amsmath}
\usepackage{mathabx}
\usepackage{amssymb}
\usepackage{amsthm}
\usepackage{amsfonts}
\usepackage{graphicx}
\usepackage{esint}
\usepackage{color}
\usepackage{mathtools}

\usepackage[colorlinks = true, citecolor = black]{hyperref}
\title[Metric rectifiability of $\He$-regular surfaces]{Metric rectifiability of $\He$-regular surfaces with H\"older continuous horizontal normal}
\author{Daniela Di Donato}
\author{Katrin F\"assler}
\author{Tuomas Orponen}
\address{Department of Mathematics and Statistics\\
University of Jyv\"{a}skyl\"{a}, P.O. Box 35 (MaD)\\FI-40014 University of Jyv\"{a}skyl\"{a}\\
Finland. \underline{Current address:} SISSA\\ Trieste\\ Italy}
\address{Department of Mathematics and Statistics\\
University of Jyv\"{a}skyl\"{a}, P.O. Box 35 (MaD)\\FI-40014 University of Jyv\"{a}skyl\"{a}\\
Finland}
\address{Department of Mathematics and Statistics\\ University of Helsinki,
P.O. Box 68 (Pietari Kalmin katu 5)\\
FI-00014 University of Helsinki\\
Finland. \underline{Current address:} Department of Mathematics and Statistics\\
University of Jyv\"{a}skyl\"{a}\\
Finland} \email{ddidonat@sissa.it} \email{katrin.s.fassler@jyu.fi}
\email{tuomas.t.orponen@jyu.fi}
\date{\today}
\thanks{D.D.D. is partially supported by the Academy of Finland (grant
288501
`\emph{Geometry of subRiemannian groups}' and by grant
322898
`\emph{Sub-Riemannian Geometry via Metric-geometry and Lie-group Theory}')
and by the European Research Council
 (ERC Starting Grant 713998 GeoMeG `\emph{Geometry of Metric Groups}').  K.F. is supported by the Academy of Finland trough the grant 321696
`Singular integrals, harmonic functions, and boundary regularity
in Heisenberg groups'. T.O. is supported by the Academy of Finland
via the project \emph{Quantitative rectifiability in Euclidean and
non-Euclidean spaces}, grant Nos. 309365, 314172.}

\newcommand{\R}{\mathbb{R}}
\newcommand{\W}{\mathbb{W}}
\newcommand{\He}{\mathbb{H}}
\newcommand{\N}{\mathbb{N}}

\newcommand{\Z}{\mathbb{Z}}

\newcommand{\calD}{\mathcal{D}}
\newcommand{\calH}{\mathcal{H}}

\newcommand{\calQ}{\mathcal{Q}}

\newcommand{\spt}{\operatorname{spt}}

\newcommand{\diam}{\operatorname{diam}}
\newcommand{\card}{\operatorname{card}}
\newcommand{\dist}{\operatorname{dist}}

\newcommand{\sgn}{\operatorname{sgn}}

\newcommand{\V}{\mathbb{V}}
\newcommand{\Tan}{\mathrm{Tan}}

\def\Barint_#1{\mathchoice
          {\mathop{\vrule width 6pt height 3 pt depth -2.5pt
                  \kern -8pt \intop}\nolimits_{#1}}%
          {\mathop{\vrule width 5pt height 3 pt depth -2.6pt
                  \kern -6pt \intop}\nolimits_{#1}}%
          {\mathop{\vrule width 5pt height 3 pt depth -2.6pt
                  \kern -6pt \intop}\nolimits_{#1}}%
          {\mathop{\vrule width 5pt height 3 pt depth -2.6pt
                  \kern -6pt \intop}\nolimits_{#1}}}

\numberwithin{equation}{section}

\theoremstyle{plain}
\newtheorem{thm}[equation]{Theorem}

\newtheorem{lemma}[equation]{Lemma}

\newtheorem{ex}[equation]{Example}

\newtheorem{proposition}[equation]{Proposition}

\newtheorem*{questions}{Questions}

\theoremstyle{definition}

\newtheorem{definition}[equation]{Definition}

\theoremstyle{remark}
\newtheorem{remark}[equation]{Remark}

\newcommand{\nref}[1]{(\hyperref[#1]{#1})}

\begin{document}

\begin{abstract}
Two definitions for the rectfiability of hypersurfaces in
Heisenberg groups $\mathbb{H}^n$ have been proposed: one based on
$\He$-regular surfaces, and the other on Lipschitz images of
subsets of codimension-$1$ vertical subgroups. The equivalence between these notions
remains an open problem. Recent partial results are due to Cole-Pauls, Bigolin-Vittone, and Antonelli-Le Donne.

This paper makes progress in one direction: the metric
Lipschitz rectifiability of $\He$-regular surfaces. We prove that $\He$-regular surfaces in $\mathbb{H}^{n}$ with $\alpha$-H\"older
continuous horizontal normal, $\alpha > 0$, are metric bilipschitz rectifiable.
This improves on the work by Antonelli-Le Donne,
where the same conclusion was obtained for $C^{\infty}$-surfaces.

In $\mathbb{H}^{1}$, we prove a slightly stronger result: every
codimension-$1$ intrinsic Lipschitz graph with an $\epsilon$ of extra regularity
in the vertical direction is metric bilipschitz rectifiable. All the proofs in the paper are based on a new general criterion for finding
bilipschitz maps between "big pieces" of metric spaces.
\end{abstract}

\maketitle

\tableofcontents

\section{Introduction}\label{s:introd}

This paper concerns the relationship between two
notions of \emph{codimension-$1$ rectifiability} in the Heisenberg
group $(\He^n,d_{\He}) = (\R^{2n+1},\cdot,d_{\He})$, where
"$\cdot$" is the group product
\begin{displaymath}
(x_1,\ldots,x_{2n},t) \cdot (x_1',\ldots,x_{2n}',t') =
\left(\sum_{i=1}^{2n} x_i + x_i',t+t'+\tfrac{1}{2}\sum_{i=1}^n x_i
x_{n+i}'-x_{n+i}x_i'\right)\in \mathbb{R}^{2n}\times\mathbb{R},
\end{displaymath} and $d_{\He}$ is the Kor\'anyi distance
$d_{\He}(p,q) := \|q^{-1} \cdot p\|$ (with $\|(x,t)\| :=
\sqrt[4]{|x|^4 + 16t^{2}}$ for $(x,t)\in \mathbb{R}^{2n}\times
\mathbb{R}$). Metric notions in $\He^{n}$, notably Hausdorff measures, are defined using the metric $d_{\He}$. Metric notions in $\R^{n}$ are defined using the standard
Euclidean distance.

In $\R^{n}$, the notion of rectifiability can be defined in two
equivalent ways. For $0 < m < n$, an $\calH^{m}$ measurable set $E
\subset \R^{n}$ is called \emph{$m$-rectifiable} if $\calH^{m}$
almost all of $E$ can be covered by either \begin{enumerate}
\item countably many Lipschitz $m$-images, or
\item countably many Lipschitz $m$-graphs.
\end{enumerate}
Here, a \emph{Lipschitz $m$-image} means a Lipschitz image of
$\R^{m}$, while a \emph{Lipschitz $m$-graph} means a set of the
form $\{v + A(v) : v \in V\}$, where $V \subset \R^{n}$ is an
$m$-dimensional subspace, and $A \colon V \to V^{\perp}$ is a
Lipschitz map. The equivalence of "Lipschitz image rectifiability" and "Lipschitz graph rectifiability" is well-known. In particular, Lipschitz $m$-graphs are trivially Lipschitz $m$-images, since $v \mapsto v + A(v)$ is Lipschitz whenever $A$ is.

We then discuss the analogues of these notions in $\He^{n}$.
Recall first that $(\mathbb{H}^n,d_{\He})$ is a metric space of
Hausdorff dimension $2n+2$. A common notion of codimension-$1$
rectifiability (see \cite[Definition 4.33]{MR2836591}) is
\emph{intrinsic Lipschitz graph (iLG) rectifiability}:
an $\calH^{2n + 1}$ measurable set $E \subset \He^{n}$ is called iLG rectifiable if $\calH^{2n+1}$ almost all of $E$ can be covered by countably many
iLGs over \emph{vertical subgroups of codimension $1$}. Here, vertical subgroups refer to codimension-$1$ subspaces of $\mathbb{R}^{2n+1}$
containing the $t$-axis, cf.\ Section \ref{s:C1alpha}, while iLGs were introduced by Franchi, Serapioni, and Serra
Cassano \cite{FSS} in 2006. They are natural $\He^{n}$
counterparts of Lipschitz graphs in $\R^{n}$, see Definition
\ref{d:intrinsicGraphs}.

The notion of \emph{Lipschitz image (LI)
rectifiability} in $\mathbb{H}^n$ was first studied by
Pauls \cite{MR2048183} in 2004. An
$\calH^{2n+1}$ measurable set $E \subset \He^n$ is called
LI rectifiable if $\calH^{2n+1}$ almost all of $E$ can be
covered by countably many Lipschitz images of closed subsets of codimension-$1$
vertical subgroups. All of these subgroups (for $n \geq 1$ fixed) are isometrically isomorphic to each other. If $n=1$, they are further isometrically isomorphic
to the \emph{parabolic plane}
\begin{displaymath} \Pi := (\R^{2},+,\| \cdot\| ), \quad \text{where} \quad \|(y,t)\| := \sqrt[4]{y^{4} + 16t^{2}}, \end{displaymath}
and if $n\geq
2$, they are
isometrically isomorphic to $(\mathbb{H}^{n-1}\times
\mathbb{R},\cdot_{\mathbb{H}^{n-1}\times \mathbb{R}},\|\cdot\|)$
with
\begin{equation}\label{eq:HeisxR}
((z,t),s)\cdot_{\mathbb{H}^{n-1}\times \mathbb{R}}
((z',t'),s')=\big(\big(z+z',t+t'+\tfrac{1}{2}\sum_{i=1}^{n-1}
(z_iz_{n-1+i}'-z_i'z_{n-1+i}) \big),s+s'\big)
\end{equation}
and
\begin{displaymath}
\|((z,t),s)\| = \sqrt[4]{|(z,s)|^4 + 16 t^2}.
\end{displaymath}
 So, $E$ is LI rectifiable if and only if $\calH^{2n+1}$ almost
all of $E$ can be covered by countably many Lipschitz images of
closed subsets of $\Pi$ (if $n=1$) or
$\mathbb{H}^{n-1}\times\mathbb{R}$ (if $n\geq 2$). The metric
induced by $\|\cdot\|$ in $\Pi$ is denoted $d_{\Pi}$, and in
$\mathbb{H}^{n-1}\times \mathbb{R}$ by $d_{\mathbb{H}^{n-1}\times
\mathbb{R}}$.

The connection between iLG and LI rectifiability in $\He^n$ is
poorly understood. It is neither known if (a) LIs of vertical subgroups are iLG
rectifiable, nor if (b) iLGs are LI rectifiable. It may appear surprising that question (b) is open: after all, to show that Lipschitz $m$-graphs in $\R^{n}$
are Lipschitz $m$-images, one only needed to observe that the graph map $v \mapsto v + A(v)$
is Lipschitz whenever $A$ is. In $\He^{n}$, this argument fails completely. We will discuss the matter further in a moment.

The purpose of this paper is to make progress in question (b). In brief, we will show intrinsic $C^{1,\alpha}$-graphs in $\He^{n}$ are LI rectifiable for any $\alpha > 0$. In $\He^{1}$, we can say something a little better. For precise statements, see Theorems \ref{mainIntro}, \ref{mainQualitative}, and \ref{mainIntroVertical}. Before formulating these new results in detail, we define our objects of study more carefully, and describe some previous work on the topic.
\begin{definition}\label{d:intrinsicGraphs} An \emph{intrinsic graph over the vertical subgroup $\W=\{x_1=0\}$} in $\He^n$  is a set of the form
\begin{equation}\label{d:ILG} S = \{w \cdot \varphi(w) : w \in \W\}, \end{equation}
where $\varphi \colon \W \to \V =: \{(x_1,0,\ldots,0) : x_1 \in
\R\}$ is an arbitrary function. The graph $S$ determines $\varphi$
uniquely. Further, $S$ is an intrinsic $L$-\textbf{Lipschitz}
graph ($L$-iLG) over $\W$ if it satisfies a \emph{cone condition}
of the form
\begin{displaymath} S \cap (p \cdot \mathcal{C}(\alpha)) = \{p\}, \qquad p \in S, \; 0 < \alpha < L^{-1}. \end{displaymath}
Here $\mathcal{C}(\alpha) := \{q \in \He^n : \|\pi_{\W}(p)\| \leq
\alpha \|\pi_{\V}(p)\|\}$, and $\pi_{\W} \colon \He^n \to \W$ and
$\pi_{\V} \colon \He^n \to \V$ are the \emph{vertical and
horizontal} projections induced by the splitting $\He^n = \W \cdot
\V$. If $S \subset \He^n$ is an ($L$-)iLG, the function $\varphi$
is called an ($L$-)\emph{intrinsic Lipschitz function}.
\end{definition} Now, viewing \eqref{d:ILG}, it is clear that an
iLG $S \subset \He^n$ has a "canonical" parametrisation by the
\emph{graph map} $\Phi \colon \W \to S$, defined by $\Phi(w) := w
\cdot \varphi(w)$. However:
\begin{itemize}
\item $\varphi \colon (\W,d_{\He}) \to (\V,d_{\He})$ is not always
a Lipschitz function, and \item the graph map $\Phi \colon
(\W,d_{\He})\to (S,d_{\He})$ is "almost never" Lipschitz.
\end{itemize}
Regarding the first point, \cite[Example 3.3]{FSS} suggests that
$\varphi(0,x_2,t) = (1 + t^{1/2},0,0)$ is an intrinsic Lipschitz
function in $\mathbb{H}^1$, which is not a Lipschitz function. For
the second point, consider the constant function $\varphi(0,x_2,t)
\equiv (1,0,0)$. Then the graph map
\begin{displaymath} \Phi(0,x_2,t) = (0,x_2,t) \cdot (1,0,0) = (1,x_2,t - \tfrac{x_2}{2}) \end{displaymath}
parametrises the vertical plane $S = \W' =
\{x_1=1\}\subset\mathbb{H}^1$, a prototypical iLG. However, $\Phi$
is not Lipschitz on any open subset of $\W$, because the only
rectifiable curves on $\W,\W'$ are the \emph{horizontal lines}
contained on $\W,\W'$, and $\Phi$ sends the horizontal lines on
$\W$ to non-horizontal lines on $\W'$ (by \textbf{right}
translations).

In spite of these difficulties, the graph map is sometimes useful
for Lipschitz parametrising iLGs: if an iLG $S \subset \He^{1}$
has enough \emph{a priori} regularity, then
$\Phi$ can precomposed with something known as the
\emph{characteristic straightening map} $\Psi \colon \Pi \to \W$
in such a way that $\Phi \circ \Psi \colon \Pi \to S$ is locally
Lipschitz -- or even bilipschitz. The following theorem is due to
Cole and Pauls \cite{MR2247905} from 2006 (the addition of the
letters "bi" is due to Bigolin and Vittone \cite{MR2603594} from
2010):
\begin{thm}[Cole-Pauls, Bigolin-Vittone]\label{CPBV} Every non-characteristic point on a Euclidean $C^{1}$ surface $S \subset \He^1$
has a neighbourhood which is the bilipschitz image of an open
subset of $\Pi$. \end{thm} Both proofs reduce the problem to
intrinsic Lipschitz graphs $S = \Phi(\W)\subset \He^1$, and the
Euclidean $C^{1}$-smoothness of $S$ then translates to properties
of the intrinsic Lipschitz function $\varphi \colon \W \to \V$.
The essential hypothesis is that $\varphi$ is a Euclidean
$C^{1}$-function, although it might suffice that $\varphi$ is Euclidean Lipschitz, viewed as a function $\R^{2} \to \R$.
 The characteristic straightening map only plays a small (and rather implicit) role in this paper, see Lemma \ref{flagimpliesbilip}. So, we refer to the "Outline of proofs" section in \cite{MR3400438}, or the proof of \cite[Theorem 3.1]{MR2603594} for more details.
 In brief, the regularity of $\varphi$ is, in Theorem \ref{CPBV}, required to control the regularity of $\Psi$,
 and if $\varphi$ fails to be Lipschitz $\R^{2} \to \R$, the map $\Psi$ does not appear to be useable for the Lipschitz parametrisation problem.

In fact, Bigolin and Vittone in \cite{MR2603594} show that Theorem \ref{CPBV} can fail without the $C^{1}$-regularity assumption.
For $\tfrac{1}{2} < \alpha < 1$, they consider the intrinsic Lipschitz function
\begin{equation}\label{BVEx} \varphi(0,x_2,t) = \begin{cases} -\tfrac{t^{\alpha}}{1 - \alpha}, & \text{if } t \geq 0, \\
 0, & \text{if } t < 0, \end{cases} \qquad (0,x_2,t) \in \W, \end{equation}
and its intrinsic graph $S = \Phi(\W)$, which fails to be
Euclidean $C^{1}$-regular in any neighbourhood of the line
$\{(0,x_2,0) : x_2 \in \R\}$. They show that no Lipschitz map from
an open subset of $\Pi$ to a neighbourhood of $0 \in \Gamma$ can
have a Lipschitz inverse.

The example \eqref{BVEx} is a good prelude to the results of this
paper. The main novelties will be to
\begin{itemize}
\item[(a)] say something about the LI
rectifiability of iLGs below the critical $C^{1}$-regularity of
$\varphi$ (in particular, our results apply to the example in
\eqref{BVEx} for $\tfrac{1}{2} < \alpha < 1$),
\item[(b)] consider the problem in higher Heisenberg groups, where the technique via the
characteristic straightening map does not seem to be easily available.
\end{itemize}
Here is the first main result:
\begin{thm}\label{mainIntro} Let $\alpha > 0$, and let $S \subset \He^n$ be the intrinsic graph of a globally defined but compactly supported
intrinsic $C^{1,\alpha}$-function. Then $S$ has big pieces of
bilipschitz images of the parabolic plane $(\Pi,d_{\Pi})$ if
$n=1$, or of $(\mathbb{H}^{n-1}\times
\mathbb{R},d_{\mathbb{H}^{n-1}\times \mathbb{R}})$ if $n\geq 2$.
In particular, $S$ is LI rectifiable.
\end{thm}

The following corollary is easier to read:

\begin{thm}\label{mainQualitative} Let $S \subset \He^n$ be a $C^{1,\alpha}_{\He}$-surface. Then $S$ is LI rectifiable. \end{thm}

\begin{remark}
A first version of the present paper, by the third author,
contained Theorems \ref{mainIntro} and \ref{mainQualitative} in $\He^{1}$. After that version appeared on the arXiv, Antonelli
and Le Donne proved in \cite{antonelli2019pauls}, building on \cite{LDY}, that every $C^{\infty}$ hypersurface $S$
in $\mathbb{H}^n$, $n\geq 2$, is rectifiable by bilipschitz images
of subsets of $\mathbb{H}^{n-1}\times\mathbb{R}$.

The result of Antonelli-Le Donne follows from Theorem
\ref{mainQualitative} (or rather, the bilipschitz version stated in
Theorem \ref{surfaces}): the $C^{\infty}$ regularity of $S$,
and a result of Balogh \cite{MR2021034}, imply that
$\mathcal{H}^{2n+1}$ almost every point on $S$ has an
open neighbourhood $U$ such that $S\cap U$ is given as the level
set of a $C^{\infty}$ function $f:U \to \mathbb{R}$ with
nonvanishing \emph{horizontal gradient} $\nabla_{\mathbb{H}}f$.
Using that $f$ is Euclidean $C^{\infty}$ (or even Euclidean $C^{2}$), one concludes that $f$
satisfies condition \eqref{hHolder} in Definition
\ref{C1alpha}, possibly on a slightly smaller open set. Thus,
outside an $\mathcal{H}^{2n+1}$ null set, $S$ is locally a $C_{\mathbb{H}}^{1,\alpha}$ surface, and it
follows from Theorem \ref{surfaces} that $S$ is rectifiable by
bilipschitz images of compact subsets of $\He^{n-1}\times \R$.
\end{remark}

The notions of regularity appearing in the theorems above will be
formally introduced in Section \ref{s:C1alpha}.
 In brief, a $C^{1,\alpha}_{\He}$-surface is an $\He$-regular surface whose \emph{horizontal normal} is $\alpha$-H\"older continuous (in the metric $d_{\He}$). Then, roughly speaking, an intrinsic $C^{1,\alpha}$-function is a function $\W \to \V$ whose intrinsic graph $\Phi(\W)$ is a $C^{1,\alpha}_{\He}$-surface, but this is a little inaccurate; see Definition \ref{C1alphaGraphs} and Remark \ref{r:warning} for more precision.

The next example points out that Theorem \ref{mainQualitative} applies to the function in \eqref{BVEx}:

\begin{ex} The horizontal normal of the intrinsic graph of the function $\varphi$ from \eqref{BVEx} is
\begin{displaymath} \nu_{\He}(\Phi(0,x_2,t)) = \left(\frac{-1}{\sqrt{1 + (c_{\alpha}t^{2\alpha - 1})^{2}}},
\frac{c_{\alpha}t^{2\alpha - 1}}{\sqrt{1 + (c_{\alpha}t^{2\alpha -
1})^{2}}} \right), \qquad x_2 \in \R, \; t > 0, \end{displaymath}
where $c_{\alpha} := \alpha/(1 - \alpha)^{2}$, and
$\nu_{\He}(\Phi(0,x_2,t)) \equiv (-1,0)$ for $x_2 \in \R$, $t \leq
0$. This follows by combining the expression for the
\emph{intrinsic gradient} of $\varphi$ on \cite[p. 166]{MR2603594}
with a known relationship between the horizontal normal of
$\Phi(\W)$ and the intrinsic gradient of $\varphi$, see the
references around \eqref{form32}. With the explicit expression in hand, let us show that the horizontal normal map $\nu_{\He} \colon (\Phi(\W),d_{\He}) \to S^{1}$ is $(2\alpha - 1)/\alpha$-H\"older. First, observe by calculating derivatives that the functions $f,g \colon \R \to \R$ defined by $f(r) := -1/\sqrt{1 + r^{2}}$ and $g(r) := r/\sqrt{1 + r^{2}}$ are (globally) Lipschitz. Consequently, for $(0,x_{2},s),(0,x_{2}',t) \in \W$, we have
\begin{align*} |\nu_{\He}(\Phi(0,x_{2},s)) - \nu_{\He}(\Phi(0,x_{2}',t))| & \lesssim |f(c_{\alpha}s^{2\alpha - 1}) - f(c_{\alpha}t^{2\alpha - 1})| + |g(c_{\alpha}s^{2\alpha - 1}) - g(c_{\alpha}t^{2\alpha - 1})|\\
& \lesssim c_{\alpha} \cdot |s^{2\alpha - 1} - t^{2\alpha - 1}|.  \end{align*}
Next, noting that $(2\alpha -
1)/\alpha \in (0,1]$ for $\tfrac{1}{2} < \alpha \leq 1$, we have
\begin{displaymath} |s^{2\alpha - 1} - t^{2\alpha - 1}| \leq |s^{\alpha} - t^{\alpha}|^{(2\alpha - 1)/\alpha}
\lesssim_{\alpha} d_{\He}(\Phi(0,x_2,s),\Phi(0,x_2',t))^{(2\alpha
- 1)/\alpha}, \end{displaymath} 
as claimed. It follows that $\Phi(\mathbb{W})$ satisfies the assumptions of
Theorem \ref{mainQualitative}, which are made precise in Definition \ref{C1alpha} (see also
Remark \ref{normalRemark}). \end{ex}

The next definition explains the rest of the terminology in Theorem \ref{mainIntro}:

\begin{definition}[BP$G$BI]\label{BPdef} Fix
$n\in \mathbb{N}$ and set $(G,d_{G})=(\Pi,d_{\Pi})$ if $n=1$, and
$(G,d_{G})=(\mathbb{H}^{n-1}\times\mathbb{R},d_{\mathbb{H}^{n-1}\times\mathbb{R}})$
if $n\geq 2$. Let $E \subset \He^n$ be closed and
$(2n+1)$-regular. Then, $E$ has \emph{big pieces of $G$
bilipschitz images} (BP$G$BI) if there exist constants $L \geq 1$
and $\theta > 0$ such that the following holds: for every $p \in
E$ and $0 < r \leq \diam_{\He}(E)$, there exists a compact set $K
\subset B(0,r) \subset G$ and an $L$-bilipschitz map $f \colon K
\to \He^n$ such that
\begin{displaymath} \calH^{2n+1}(f(K) \cap [E \cap B(p,r)]) \geq \theta r^{2n+1}. \end{displaymath}
 \end{definition}
So, up to some technical assumptions,
 Theorem \ref{mainIntro} states that intrinsic $C^{1,\alpha}$-graphs in $\mathbb{H}^n$ are \emph{uniformly rectifiable} by $(G,d_{G})$,
 in the spirit  of David and Semmes \cite{DS1}. We next formulate a stronger result in $\He^{1}$. Informally, the point is that we can relax $C^{1,\alpha}$-regularity to Lipschitz regularity in the "horizontal" directions, but we still need to assume an $\epsilon$ of additional \emph{a priori} regularity in the vertical direction. To motivate the definition, we note that intrinsic $C^{1,\alpha}$-functions are locally Euclidean $(1+\alpha)/2$ H\"older
continuous along vertical lines by \cite[Proposition 4.2]{CFO2}. In $\He^{1}$, it turns out that this property of
\emph{extra vertical H\"older regularity} alone implies the conclusion of Theorem
\ref{mainIntro}.

\begin{thm}\label{mainIntroVertical} Let $S \subset \He^1$
be the intrinsic graph of a globally defined but compactly
supported intrinsic Lipschitz function with extra vertical H\"older
regularity, see Definition \ref{d:ExtraVerticalHolder}.
Then $S$ has big pieces of bilipschitz images of the parabolic
plane $(\Pi,d_{\Pi})$. In particular, $S$ is LI rectifiable.
\end{thm}
The "extra vertical H\"older regularity" is often weaker than intrinsic $C^{1,\alpha}$-regularity: for example, intrinsic Lipschitz functions of the form $\varphi(0,y,t) = \tilde{\varphi}(y)$, with $\tilde{\varphi} \colon \R \to \R$ Lipschitz, are not necessarily intrinsic $C^{1,\alpha}$, but they are very smooth along vertical lines.

Theorem \ref{mainIntroVertical} can be used to give alternative
proofs for the  LI rectifiability of Euclidean $C^1$ hypersurfaces
in $\He^1$ (originally due to Cole-Pauls \cite{MR2247905}) and of
the $n=1$ case of Theorem \ref{mainIntro} (originally due to the
third author). Intrinsic graphs that satisfy the assumptions of
Theorem \ref{mainIntroVertical} are examples of sets on which the
$3$-dimensional \emph{Heisenberg Riesz transform} is
$L^2$-bounded, see
 \cite{fssler2018riesz}. The essential property of such graphs used here is
that they can be well approximated  by \emph{Lipschitz flags}
(Definition \ref{d:FlagSurface}), and that Lipschitz flags can be
bilipschitz parametrised using the characteristic straightening
map of Cole-Pauls and Bigolin-Vittone. No counterparts for these
properties are known in higher dimensions.

\medskip

 We close this section with a few questions:
\begin{questions}Are all   intrinsic Lipschitz graphs in $\He^n$
LI rectifiable? If so, do they have big pieces of  Lipschitz
images of $G$, or
 BP$G$BI? In the converse direction: Are (bi-)Lipschitz images of vertical subgroups in $\mathbb{H}^n$ iLG rectifiable? \end{questions}

\subsection{Bilipschitz maps between big pieces of metric spaces} As mentioned above Theorem \ref{mainIntro},
adapting the techniques of Cole-Pauls and Bigolin-Vittone seems
difficult without something close to $C^{1}$-regularity, or if
$n\geq 2$.
 Instead, Theorems \ref{mainIntro} and \ref{mainIntroVertical} will follow from an application of a general result concerning metric spaces, Theorem \ref{main} below. We now formulate the abstract hypotheses of that theorem.

 Let $(G,d_{G})$ and $(M,d_{M})$ be metric spaces. Assume the
following "local correspondence" between $G$ and $M$,
 for constants $\alpha > 0$, $L\geq 1$, $A \geq 1$, and for some $x_{0} \in G$ and $p_{0} \in M$ fixed.
 For every $x \in B_{G}(x_{0},1)$, $p \in B_{M}(p_{0},1)$ and
  $n \in \N \cup \{0\}$ there exists a map $i^{n}_{x \to p} \colon G \to M$ with $i^{n}_{x \to p}(x) = p$ such that
\begin{equation}\label{ISO}
L^{-1}d_G(y,z)-A2^{-n(1+\alpha)}\leq d_M(i^{n}_{x \to
p}(y),i^{n}_{x \to p}(z))\leq L d_G(y,z)+ A 2^{-n(1+\alpha)},\;
y,z\in B(x,2^{-n}).
 \end{equation} Moreover, if
$n \geq 0$, $x,y \in B_{G}(x_{0},1)$ and $p,q \in B_{M}(p_{0},1)$
with $d_{G}(x,y) \leq 2^{-n}$ and $i^{n}_{x \to p}(y) = q$, then
\begin{equation}\label{comp} d_{M}(i^{n}_{x \to p}(z),i^{n + 1}_{y \to q}(z)) \leq A2^{-n(1 + \alpha)} \qquad z \in B(x,2^{-n}). \end{equation}

These assumptions are reminiscent of \cite[(1.6)-(1.9)]{MR2907827}. See also \cite[Appendix 1]{CheegerColding} for related results. The assumption \eqref{ISO} postulates that the maps $i_{x \to p}^{n}$ are bilipschitz continuous at the scale $2^{-n}$, up to an error which is much smaller than $2^{-n}$. The assumption \eqref{comp} is "compatibility condition": it states that the maps $i_{x \to p}^{n}$ and $i_{y \to q}^{n + 1}$ nearly coincide at scale $2^{-n}$, again up to an error which is much smaller than $2^{-n}$. The \emph{a priori} condition "$i_{x \to p}^{n}(y) = q$" above \eqref{comp} is a technically convenient way of assuming that $p$ and $q$ are close to each other on $M$: indeed it follows from the hypotheses, including \eqref{ISO}, that
\begin{displaymath} d_{M}(p,q) = d_{M}(i_{x \to p}^{n}(x),i_{x \to p}^{n}(y)) \leq Ld_{G}(x,y) + A2^{-n(1 + \alpha)} \lesssim_{A,L} 2^{-n}. \end{displaymath}

The assumptions \eqref{ISO}-\eqref{comp} are designed to enable
the construction of bilipschitz maps from subsets of $G$ to $M$,
at least if $G$ is complete and doubling:
\begin{thm}\label{main}Assume that $(G,d_{G})$ is a complete metric space with $\diam_{G}(G) \geq 1$
and equipped with a nontrivial doubling measure $\mu$, and
$(M,d_{M})$ is complete. Assume that the conditions
\eqref{ISO}-\eqref{comp} hold for some $\alpha > 0$, $L\geq 1$,
and $A \geq 1$. Then, there exists a constant $\delta > 0$, a
compact set $K \subset B(x_{0},1)$ with $\mu(K) \geq \delta
\mu(B(x_{0},1))$ and a $2L$-bilipschitz embedding $F \colon K \to
M$ with $F(K) \subset B(p_{0},1)$. The constant $\delta > 0$ only
depends on the doubling constant of $(G,d_G,\mu)$, and the
constants $\alpha$, $L$, and $A$ in \eqref{ISO}-\eqref{comp}.
\end{thm}

The plan of the paper is to prove Theorem \ref{main} in Section
\ref{s:construction} and Appendix \ref{FCSPropProof}, and then
apply it to prove Theorem \ref{mainIntro} for $\He^n$, $n>1$, in
Section \ref{s:C1alpha}. In Section \ref{s:LipFlag} we prove
Theorem \ref{mainIntroVertical}, which yields as a corollary the
case $n=1$ of Theorem \ref{mainIntro}. Theorem
\ref{mainQualitative} is "morally" a direct corollary of Theorem
\ref{mainIntro}: by the implicit function theorem of Franchi,
Serapioni, and Serra Cassano \cite[Theorem 6.5]{FSSC}, a
$C^{1,\alpha}_{\He}$-surface is locally parametrisable by an
intrinsic $C^{1,\alpha}$-function over some vertical subgroup $\W
\subset \He^n$. However, the parametrisation may only be defined
on a strict subset of $\W$, and fail to literally satisfy the
assumptions of Theorem \ref{mainIntro}. As far as we know, there
is no extension theorem for intrinsic $C^{1,\alpha}$-functions
available. To bypass the issue, Appendix \ref{extensionAppendix}
contains a proposition saying that every point on a
$C_{\He}^{1,\alpha}$-surface has a neighbourhood which is
contained on the intrinsic graph of a globally defined, compactly
supported intrinsic $C^{1,\alpha/3}$-function. With this
proposition in hand, Theorem \ref{mainQualitative} is indeed a
corollary of Theorem \ref{mainIntro}.

\subsection{Notations} For $A,B > 0$, we write $A\lesssim B$ if there is a constant $C >
0$ such that $A \leq CB$. If we want to specify that the value of
$C$ is allowed to depend on an auxiliary parameter $h$, we will
write $A\lesssim_h B$.

\subsection{Acknowledgements}
K.F. and T.O. would like to thank Enrico Le Donne and S\'everine
Rigot for many fruitful discussions on the topic of the paper. We
are also grateful to Davide Vittone for tips on proving
Proposition \ref{appProp}, which was needed to reduce Theorem
\ref{mainQualitative} to Theorem \ref{mainIntro}, and to Damian D\k{a}browski for pointing out the reference
\cite{2017arXiv171103088A}. D.D.D. would like to thank Enrico Le
Donne and Raul Serapioni for important suggestions on the subject. Finally, we thank the anonymous referees for a careful reading of the manuscript, and for making many useful suggestions.

\section{Proof of the main theorem for metric spaces}\label{s:construction}

The proof of Theorem \ref{main} is inspired by the recent work of Le Donne and Young \cite{LDY} on the Carnot rectifiability of sub-Riemannian manifolds.

Before starting the proof of Theorem \ref{main} in earnest, we
need to introduce some terminology. Assume for a moment that
$n_{0} \geq 0$, and there exist families $\{\calD_{n}\}_{n \geq
n_{0}}$ of subsets of $G$, known as \emph{cubes}, with the
following properties:
\begin{itemize}
\item[(i)\phantomsection \label{i}] Each $\calD_{n}$ consists of a finite number of disjoint non-empty compact sets, and in particular $\card \calD_{n_{0}} = 1$.
\item[(ii) \phantomsection \label{ii}] For $n > n_{0}$, each cube $Q \in \calD_{n}$ is contained in a unique cube $\hat{Q} \in \calD_{n - 1}$, called the \emph{parent} of $Q$. For $Q \in \calD_{n - 1}$ fixed, write $\mathrm{ch}(Q) := \{Q \in \calD_{n} : \hat{Q} = Q\}$.
\item[(iii) \phantomsection \label{iii}] $\diam_{G}(Q) < 2^{-n}$ for all $Q \in \calD_{n}$.
\item[(iv) \phantomsection \label{iv}] There are constants $\epsilon, \tau > 0$ such that if $n \geq 0$ and $Q_{1},Q_{2} \in \calD_{n}$ are distinct, then $d_{G}(Q_{1},Q_{2}) \geq \tau 2^{-(1 + \epsilon)n}$.
\end{itemize}
An \emph{$(\epsilon,n_{0},\tau)$-Cantor set} is any set of the form
\begin{displaymath} K := \bigcap_{n \geq n_{0}} K_{n} := \bigcap_{n \geq n_{0}} \bigcup_{Q \in \calD_{n}} Q, \end{displaymath}
where the families $\calD_{n}$, $n_{0} \geq 0$, satisfy properties (i)-(iv).
\begin{definition}\label{FCS} A metric measure space $(X,d,\mu)$ \emph{admits fat Cantor sets} if for all $\epsilon > 0$ and $n_{0} \geq 0$, there exist constants $\delta = \delta(n_{0}) > 0$, and $\tau = \tau(\epsilon) > 0$ such that the following holds. For all $x \in X$, there exists an $(\epsilon,n_{0},\tau)$-Cantor set $K \subset B(x,1)$ with $\mu(K) \geq \delta \mu(B(x,1))$. \end{definition}

We will only use the assumption that $(G,d_{G},\mu)$ is doubling and complete to ensure that it admits fat Cantor sets.
\begin{proposition}\label{FCSProp} Every doubling and complete metric measure space $(X,d,\mu)$ of diameter $\geq 1$ admits fat Cantor sets. In other words, for every $\epsilon > 0$ and $n_{0} \geq 0$, the constants $\delta(n_{0}) > 0$ and $\tau(\epsilon) > 0$ can be found as in Definition \ref{FCS}. They are also allowed to depend on the doubling constant of $(X,d,\mu)$.
\end{proposition}

The proof of the proposition above is essentially contained in
\cite[Section 4.1]{LDY}, but since the statement is not explicitly
given in \cite{LDY}, we repeat the details in Appendix
\ref{FCSPropProof}. What follows next is a proof of Theorem
\ref{main}, assuming Proposition \ref{FCSProp}. Fix $x_{0} \in G$
and $p_{0} \in M$. Let $\epsilon := \alpha/2$ (where $\alpha > 0$
is the parameter appearing in \eqref{ISO}-\eqref{comp}), let
$n_{0} \geq 0$ be a large integer to be determined later, and let
$K \subset B(x_{0},1)$ be an $(\epsilon,n_{0},\tau)$-Cantor set
associated to families of cubes $\{\calD_{n}\}_{n \geq n_{0}}$, as
in \nref{i}-\nref{iv}. For each $Q \in \calD_{n}$, $n \geq n_{0}$,
pick a \emph{centre} $c_{Q} \in Q$. Set $\calD_{n_{0}-1} := \{G\}$
and $K_{n_{0} - 1} := G$.

The map $F \colon K \to M$ will be defined as the limit of certain intermediate maps $F_{n} \colon K_{n} \to M$ for $n \geq n_{0} - 1$. Set $F_{n_{0}-1} \equiv p_{0}$. To proceed, assume that $n \geq n_{0}$, and the map $F_{n - 1} \colon K_{n - 1} \to M$ has already been defined. Then, fix $Q_{n - 1} \in \calD_{n - 1}$, and set
\begin{displaymath} F_{n}|_{Q} := i^{n}_{c_{Q} \to F_{n - 1}(c_{Q})}|_{Q}, \qquad Q \in \mathrm{ch}(Q_{n - 1}) \subset \calD_{n}. \end{displaymath}
The next lemma shows, in particular, that if $x \in K$, then the sequence $(F_{n}(x))_{n \in \N}$ is Cauchy in $(M,d_{M})$. Hence $F(x) := \lim_{n \to \infty} F_{n}(x)$ exists by the completeness of $(M,d_{M})$.

\begin{lemma} If $n_{0} \geq 1$ is large enough (depending on $A$ and the constant $L$ in \eqref{ISO}), the following holds for all $w \in K$ and $n \geq n_{0}$:
\begin{equation}\label{form15} d_{M}(F_{n}(w),F_{n + 1}(w)) \leq  A2^{-n(1 + \alpha)}. \end{equation}
\end{lemma}
\begin{proof} For $w \in K$, let $c_{n}(w)$ be the centre of $Q_{n}(w)$, where $Q_{n}(w)$ is the unique cube in $\calD_{n}$ containing $w$. One can infer \eqref{form15} from the compatibility condition \eqref{comp} in the following way:
\begin{align*} d_{M}(F_{n}(w),F_{n + 1}(w)) = d_{M}(i^{n}_{c_{n}(w) \to F_{n - 1}(c_{n}(w))}(w),i^{n + 1}_{c_{n + 1}(w) \to F_{n}(c_{n + 1}(w))}(w)) \leq  A2^{-n(\alpha + 1)}. \end{align*}
To check that the assumptions of \eqref{comp} are really in force, use the notational substitutions
\begin{displaymath} q :=  F_{n}(c_{n + 1}(w)), \: p := F_{n - 1}(c_{n}(w)), \: x := c_{n}(w), y := c_{n + 1}(w), \quad \text{and} \quad z := w.\end{displaymath}
Then, note that $d_{G}(x,y) \leq \diam_{G}(Q_{n}(x)) < 2^{-n}$ by \nref{iii}, $z \in B(x,2^{-n})$ again by \nref{iii}, and
\begin{equation}\label{form38} q = F_{n}(c_{n + 1}(w)) = i_{c_{n}(w) \to F_{n - 1}(c_{n}(w))}^{n}(c_{n + 1}(w)) = i^{n}_{x \to p}(y). \end{equation}
by definition. This shows that the assumptions of \eqref{comp} are valid, except for one small issue: is it clear that $p,q \in B(p_{0},1)$? For $n = n_{0}$, simply $p = F_{n_{0} - 1}(c_{n_{0}}(w)) = p_{0}$. Also, using \eqref{form38} and \eqref{ISO},
we find that
\begin{displaymath} d_{M}(q,p) = d_{M}(i^{n}_{x \to p}(y),i^{n}_{x \to p}(x)) \leq L d_{G}(x,y) + A2^{-n(1 + \alpha)}
\lesssim_{L} A2^{-n}. \end{displaymath}
Recalling the definitions of $p,q$, and using the estimate above
repeatedly shows that $\max\{d_{M}(p,p_{0}),d_{M}(q,p_{0})\}
\lesssim_{L} A2^{-n_{0}}$ for all $n \geq n_{0}$. In particular,
$p,q \in B(p_{0},1)$ if $n_{0} \geq 1$ is large enough, depending
on $A$ and the constant $L$ in \eqref{ISO}. The proof of the lemma
is complete.
\end{proof}

As an immediate corollary, one deduces the useful estimate
\begin{equation}\label{form4} d_{M}(F_{n}(x),F(x)) \lesssim A2^{-n(1 + \alpha)}, \qquad x \in K, \: n \geq n_{0}. \end{equation}
It remains to prove that $F$ is $2L$-bilipschitz on $K$ if the
index $n_{0} \in \N$ was chosen large enough, depending on the
parameters
 $\alpha > 0$, $L\geq 1$ and $A \geq 1$. Fix $x,y \in K$ arbitrary with $x \neq y$, and let $Q \in \calD_{n}$, $n \geq n_{0}$, be the smallest cube with $x,y \in Q$ (thus $x,y$ lie in distinct cubes in $\calD_{n + 1}$). Write $c := c_{Q} \in Q$. Then, using properties \nref{iii}-\nref{iv} of the cubes $\calD_{n}$,
\begin{equation}\label{form1} \max\{d_{G}(x,c),d_{G}(y,c)\} \leq 2^{-n} \quad \text{and} \quad \tau 2^{-(n + 1)(1 + \epsilon)} \leq d_{G}(x,y) \leq 2^{-n}. \end{equation}
Also, by definition,
\begin{displaymath} F_{n}(x) = i^{n}_{c \to F_{n - 1}(c)}(x) \quad \text{and} \quad F_{n}(y) = i^{n}_{c \to F_{n - 1}(c)}(y), \end{displaymath}
since $x,y \in Q$. We deduce from \eqref{ISO} (using \eqref{form1}) that
\begin{equation}\label{form7} d_{M}(F_{n}(x),F_{n}(y)) = d_{M}(i^{n}_{c \to F_{n - 1}(c)}(x),i^{n}_{c \to F_{n - 1}(c)}(y)) \geq
L^{-1} d_{G}(x,y) - A2^{-n(1 + \alpha)}. \end{equation} We also
have the upper bound analogous to \eqref{form7},
\begin{displaymath} d_{M}(F_{n}(x),F_{n}(y)) \leq
L d_{G}(x,y) + A2^{-n(1 + \alpha)}. \end{displaymath} Recalling
that $\epsilon = \alpha/2$ and $\tau=\tau(\varepsilon)$, the error
term $A2^{-n(1 + \alpha)}$ is smaller than $$L^{-1}\tau 2^{-(n +
1)(1 + \epsilon)}/4 \leq L^{-1}d_{G}(x,y)/4$$ for $n \geq n_{0}$,
if $n_{0} \in \N$ was chosen large enough, depending on $A$, $L$
and $\alpha$. Thus,
\begin{equation}\label{form2}L^{-1} \frac{3 d_{G}(x,y)}{4} \leq d_{M}(F_{n}(x),F_{n}(y)) \leq L\frac{5 d_{G}(x,y)}{4}. \end{equation}
Finally, if $n_{0} \geq 0$ is large enough, depending again on
$\alpha$, $L$ and $A$, one sees from \eqref{form4}-\eqref{form1}
that
\begin{equation}\label{form3} \max\{d_{M}(F(x),F_{n}(x)),d_{M}(F(y),F_{n}(y))\} \leq L^{-1}\frac{d_{G}(x,y)}{8}. \end{equation}
It then follows by combining \eqref{form2}, \eqref{form3}, and the triangle inequality that
\begin{displaymath} L^{-1}\frac{d_{G}(x,y)}{2} \leq d_{M}(F(x),F(y)) \leq L 2d_{G}(x,y), \end{displaymath}
as desired. The proof of Theorem \ref{main} is complete, except for the claim $F(K) \subset B(p_{0},1)$. Pick $x \in K \subset Q_{n_{0}} \subset B(c_{Q_{n_{0}}},2^{-n_{0}})$, and write, using \eqref{ISO},
\begin{displaymath} d_{M}(p_{0},F_{n_{0}}(x)) = d_{M}(i^{n_{0}}_{c_{Q_{n_{0}}} \to p_{0}}(c_{Q_{n_{0}}}),i^{n_{0}}_{c_{Q_{n_{0}}} \to p_{0}}(x)) \leq
Ld_{G}(c_{Q_{n_{0}}},x) + A2^{-n_{0}} \leq (L + A)2^{-n_{0}}.
\end{displaymath} Further, it follows from \eqref{form4} that
$d_{M}(F(x),F_{n_{0}}(x)) \lesssim A2^{-n_{0}}$. So, if $n_{0}
\geq 2$ was chosen large enough, depending on $A$ and $L$, one has
$F(x) \in B(p_{0},\tfrac{1}{2})$. Finally, since $\diam_{G}(K)
\leq 2^{-n_{0}} \leq \tfrac{1}{4L}$ for $n_0$ large enough
depending on $L$, and $F$ is $2L$-Lipschitz, one has $F(K) \subset
B(F(x),\tfrac{1}{2}) \subset B(p_{0},1)$. The proof of Theorem
\ref{main} is complete.


\section{Graphs and surfaces with H\"older-continuous horizontal normals}\label{s:C1alpha}

Throughout this section, we use coordinates
$(x_1,\ldots,x_{2n},t)$ on
 $\mathbb{H}^n$ as defined at the beginning of Section
 \ref{s:introd}. In these coordinates, a frame for the
 left invariant vector fields on $\mathbb{H}^n$ is given by
 \begin{equation}\label{eq:vfd}
 X_i = \partial_{x_i}-\tfrac{x_{n+i}}{2}\partial_t,\quad X_{n+i} = \partial_{x_{n+i}}+\tfrac{x_{i}}{2}\partial_t,\text{ for
$i=1,\ldots,n$,}\quad
 \text{and}\quad
 T=\partial_t,
 \end{equation} which yields the nontrivial commutator relations
\begin{displaymath}
[X_i,X_{n+i}]=T,\quad i=1,\ldots,n.
\end{displaymath}
The \emph{horizontal gradient} of a function $f:U \subset
\mathbb{H}^n \to \mathbb{R}$ is
\begin{displaymath}
\nabla_{\He}f=(X_1 f,\ldots,X_{2n} f).
\end{displaymath}
A \emph{vertical subgroup $\mathbb{W}$ of codimension $1$} in
$\mathbb{H}^n$ is, in the above coordinate system, a
$2n$-dimensional subspace of $\mathbb{R}^{2n+1}$ containing the
$t$-axis. Given such a subgroup $\mathbb{W}$, we denote by
$\mathbb{V}$ its Euclidean orthogonal complement. We recall that
$\pi_{\W} \colon \He^n \to \W$ is the \emph{vertical projection}
to $\W$, and $\pi_{\V} \colon \He^n \to \V$ is the
\emph{horizontal projection} to $\V$, induced by the splitting
$\He = \W \cdot \V$, see \cite[Proposition 2.2]{FSS}. In
particular,
\begin{displaymath}
p= \pi_{\W}(p) \cdot \pi_{\V}(p),\quad\text{for all
}p\in\mathbb{H}^n.
\end{displaymath}

\subsection{Definitions and preliminaries}\label{s:DefPrelim} A \emph{$C^{1,\alpha}_{\He}$-surface} is locally a non-critical level set of a $C^{1,\alpha}_{\He}$-function $f \colon \He^n \to \R$.
Here is the precise definition:

\begin{definition}[$C^{1,\alpha}_{\He}$-surfaces]\label{C1alpha} Let $S \subset \He^n$ be an $\He$-regular surface in the sense of Franchi, Serapioni, and Serra Cassano \cite[Definition 6.1]{FSSC}: for every $p \in S$, there exists an open ball $B(p,r)$ and a function $f \in C^{1}_{\He}(B(p,r))$ such that $\nabla_{\He}f(p) \neq 0$, and
\begin{equation}\label{eq:Slevel} S \cap B(p,r) = \{q \in B(p,r) : f(q) = 0\}. \end{equation}
For $0 < \alpha \leq 1$, the set $S$ is called a
\emph{$C^{1,\alpha}_{\He}$-surface} if one can choose $f$ so that
there exists a constant $H = H_{p} \geq 1$ such that
\begin{equation}\label{hHolder} |\nabla_{\He}f(q_{1}) - \nabla_{\He}f(q_{2})|
 \leq Hd_{\He}(q_{1},q_{2})^{\alpha}, \qquad q_{1},q_{2} \in S \cap B(p,r).  \end{equation}
\end{definition}

\begin{remark}\label{normalRemark} If $S \cap B(p,r) = B(p,r) \cap \{f = 0\}$, as above, then \cite[Theorem 6.5]{FSSC} states that, after making $r > 0$ possibly a little smaller, the inward-pointing horizontal normal of $E = \{f < 0\}$ is given by the expression
\begin{displaymath} \nu_{E}(q) = -\frac{\nabla_{\He}f(q)}{|\nabla_{\He}f(q)|}, \qquad q\in S \cap B(p,r). \end{displaymath}
Clearly, if $f$ satisfies \eqref{hHolder}, then this  choice of
horizontal normal is (locally) $\alpha$-H\"older continuous as a
map $(S,d_{\He}) \to S^{2n-1}$.

Conversely, if $S\subset \mathbb{H}^n$ is an $\mathbb{H}$-regular
surface with $\alpha$-H\"older continuous $\nu_E$, if $p$ is an
arbitrary point in $S$, and $r> 0$ is small enough, we claim that
there exists $f\in C_{\mathbb{H}}^1(B(p,r))$ so that
\eqref{eq:Slevel} and \eqref{hHolder} hold. Not every function $f$
which satisfies \eqref{eq:Slevel} necessarily fulfills
\eqref{hHolder}, cf.\ the related Remark \ref{r:warning} below. In
order to find $f$ which satisfies simultaneously the two
conditions, it is convenient to write $S$ locally as an
\emph{intrinsic graph}. First, by assumption there exists $i\in
\{1,\ldots,2n\}$ so that the component $\nu_{E}^i$ does not vanish
on $S \cap B(p,r)$ for small enough $r>0$. Without loss of
generality, we may assume that $i=1$ and $\nu_{E}^i< 0$ on $S \cap
B(p,r)$. Then, the implicit function theorem of Franchi,
Serapioni, and Serra Cassano \cite[Theorem 6.5]{FSSC}, combined
with \cite[Theorem 1.2]{MR2223801}, implies that for small enough
$r>0$, the set $S \cap B(p,r)$ can be written as {intrinsic graph}
in $X_1$-direction of a function $\varphi$ whose \emph{intrinsic
gradient} $\nabla^{\varphi}\varphi$ (in the sense of Definition
\ref{d:IntrGrad}) exists and is a continuous
$\mathbb{R}^{2n-1}$-valued function. This allows us to express
locally the horizontal normal  $\nu_{E}$ in a convenient form,
cf.\ \eqref{form32}. Using this expression, it is easy to see that
$\alpha$-H\"older continuity of $\nu_{E}$ implies
$\alpha$-H\"older continuity of $[\nabla^{\varphi}\varphi] \circ
\pi_{\mathbb{W}}|_{S\cap B(p,r)}$ for $r$ small enough, cf.\
\eqref{eq:HolForm}. Then
\begin{displaymath}
f(x_1,\ldots,x_{2n},t):=
x_1-\varphi\left(\pi_{\mathbb{W}}(x_1,\ldots,x_{2n},t)\right)=
x_1-\varphi\left(0,x_2,\ldots,x_{2n},t+\tfrac{1}{2}x_1x_{n+1}\right)
\end{displaymath}
has the properties \eqref{eq:Slevel} and \eqref{hHolder}. To see
this, use \cite[Proposition 2.22]{MR2223801} for the expression of
$\nabla_{\He}f$ in terms of $\nabla^{\varphi}\varphi$, and note
that $X_1 f =1$.
\end{remark}

A case of particular interest in this paper are intrinsic graphs
$S \subset \He^n$ that also happen to be
$C^{1,\alpha}_{\He}$-surfaces. Specifically, consider the
following definition:
\begin{definition}[Intrinsic $C^{1,\alpha}$-functions and intrinsic $C^{1,\alpha}$-graphs]\label{C1alphaGraphs} Let
\begin{displaymath}
\W := \{(0,x_2,\ldots,x_{2n},t) : (x_2,\ldots,x_{2n},t)  \in
\R^{2n}\}
\end{displaymath}
 and $\V := \{ (x_1,0\ldots,0): x_1 \in \R\} \cong \R$. Let $U \subset \W$ be open, and let $\varphi \colon U \to \V$ be continuous. Write $\Phi(w) := w \cdot \varphi(w)$ for the \emph{graph map} of $\varphi$. We say that $\varphi$ is an \emph{intrinsic $C^{1,\alpha}$-function on $U$}, denoted $\varphi \in C^{1,\alpha}_{\He}(U)$ if
\begin{itemize}
\item[(i)] $S := \Phi(U) \subset \He^n$ is a
$C^{1,\alpha}_{\He}$-surface in the sense of Definition
\ref{C1alpha}, and \item[(ii)] the horizontal normal $\nu_{\He} =
(\nu_{\He}^{1},\ldots,\nu_{\He}^{2n}) \colon S \to \R^{2n}$ of the
subgraph
\begin{displaymath}
\{w \cdot v : v < \varphi(w)\}
\end{displaymath}
satisfies $\nu_{\He}^{1}(p) < 0$ for all $p \in S$.
\end{itemize}
The intrinsic graph $S$ of any intrinsic $C^{1,\alpha}$-function
is an \emph{intrinsic $C^{1,\alpha}$-graph}.
\end{definition}

\begin{remark}\label{r:warning} The conditions (i)-(ii) are $C^{1,\alpha}$-versions of the conditions appearing in \cite[Theorem 1.2(i)]{MR2223801}. Condition (ii) is not as odd as it looks: a similar hypothesis would also be required to characterise $C^{1,1}$-functions $f \colon \R \to \R$ via the properties of $\Gamma(f) := \{(x,f(x)) : x \in \R\}$. To see this, consider $f \colon \R \to \R$ given by $f(x) = \mathrm{sgn}(x)\sqrt{|x|}$. Then $\Gamma(f)$ is a $C^{1,1}$-surface as a subset of $\R^{2}$, because $\Gamma(f)$ can also be written as $\Gamma(f) = \{(\sgn(y)y^{2},y) : y \in \R\}$, where $y \mapsto \sgn(y)y^{2} \in C^{1,1}(\R)$. Nonetheless, $f \notin C^{1,1}(\R)$. \end{remark}

A previous notion of \emph{intrinsic $C^{1,\alpha}$-functions} and
\emph{graphs} in $\mathbb{H}^1$ already exists, see
\cite[Definition 2.16]{CFO2} or \eqref{CFOdef} below, and it looks
different than Definition \ref{C1alphaGraphs}. The connection
needs to be clarified immediately, because a result concerning
intrinsic $C^{1,\alpha}$-functions in the sense of \cite{CFO2},
namely \cite[Proposition 4.2]{CFO2}, will be used also in this
paper. Definition 2.16 in \cite{CFO2} was stated for
$\mathbb{H}^1$, but it can be extended in an obvious way to higher
dimensional Heisenberg groups, and this version appears in
Proposition \ref{equivProp}. The precise formulation requires the
notion of \emph{intrinsic differentiability}.

\begin{definition}\label{d:IntrDiff} A function $\varphi:\W \to
\V$ is \emph{intrinsically differentiable} at the point $w_0\in
\W$ if there exists a map $L:\W \to \V$ whose intrinsic graph
$\{w\cdot L(w):\;w\in \W\}$ is a vertical subgroup and which
satisfies
\begin{displaymath}
\lim_{\|w\|\to 0} \frac{\|L(w)^{-1}\cdot
\varphi^{(p^{-1})}(w)\|}{\|w\|}=0,\quad w\in \W.
\end{displaymath}
Here $p=\Phi(w_0)$, and $\varphi^{(p^{-1})} \colon \W \to \V$ is
the unique function $\W \to \V$ whose intrinsic graph is $p^{-1}
\cdot \Phi(\mathbb{W})$ (see \eqref{form34} for a formula).
\end{definition}
For equivalent definitions, see \cite[Proposition
4.76]{MR3587666}. If $\varphi$ is intrinsically differentiable at
$w_0$, then there is a unique map $L:\W \to \V$ with the
properties stated in Definition \ref{d:IntrDiff}. This map is
called the \emph{intrinsic differential} of $\varphi$ at $w_0$ and
it is denoted by $d^{\varphi}\varphi(w_0)$. Moreover, there is a
unique vector $\nabla^{\varphi}\varphi(w_0)\in\mathbb{R}^{2n-1}$,
such that
\begin{displaymath}
d^{\varphi}\varphi(w_0)(w)=\langle\nabla^{\varphi}\varphi(w_0),
\pi(w) \rangle,\quad w\in \W,
\end{displaymath}
where $\langle \cdot,\cdot\rangle$ denotes the scalar product in
$\mathbb{R}^{2n-1}$, and
\begin{displaymath}
\pi(0,x_2,\ldots,x_{2n},t)=(x_2,\ldots,x_{2n}).
\end{displaymath}

\begin{definition}\label{d:IntrGrad}
If $\varphi:\W\to \V$ is intrinsically
differentiable at $w_0 \in \W$, its \emph{intrinsic gradient at
$w_0$} is the vector $\nabla^{\varphi}\varphi(w_0)$. The
components of $\nabla^{\varphi}\varphi(w_0)$ are denotes as
follows:
\begin{displaymath}
\nabla^{\varphi}\varphi(w_0)=(D_2^{\varphi}\varphi(w_0),\ldots,D_{2n}^{\varphi}\varphi(w_0)).
\end{displaymath}
\end{definition}
More information about the functions $D_i^{\varphi}\varphi$,
$i=2,\ldots,2n$, will only be required once we arrive at the proof
of Proposition \ref{p:Approx}, so we postpone the detailed
discussion. At this point we just mention that the component
$D_{n+1}^{\varphi}\varphi$ will play a distinguished role as we
consider intrinsic graphs in the $X_1$ direction, and
$[X_1,X_i]=0$ for all $i=1,\ldots,2n$, except for $i=n+1$.

\begin{proposition}\label{equivProp} Let $0<\alpha \leq 1$ and $\varphi \colon \W \to \V$ be a compactly supported function between the subgroups
$\W$ and $\V$ in $\He^n$. Then $\varphi \in
C^{1,\alpha}_{\He}(\W)$ in the sense of Definition
\ref{C1alphaGraphs} if and only if $\varphi$ is intrinsically
differentiable, and the intrinsic gradient
$\nabla^{\varphi}\varphi$ satisfies
\begin{equation}\label{CFOdef} |\nabla^{\varphi^{(p^{-1})}} \varphi^{(p^{-1})}(w) - \nabla^{\varphi^{(p^{-1})}}\varphi^{(p^{-1})}(0)| \leq H\|w\|^{\alpha}, \qquad w \in \W, \: p \in \Phi(\W) \end{equation}
for a constant $H \geq 1$.
\end{proposition}

\begin{remark}\label{r:GradProp} Since \cite[Definition 2.16]{CFO2} imposes "global" H\"older continuity for $\nabla^{\varphi}\varphi$, whereas the assumptions in Definition \ref{C1alphaGraphs} are of local nature, the notions cannot be equivalent without some \emph{a priori} assumptions
-- as the compact support of $\varphi$ in Proposition
\ref{equivProp}. We also remark that condition \eqref{CFOdef}
implies that $\nabla^{\varphi}\varphi$ is continuous, as can be
deduced for instance from  formula \eqref{eq:translatedGrad}
below.
\end{remark}

\begin{proof}[Proof of Proposition \ref{equivProp}] Assume that $\varphi \in C^{1,\alpha}_{\He}(\W)$ in the sense of Definition \ref{C1alphaGraphs}.
Then, \cite[Theorem 1.2]{MR2223801} states \textrm{in particular}
that $\varphi$ is intrinsically differentiable, the intrinsic
gradient $\nabla^{\varphi}\varphi(w)$ exists for all $w \in \W$, and $w
\mapsto \nabla^{\varphi}\varphi(w)$ is continuous (see also
\cite[Theorem 4.95]{MR3587666}).

Additionally, \cite[Theorem 1.2]{MR2223801} promises that the
horizontal normal $\nu_{\He}$ in Definition
\ref{C1alphaGraphs}(ii) has the representation
\begin{equation}\label{form32} \nu_{\He}(\Phi(w)) = \left(\frac{-1}{\sqrt{1 + |\nabla^{\varphi}\varphi(w)|^{2}}},\frac{\nabla^{\varphi}\varphi(w)}{\sqrt{1 + |\nabla^{\varphi}\varphi(w)|^{2}}} \right), \qquad w \in \W. \end{equation}
Let us now argue that $\nu_{\He}$, above, is (globally)
$\alpha$-H\"older continuous on $S = \Phi(\W)$. Recall that by the
assumption that $S$ is a $C^{1,\alpha}_{\He}$-surface, and Remark
\ref{normalRemark}, for every $p \in S$ there exists \textbf{some}
$\alpha$-H\"older continuous choice of a horizontal normal
$\nu_{\He}^{p}$, defined in a neighbourhood $S \cap U$ of $p$
(note that there are two horizontal normals at every point in
$S$). However, it is easy to see that $\nu_{\He}^{p}$ must
coincide on $S \cap U$ with either $\nu_{\He}$ or $-\nu_{\He}$
whenever $S \cap U$ is connected (that is, the sign depends only
on $U$).  But since $S$ is locally connected, and since
$\nu_{\He}(p) \equiv (-1,0,\ldots,0)$ outside the compact set
$\Phi(\spt \varphi) \subset S$, one infers that the horizontal
normal $\nu_{\He}$ in \eqref{form32} is $\alpha$-H\"older
continuous on $S$.

Another remark: using again that $\spt \varphi$ is compact,
$\nabla^{\varphi}\varphi$ is continuous and supported in $\spt
\varphi$, we see that $L :=
\|\nabla^{\varphi}\varphi\|_{L^{\infty}(\W)} < \infty$. From
\eqref{form32} and the $\alpha$-H\"older continuity of $\nu_{\He}$
on $S=\Phi(\mathbb{W})$, it can easily be deduced that
\begin{displaymath}
\ell:(S,d_{\mathbb{H}})\to \mathbb{R},\quad p\mapsto\ell(p):=
\sqrt{1+|\nabla^{\varphi}\varphi(\pi_{\mathbb{W}}(p))|^2}
\end{displaymath}
is $\alpha$-H\"older continuous with H\"older constant bounded in
terms of $L$ and the $\alpha$-H\"older constant of
$\nu_{\mathbb{H}}$. Then
\begin{align}\label{eq:HolForm}
\left|\ell(p)\nu_{\mathbb{H}}(p)-\ell(p')\nu_{\mathbb{H}}(p')
\right|&\leq \ell(p)
\left|\nu_{\mathbb{H}}(p)-\nu_{\mathbb{H}}(p') \right|+
\left|{\ell(p)}-{\ell(p')}\right|\lesssim_L d(p,p')^{\alpha}
\end{align}
for $p,p'\in S$. Now we are quite well equipped to check that
$\varphi$ satisfies \eqref{CFOdef}. To this end, fix $w \in \W$ and $p \in \Phi(\W)$. Observe the
explicit formula
\begin{equation}\label{eq:translatedGrad} \nabla^{\varphi^{(p^{-1})}}\varphi^{(p^{-1})}(w) = \nabla^{\varphi}\varphi(\pi_{\W}(p \cdot w)),
\qquad w \in \W, \: p \in \Phi(\W), \end{equation} proven in Lemma \ref{invintrgrad} below. Thus, starting from the left hand
side of \eqref{CFOdef},
\begin{align}
 |\nabla^{\varphi^{(p^{-1})}} \varphi^{(p^{-1})}(w) - \nabla^{\varphi^{(p^{-1})}}\varphi^{(p^{-1})}(0)|
 & = |\nabla^{\varphi}\varphi(\pi_{\W}(p \cdot w)) - \nabla^{\varphi}\varphi(\pi_{\W}(p))| \notag \\
&= |\nabla^{\varphi}\varphi(\pi_{\W}\left(\Phi(\pi_{\mathbb{W}}(p \cdot
w))\right)) - \nabla^{\varphi}\varphi(\pi_{\W}(p))| \notag
\\&\label{form33} \overset{\eqref{eq:HolForm}}{\lesssim}_L d_{\He}(\Phi(\pi_{\W}(p \cdot w)),p)^{\alpha}. \end{align}
The following formula is needed to make sense of the right hand side:
\begin{lemma}\label{l:translated} For any $p \in \He^n$ and $w \in \W$, the following relation holds:
\begin{equation}\label{form36} \Phi(\pi_{\W}(p \cdot w)) = p \cdot \Phi^{(p^{-1})}(w). \end{equation}
Here $\Phi^{(p^{-1})}$ is the graph map of $\varphi^{(p^{-1})}$.
\end{lemma}
\begin{proof} The function $\varphi^{(p^{-1})} \colon \W \to \V$ is explicitly given by
\begin{equation}\label{form34} \varphi^{(p^{-1})}(w) = \pi_{\V}(p^{-1}) \cdot \varphi(\pi_{\W}(p \cdot w)), \end{equation}
where $\pi_{\V}(x_1,\ldots,x_{2n},t) = x_1$ is the horizontal
projection induced by the splitting $\He^n = \W \cdot \V$, see for
instance \cite[Lemma 4.7]{CFO} (for $n=1$). The map $\pi_{\V}$ is
a group homomorphism $\He^n \to \V$ with $\pi_{\V}(w) = 0$ for all
$w = (0,x_2,\ldots,x_{2n},t) \in \W$. Therefore,
\begin{align*} & \Phi(\pi_{\W}(p \cdot w)) = p \cdot \Phi^{(p^{-1})}(w)\\
& \quad \overset{\eqref{form34}}{\Longleftrightarrow} \pi_{\W}(p \cdot w) \cdot \varphi(\pi_{\W}(p \cdot w)) = p \cdot [w \cdot \pi_{\V}(p^{-1}) \cdot \varphi(\pi_{\W}(p \cdot w))]\\
& \quad \Longleftrightarrow \pi_{\W}(p \cdot w) = p \cdot w \cdot \pi_{\V}(p^{-1}) = p \cdot w \cdot [\pi_{\V}(p)]^{-1} \\
& \quad \Longleftrightarrow p \cdot w = \pi_{\W}(p \cdot w) \cdot
\pi_{\V}(p) = \pi_{\W}(p \cdot w) \cdot \pi_{\V}(p \cdot w).
\end{align*} The last equation is true by the definition of
the projections $\pi_{\W}$ and $\pi_{\V}$, so the proof is
complete. \end{proof}

We deduce that the right hand side of \eqref{form33} equals, up to
a multiplicative constant, the term
$d_{\He}(\Phi^{(p^{-1})}(w),0)^{\alpha}$, which we will now
further bound from above in order to conclude the proof of the
first implication in Proposition \ref{equivProp}. To do so, we
observe that \cite[Theorem 1.2]{MR2223801} implies more than mere
intrinsic differentiability of $\varphi$: it shows that $\varphi$
is \emph{uniformly intrinsically differentiable} in the sense of
\cite[Definition 3.16]{MR2496655}. Recalling that $\varphi$ has
compact support, this implies that there exists a function
$\varepsilon=\varepsilon_{\mathrm{spt}\varphi}:(0,\infty) \to
(0,\infty)$ with $\lim_{s\to 0}\varepsilon(s)=0$ such that
\begin{displaymath}
 |\varphi^{(p^{-1})}(y,t) - \langle \nabla^{\varphi}\varphi(w_0),y\rangle|
\leq\varepsilon\left(\|(y,t)\|\right) \|(y,t)\|, \qquad
p=\Phi(w_0)\in \Phi(\W),\;(y,t) \in
\mathrm{spt}\,\varphi^{(p^{-1})}.
\end{displaymath}
Then there is  a constant $\delta>0$ such that
\begin{displaymath}
|\varphi^{(p^{-1})}(y,t)|\leq |\varphi^{(p^{-1})}(y,t) - \langle
\nabla^{\varphi}\varphi(w_0),y\rangle|+|\langle
\nabla^{\varphi}\varphi(w_0),y\rangle|\leq (1+L)\|(y,t)\|
\end{displaymath}
for all $p=\Phi(w_0)\in \Phi(\W)$ and all $(y,t)\in
\mathrm{spt}\,\varphi^{(p^{-1})}$ with $\|(y,t)\|\leq \delta$,
where $L:=\|\nabla^{\varphi}\varphi\|_{L^{\infty}(\W)}$. On the
other hand, if $\|(y,t)\|\in \mathrm{spt}\,\varphi^{(p^{-1})}$
satisfies $\|(y,t)\|> \delta$, then trivially,
\begin{displaymath}|\varphi^{(p^{-1})}(y,t)| = \|\pi_{\V}(\Phi(w_0)^{-1})\cdot \varphi(\pi_{\W}(\Phi(w_0)\cdot w)\|
\leq \tfrac{2\|\varphi\|_{L^{\infty}(\W)}}{\delta} \|(y,t)\|. \end{displaymath}
In conclusion, there exists a constant $L'\geq 1$ such that for
all $p\in \Phi(\W)$, it holds
\begin{equation}\label{eq:LipConclu}
\|\Phi^{(p^{-1})}(w)\|\leq \|w\|+ |\varphi^{(p^{-1})}(w)|\leq L'
\|w\|,\quad w\in \W.
\end{equation}
Finally, plugging formula \eqref{form36} into \eqref{form33} and
using the left-invariance of $d_{\He}$,
\begin{equation}\label{form35} |\nabla^{\varphi^{(p^{-1})}} \varphi^{(p^{-1})}(w) - \nabla^{\varphi^{(p^{-1})}}\varphi^{(p^{-1})}(0)|
\lesssim d_{\He}(\Phi^{(p^{-1})}(w),0)^{\alpha}
\overset{\eqref{eq:LipConclu}}{\lesssim} \|w\|^{\alpha}.
\end{equation}
Now \eqref{form35} proves that $\varphi$ satisfies \eqref{CFOdef}.

The converse implication stated in Proposition \ref{equivProp} is
not needed in the paper, so we only sketch the argument. Let
$\varphi$ be intrisically differentiable with intrinsic gradient satisfying
\eqref{CFOdef}. Then $\varphi$ is again uniformly intrinsically
differentiable. This is a consequence of Proposition
\ref{p:Approx},
 see also \cite[Remark 2.24]{CFO2}. Therefore, according to \cite[Theorem 1.2]{MR2223801},
  the intrinsic graph $S = \Phi(\W)$ is an $\He$-regular surface, and the condition
  (ii) in Definition \ref{C1alphaGraphs} holds.
   So, it remains to check that the horizontal normal of $S$ is $\alpha$-H\"older continuous.
   This follows from the expression \eqref{form32} (which is available by \cite[Theorem 1.2]{MR2223801})
   using estimates similar to those above \eqref{form39}.
   \end{proof}

  The proof of Proposition \ref{equivProp} referred to
  the following auxiliary lemma.

\begin{lemma}\label{invintrgrad}
Let $\varphi :\W \to \V$ be defined on the codimension-$1$ vertical subgroup
$\mathbb{W}\subset \He^n$. If $p\in \He^n$ and $w\in \W$ are such
that $\varphi$ is intrinsically differentiable at $\pi_{\W}(p\cdot
w)$, then
\begin{equation}\label{eq:LeftInvGrad} \nabla^{\varphi^{(p^{-1})}}\varphi^{(p^{-1})}(w) = \nabla^{\varphi}\varphi(\pi _\W(p\cdot w)). \end{equation}
\end{lemma}

\begin{proof}
By its very definition,  a function $\psi:\mathbb{W}\to
\mathbb{V}$ is intrinsically differentiable at a point $w_0\in
\mathbb{W}$  if and only if $\psi^{(p_0^{-1})}$ for $p_0=w_0\cdot
\psi(w_0)$ is intrinsically differentiable at $0$, and in that
case
\begin{equation}\label{eq:LeftInvGrad}
\nabla^{\psi}\psi(w_0)=
\nabla^{\psi^{(p_0^{-1})}}\psi^{(p_0^{-1})}(0),
\end{equation}
 see, for instance, \cite[Definition 4.71, Proposition 4.76.]{MR3587666} and \cite[top of p.192]{MR2223801}.
Now fix $\varphi$, and points $w \in \W$, and  $p \in \He^n$ as in
the statement of the lemma. We first apply formula
\eqref{eq:LeftInvGrad} to $\psi:=\varphi$ and $w_0:=\pi_{\W}(p
\cdot w)$. Hence,
\begin{displaymath}
p_0 = w_0 \cdot \varphi(w_0) = \pi_{\W}(p \cdot w) \cdot
\varphi(\pi_{\W}(p \cdot w)) = p \cdot w\cdot \pi_{\V}(p^{-1})
\cdot \varphi(\pi_{\W}(p \cdot w))\overset{\eqref{form34}}{=} p
\cdot w\cdot \varphi^{(p^{-1})}(w)
\end{displaymath}
and \eqref{eq:LeftInvGrad} reads
\begin{equation}\label{eq:grad1}
\nabla^{\psi}\psi(\pi _\W(p\cdot w))= \nabla^{\varphi^{([ p \cdot
w\cdot \varphi^{(p^{-1})}(w)]^{-1})}}\varphi^{([ p \cdot w\cdot
\varphi^{(p^{-1})}(w)]^{-1})}(0).
\end{equation}
On the other hand, denoting \begin{displaymath} q_0  = w \cdot
\varphi^{(p^{-1})}(w),
\end{displaymath}
we observe that the graph of $[\varphi^{(p^{-1})}]^{(q_0^{-1})}$
is $q_0^{-1}\cdot \Gamma'$, where $\Gamma'$ is the graph of
$\varphi^{(p^{-1})}$, so
\begin{displaymath}
q_0^{-1}\cdot \Gamma'= q_0^{-1} \cdot p^{-1}\cdot
\Phi(\mathbb{W})=[p\cdot q_0]^{-1} \cdot \Phi(\mathbb{W}).
\end{displaymath}
This shows that
\begin{displaymath}
\varphi^{([p\cdot q_0]^{-1})}= [\varphi^{(p^{-1})}]^{(q_0^{-1})}
\end{displaymath}
and hence
\begin{displaymath}
\eqref{eq:grad1}= \nabla^{\varphi^{([p\cdot
q_0]^{-1})}}\varphi^{([p\cdot q_0]^{-1})}(0)=
\nabla^{[\varphi^{(p^{-1})}]^{(q_0^{-1})}}[\varphi^{(p^{-1})}]^{(q_0^{-1})}(0).
\end{displaymath}
In particular, $[\varphi^{(p^{-1})}]^{(q_0^{-1})}$ is
intrinsically differentiable at $0$. Formula
\eqref{eq:LeftInvGrad} applied to $\psi:=\varphi^{(p^{-1})}$,
$p_0=q_0$ and $w_0:=w$ yields
\begin{equation}\label{eq:grad2}
\nabla^{\varphi^{(p^{-1})}}\varphi^{(p^{-1})}(w) =
\nabla^{[\varphi^{(p^{-1})}]^{(q_0^{-1})}}[\varphi^{(p^{-1})}]^{(q_0^{-1})}(0). \end{equation}
The lemma follows by observing that the right hand sides of
\eqref{eq:grad1} and \eqref{eq:grad2} are equal.
\end{proof}

With the main definitions now in place, we repeat Theorem
\ref{mainIntro} below.

\begin{thm}\label{mainGraphs} Let $S = \Phi(\W) \subset \He^n$, where $\varphi \in C^{1,\alpha}_{\He}(\W)$ is compactly supported.
Then $S$ has BP$G$BI in the sense of Definition \ref{BPdef}.
\end{thm}

Before proving this, let us deduce the qualitative corollary,
Theorem \ref{mainQualitative}:

\begin{thm}\label{surfaces} Let $S \subset \He^n$ be a $C^{1,\alpha}_{\He}$-surface. Then
$\mathcal{H}^{2n+1}$ almost all of $S$ can be covered by
bilipschitz images of closed subsets of codimension-1 vertical
subgroups. In particular, $S$ is LI rectifiable. \end{thm}
\begin{proof} There are (at least) two possible approaches. One is to use the implicit function theorem \cite[Theorem 6.5]{FSSC}
to express the surface $S$ locally as the intrinsic graph of a \emph{locally defined} intrinsic $C^{1,\alpha}$-function.
 This function does not, literally, satisfy the assumptions of Theorem \ref{mainGraphs}, but the proof of Theorem \ref{mainGraphs} could be localised with some effort.

The biggest difficulty in this approach is of expository nature as
localising the proof of Theorem \ref{mainGraphs} would lead to a
more cumbersome version of
 Proposition \ref{p:Approx} below. So we take the following alternative route: in Appendix
\ref{extensionAppendix}, we show that every point on $S$ has a
neighbourhood which can be contained on the intrinsic graph
$\Gamma$ of a globally defined, compactly supported intrinsic
$C^{1,\alpha/3}$-function. Since $\Gamma$ is LI rectifiable by
Theorem \ref{mainGraphs}, the same is also true for $S$.
\end{proof}

It remains to prove Theorem \ref{mainGraphs}. Recall the notation
$\V$ and $\W$ from Definition \ref{C1alphaGraphs}. We will
frequently abbreviate $(0,x_2,\ldots,x_{2n},t) =:
(x_2,\ldots,x_{2n},t)=:(y,t)$ for points in $\W$, and continue to
use the notation
\begin{displaymath}
\langle \nabla^{\varphi}\varphi(w),y\rangle =\sum_{i=2}^{2n}
D_i^{\varphi} \varphi(w) x_i
\end{displaymath}
for an intrinsically differentiable $\varphi$. The functions
$D_i^{\varphi} \varphi$ have been introduced as the components of
the intrinsic gradient $\nabla^{\varphi}\varphi$ in Definition
\ref{d:IntrGrad}, but under additional regularity assumptions on
$\varphi$, they can also be obtained as derivatives of $\varphi$
in the direction of the vector fields
\begin{displaymath}
D^{\varphi}_j: = X_j,\quad j\in \{2,\ldots,2n\}\setminus
\{n+1\}\quad\text{and}\quad D^\varphi _{n+1}: = \partial_{x_{n+1}}
+ \varphi
\partial_t.
\end{displaymath} A first result of this type is \cite[Proposition
3.7]{MR2223801}, and more specific statements will follow shortly
in the proof of the next result, which is a key ingredient in the
proof of Theorem \ref{mainGraphs}.
\begin{proposition}\label{p:Approx} Assume that $\varphi:\W \to \V$
is intrinsically differentiable on $\W$ and it has a continuous
intrinsic gradient which satisfies \eqref{CFOdef} with constant
$H\geq 1$ and $0<\alpha\leq 1$. Suppose further that
$L:=\|\nabla^{\varphi}\varphi\|_{L^{\infty}(\W)}<\infty$, and $p =
\Phi(w)$ for some $w \in \W$. Then,
\begin{equation}\label{CFOEstimate} |\varphi^{(p^{-1})}(y,t) - \langle \nabla^{\varphi}\varphi(w),y\rangle|
\lesssim \|(y,t)\|^{1 + \alpha}, \qquad (y,t) \in \W.
\end{equation} The implicit constant in \eqref{CFOEstimate} only
depends on $H$  and $L$.
\end{proposition}

The proposition says, in a "left-invariant" way, that $\varphi$ is
locally well-approximated by linear functions. The main corollary
is Proposition \ref{approxProp}, which quantifies how well the
intrinsic graph of $\varphi$ around $\Phi(w)$ is approximated by
the \emph{vertical tangent plane} determined by
$\nabla^{\varphi}\varphi(w)$.

\begin{proof}[Proof of Proposition \ref{p:Approx}]
For $n=1$ the proposition was established in \cite[Proposition
2.23]{CFO2}. The case $n>1$ can be proven in a simpler way without
the arguments that were used in \cite[Proposition 4.2]{CFO2}. We
include here a self-contained proof for that case. By the
definition of the intrinsic differentiability and intrinsic
gradients, we have
\begin{equation*}
|\varphi ^{(p^{-1})}(y,t) -
\langle\nabla^{\varphi}\varphi(w),y\rangle | = |\varphi
^{(p^{-1})} (y,t) - \langle \nabla ^{\varphi ^{(p^{-1})}} \varphi
^{(p^{-1})}(0) , y\rangle |
\end{equation*}
and so we just prove that
\begin{equation}\label{finale1}
 |\varphi  (y,t) -   \langle \nabla ^{\varphi } \varphi (0), y \rangle| \lesssim \| (y,t)\|^{\alpha + 1}, \quad \mbox{ for all } (y,t) \in \W,
\end{equation}
under the assumption that $p=0$ and $\varphi (0)=0$. Notice that
the constants $L$ and $H$ are not changed under left translations.

To explain the idea, let us first consider a $C^{1,\alpha}(\R)$
function $h:\R \to \R$ with $h(0)=0$. Then, for $y>0$,
\begin{displaymath}
|h(y)-h'(0)y|=\left|\int_0^y h'(s)-h'(0)\,ds\right|\leq \int_0^y
|h'(s)-h'(0)|\,ds \lesssim |y|^{1+\alpha}.
\end{displaymath}
To prove \eqref{finale1}, we apply the same idea, but we integrate
along integral curves of vector fields $D_j^{\varphi}$,
$j=2,\ldots,2n$. We use the assumption $n>1$ to ensure that the
origin can be connected to any point
\begin{displaymath}
(y,t)=(0,x_2,\ldots,x_{2n},t)\in \W \end{displaymath} by a curve
 $\gamma :I \to \W$ that is defined as a
concatenation $\gamma := \gamma_1 \star \cdots \star \gamma
_{2n+3} $ of the following curves:
\begin{itemize}
\item $\gamma_1$ is an integral curve of
\begin{displaymath}
D^\varphi _{n+1} = \partial_{x_{n+1}} + \varphi \partial_t
\end{displaymath}
that connects $0$ to a point of the form
\begin{displaymath}a:= (0,\ldots , 0 ,x_{n+1},0,\ldots,0, \tau
(x_{n+1})).\end{displaymath} \item $ \gamma _2 \star \dots \star
\gamma _{2n+3}$ is a concatenation of integral curves of
\begin{displaymath}
D^{\varphi}_j = X_j,\quad j\in \{2,\ldots,2n\}\setminus \{n+1\}
\end{displaymath}
with the property that
$$\mbox{length}_{\mathbb{H}^n} ( \gamma _2 \star \dots  \star
\gamma _{2n+3}) \lesssim \|(y,t)\|$$ and  $ \gamma _2 \star \dots
\star \gamma _{2n+3}$ connects $a$ to $(y,t)$.
\end{itemize}
A similar construction was used in \cite[Proposition
6.10]{antonelli2019pauls}. We now explain in detail how
$\gamma_1,\ldots,\gamma_{2n+3}$ are defined. First, let
$\lambda_{n+1}$ be an integral curve of $D_{n+1}^{\varphi}$, given
by
\begin{equation}\label{eq:lambda_n+1}
\lambda_{n+1}(s):= (0,\ldots 0,s,0,\ldots,\tau (s)),
\end{equation} where $s$ is the component corresponding to the
coordinate $x_{n+1}$, and $\tau: J\to \R$ is a solution of the
Cauchy problem
 \begin{equation*}
\left\{
\begin{array}{lcl}
\tau '(s)= \varphi(0, \ldots , 0 , s, 0,\ldots,0,\tau (s)),  \\
\tau (0)=0. \\
\end{array}
\right.
\end{equation*}
Such a solution exists by Peano's theorem since $\varphi$ is
continuous, and we assume that $J$ is the maximal interval of
existence for $\tau$ containing the point $0$. As moreover
$\nabla^{\varphi}\varphi$ is continuous, which follows from the
assumption \eqref{CFOdef} by Remark \ref{r:GradProp}, we find by
\cite[Lemma 4.4]{CFO2} that
\begin{equation}\label{(6.31).2:_n+1}
\begin{aligned}
 (\varphi \circ \lambda _{n+1})' (s) &= D^\varphi_{n+1}\varphi (\lambda _{n+1}(s)),\quad s\in J. \\
\end{aligned}
\end{equation}
To be precise, \cite[Lemma 4.4]{CFO2} is stated in $\mathbb{H}^1$,
but since $\lambda_{n+1}$ is entirely contained in the
$x_{n+1}t$-plane, we can apply the result to
\begin{displaymath}
(x_{n+1},t)\mapsto \varphi(0,\ldots,0,x_{n+1},0,\ldots,0,t),
\end{displaymath}
interpreted as a function on a vertical subgroup in
$\mathbb{H}^1$. Moreover, using the same proof as for
\cite[(4.4)]{CFO2}, one can show that in fact $J=\R$, and
\begin{equation}\label{(6.31)}
|\tau (s) | \lesssim _L |s|^2, \qquad s\in \R.
\end{equation}

The curve $\gamma_1$ will be an appropriately parametrized
subcurve of $\lambda_{n+1}$. If $x_{n+1}>0$,  we naturally define
$\gamma_1$ to be the restriction of $\lambda_{n+1}$ to the
interval $[0,x_{n+1}]$. If $x_{n+1}<0$, we have to reverse the
order of the parametrization, so we make the general definition:
$\gamma_1(s):= \lambda_{n+1}(\mathrm{sgn}(x_{n+1})s)$, $s\in
[0,|x_{n+1}|]$, if $x_{n+1}\neq 0$, and we let $\gamma_1$ be the
constant curve $\gamma_1\equiv 0$ otherwise. Note that the points
$a:=\gamma_1(|x_{n+1}|)$ and $(y,t)$ belong to
\begin{displaymath}
\mathbb{G}:=\{(z_1,\ldots,z_{2n+1})\in \mathbb{R}^{2n+1}:\,
z_1=0\quad\text{and}\quad z_{n+1}=x_{n+1}\},
\end{displaymath}
and $\mathbb{G}$ with the group law and metric induced from
$\mathbb{H}^n$ is isometrically isomorphic to $\mathbb{H}^{n-1}$
with $d_{\mathbb{H}}$. We next show that there exists a compact
interval $I=[|x_{n+1}|,b]$ such that $a$ and $(y,t)$ can be
connected by a concatenation $\gamma _2 \star \dots \star \gamma
_{2n+3}:I\to \mathbb{G}$ of integral curves of $D^\varphi _j=X_j$
, $j \in \{2,\dots, 2n \}\setminus\{n+1\}$ so that
\begin{equation}\label{equation1803}
\max_{s\in I} \|(\gamma _2 \star \dots  \star \gamma _{2n+3}) (s)
\| \lesssim_L \|(y,t)\| \mbox{\,\, and \,\, length}_{\mathbb{H}^n}
( \gamma _2 \star \dots  \star \gamma _{2n+3}) \lesssim
d_{\mathbb{H}}(a,(y,t)).
\end{equation}
The first inequality in \eqref{equation1803} follows from the
second one since
\begin{equation}\label{eq:a_2}
\|a\|\overset{\eqref{(6.31)}}{\lesssim_L } |x_{n+1}|\quad
\text{and}\quad  |x_{n+1}|\leq\|(y,t)\|
\end{equation}
and
\begin{align*}\|(\gamma _2 \star \dots  \star \gamma
_{2n+3}) (s) \|  &\leq  \mathrm{length}_{\mathbb{H}^n} ( \gamma _2
\star \dots \star \gamma _{2n+3})+\|a\|,\quad s\in I.
\end{align*}
Thus it suffices to find curves which satisfy the second condition
in \eqref{equation1803}. To achieve this, first concatenate curves
$\gamma_2,\ldots,\gamma_{2n-1}$ in the following way: follow the
unique integral curve of $D^\varphi _2 =X_2$ starting at $a$ for
time $x_2$, then follow the integral curve of $D^\varphi _{j}
=X_{j}$ for time $x_{j}$, etc.\ for $j=3,\ldots,n,n+2,\ldots,2n$,
until you reach $a' =(y,\tau (x_{n+1}) + \frac 1 2 \sum_{i=2}^n
x_{i}x_{n+i})$. Then connect $a'$ to $(y,t)$ by a curve $
\gamma_{2n}\star \dots \star \gamma _{2n+3}$ with
\begin{displaymath}
\mathrm{length}_{\mathbb{H}^n}(\gamma_{2n}\star \dots \star \gamma
_{2n+3}) \lesssim \left|t-\left(\tau (x_{n+1}) + \tfrac{ 1}{ 2}
\sum_{i=2}^n x_{i}x_{n+i}\right)\right|^{1/2}.
\end{displaymath}
This is possible since $[X_2,X_{n+2}]=\partial_t$. Now $\gamma _2
\star \dots \star \gamma _{2n+3}$ connects $a$ to $(y,t)$ as
desired, and
\begin{equation}\label{nuvola}
\begin{aligned}
 \mbox{length}_{\mathbb{H}^n} ( \gamma_{2} \star \dots   \star \gamma _{2n+3})&  \lesssim |x_2|+\cdots+|x_{2n}|+
 \left|t-\left(\tau (x_{n+1}) + \tfrac{ 1 }{2}
\sum_{i=2}^n x_{i}x_{n+i}\right)\right|^{1/2}\\
& \lesssim d_{\mathbb{H}}(a,(y,t)) + |\tau (x_{n+1}) -t|^{1/2}\\
 & \lesssim d_{\mathbb{H}}(a,(y,t)).\\
 \end{aligned}
\end{equation}
Hence we have found curves $\gamma_2,\ldots,\gamma_{2n+3}$ so that
the conditions in \eqref{equation1803} are satisfied.

Now we are ready to implement the idea explained at the beginning
of the proposition. Namely, we will write $\varphi(y,t)$ as an
integral of $(\varphi \circ \gamma)'$, where $\gamma$ is the
concatenation of the curves $\gamma_1,\ldots,\gamma_{2n+3}$. To
relate this to the intrinsic gradient, we apply again the
intrinsic differentiability of $\varphi$, and observe that
 \begin{equation}\label{(6.31).2}
\begin{aligned}
 (\varphi \circ \lambda _j)' (s) &= D^\varphi_j\varphi (\lambda _j(s)),\\
\end{aligned}
\end{equation}
for all $j=2,\ldots,2n$, where $\lambda_{n+1}$ is as in
\eqref{eq:lambda_n+1}, and $\lambda_j$ for $j\neq n+1$ is an
arbitrary integral curve of $D_j^{\varphi}=X_j$. For $j=n+1$, we
saw \eqref{(6.31).2} already in \eqref{(6.31).2:_n+1}. For $j\neq
n+1$, the vector field $ D^\varphi_j$ is linear and independent of
$\varphi$, and \eqref{(6.31).2} follows directly from the
intrinsic differentiability assumption by the argument given at
the beginning of the proof of \cite[Proposition 3.7]{MR2223801}.
Hence, if $\gamma$ is the piecewise $C^1$ curve given by
$$\gamma (s)=(\gamma_1 \star \cdots \star \gamma
_{2n+3}) (s)=(y(s),t(s)),\quad s\in[0,b],$$ we get that $\varphi
\circ \gamma$ is piecewise $C^1$ and
\begin{equation}\label{stima702}
\begin{aligned}
 |\varphi(y,t)  - \langle \nabla ^{\varphi } \varphi (0) ,y \rangle| & =
 \left| \int_0^b (\varphi \circ \gamma )' (s)\, \mathrm{d} s -\int _0^b  \langle \nabla ^{\varphi } \varphi (0) , \dot{y}(s)\rangle  \, \mathrm{d} s \right|\\
 & \overset{\eqref{(6.31).2}}{\leq} \int _0^b | \nabla ^\varphi \varphi (\gamma (s) ) - \nabla ^\varphi \varphi (0)| \, \mathrm{d} s, \\
& \lesssim _H \int _0^b \|\gamma (s)\|^\alpha   \, \mathrm{d} s \\
\end{aligned}
\end{equation}
where in the last inequality we have used the assumption that
$\varphi$ satisfies \eqref{CFOdef}. In the first inequality, we
used the fact that for almost all $s\in [0,b]$, the tangent vector
$\dot{y}(s)$ exists by construction, and is of the form
$(0,\ldots,0,\pm 1,0,\ldots,0)$ with $$(\varphi \circ
\gamma)'(s)\overset{\eqref{(6.31).2}}{=}
D_j^{\varphi}\varphi(\gamma(s))=\langle
\nabla^{\varphi}\varphi(\gamma(s)),\dot{y}(s)\rangle$$ if the
$x_j$-component of $\dot{y}(s)$ is $+1$, and $$(\varphi \circ
\gamma)'(s)\overset{\eqref{(6.31).2}}{=}
-D_j^{\varphi}\varphi(\gamma(s))=\langle
\nabla^{\varphi}\varphi(\gamma(s)),\dot{y}(s)\rangle$$ if the
$x_j$-component of $\dot{y}(s)$ is $-1$. Having established
\eqref{stima702}, we will estimate from above the expression
\begin{equation*}
  \int _0^b \|\gamma (s)\|^\alpha  \, \mathrm{d} s  = \int _0^{|x_{n+1}|} \|\gamma _1(s)\|^\alpha \, \mathrm{d} s +
   \int _{|x_{n+1}|}^b  \| (\gamma _2 \star \dots  \star \gamma _{2n+3}) (s)\|^\alpha \,
  \mathrm{d} s.
\end{equation*}
Firstly, since \eqref{(6.31)} holds, we have
\begin{equation*}
\begin{aligned}
 \int _0^{|x_{n+1}|} \|\gamma _2(s)\|^\alpha \, \mathrm{d} s & \lesssim
\int _0^{|x_{n+1}|} |s |^\alpha + \sqrt{|\tau (\mathrm{sgn}(x_{n+1})s)|^\alpha } \, \mathrm{d} s  \lesssim _L \| (y,t)\|^{\alpha +1} \\
\end{aligned}
\end{equation*}
and secondly, by definition of $\gamma_j$, $j \in \{2,\dots, 2n
+3\} $, we have
\begin{equation*}\begin{aligned}
|b-|x_{n+1}||& = \mbox{length}_{\mathbb{H}^n} (\gamma _2 \star
\dots \star \gamma _{2n+3})
\stackrel{\eqref{eq:a_2},\eqref{nuvola}}{ \lesssim }_L  \| (y,t)\|. \\
 \end{aligned}
\end{equation*}
This combined with the first condition in \eqref{equation1803}
yields
\begin{equation*}
\begin{aligned}
 \int _{|x_{n+1}|}^b  \| (\gamma _2 \star \dots  \star \gamma _{2n+3}) (s)\|^\alpha \,ds \lesssim_L  \| (y,t)\|^{\alpha +1} .
\end{aligned}
\end{equation*}
 Putting together \eqref{stima702} and the last estimates, we can conclude
\begin{equation*}
\begin{aligned}
 |\varphi  (y,t) - \langle \nabla ^{\varphi } \varphi (0), y\rangle|
 & \lesssim_H \int _0^b \|\gamma (s)\|^\alpha  \, \mathrm{d} s   \lesssim_{H,L} \| (y,t)\|^{\alpha +1}\\
\end{aligned}
\end{equation*}
and the proof of Proposition  \ref{p:Approx} is complete.
\end{proof}

\begin{remark}\label{r:constants}
Recall from Proposition \ref{equivProp} that a compactly supported
function $\varphi \in C^{1,\alpha}_{\He}(\W)$ (as in Theorem
\ref{mainGraphs}) satisfies \eqref{CFOdef} for some constant $H
\geq 1$. The letter $H$ will refer to this constant for the rest
of Section \ref{s:C1alpha}. We also remark that $\varphi$ is
intrinsic Lipschitz, recall Definition \ref{d:intrinsicGraphs}.
Indeed, a $C^{1,\alpha}_{\He}(\W)$ function is intrinsically
differentiable with continuous intrinsic gradient, and the compact
support assumption implies that $\nabla^{\varphi}\varphi \in
L^{\infty}(\W)$. In the case $n=1$, \cite[Lemma 2.22]{CFO2} states
that then $\varphi$ is intrinsic Lipschitz. One could adapt the
proof to higher dimensions using Proposition \ref{p:Approx}, or
alternatively refer to \eqref{eq:LipConclu} to conclude also for
$n>1$ that $\varphi$ is intrinsic Lipschitz. We denote by $L$ the
maximum of the intrinsic Lipschitz constant of $\varphi$, and the
sup-norm $\|\nabla^{\varphi}\varphi\|_{L^{\infty}(\W)}$. The
compact support assumption of $\varphi \in C^{1,\alpha}_{\He}(\W)$
is initially needed to ensure that $\max\{H,L\} < \infty$.
However, the constants are then left-invariant: if $p \in \He$,
then $\varphi^{(p^{-1})}$ is intrinsically differentiable, its
intrinsic gradient is continuous, and satisfies \eqref{CFOdef}
 with the same constant $H$ (see
\cite[Lemma 2.25]{CFO2}),
$\|\nabla^{\varphi^{(p^{-1})}}\varphi^{(p^{-1})}\|_{L^{\infty}(\W)}\leq
L$ by Lemma \ref{invintrgrad}, and $\varphi^{(p^{-1})}$ is
intrinsic Lipschitz with constant $L$, even though the support of
$\varphi^{(p^{-1})}$ of course depends on $p$.
\end{remark}

We use Proposition \ref{p:Approx} to quantify how well the
intrinsic graph $S\subset \He^n$ of a compactly supported
$C^{1,\alpha}_{\He}$ function is approximated at a point $p\in S$
by a certain vertical plane. For $p  \in S$, let $\W_p = W_p
\times \mathbb{R}$ be the unique vertical subgroup with the
property that $W_p$ is a $(2n-1)$ dimensional subspace of
$\mathbb{R}^{2n}$ which is perpendicular to the line spanned by
the vector $\nu_{\He}(p)$ using coordinates as in \eqref{form32}.

\begin{proposition}\label{approxProp} Fix $n\in \N$ and $0<\alpha\leq 1$. Let $\varphi \in C^{1,\alpha}_{\He}(\W)$ be compactly supported
on the codimension-$1$ vertical subgroup $\W\subset \He^n$, and
write $S := \Phi(\W)$. Then, there exists a constant $A = A(H,L)
\geq 1$ such that for every $p \in S$,
\begin{equation}\label{form20} \dist_{\He}(q,S) \leq A d_{\He}(p,q)^{1 + \alpha}, \qquad q \in p \cdot \W_{p}. \end{equation}
Here $H$ and $L$ are defined as in Remark \ref{r:constants}, that
is, $\varphi$ satisfies \eqref{CFOdef} with constant $H$,
$\varphi$ is intrinsic $L$-Lipschitz, and
$\|\nabla^{\varphi}\varphi\|_{L^{\infty}(\W)}\leq L$.
\end{proposition}

\begin{proof}[Proof of Proposition \ref{approxProp}] The plan is to apply estimate \eqref{CFOEstimate} from Proposition \ref{p:Approx}. Note that
\begin{displaymath}
L_{p}(y,t) := (0,y,t) \cdot
(\langle\nabla^{\varphi}\varphi(w),y\rangle,0,0)
\end{displaymath}
 defines a map $\W \to \W_{p}$ since
 \begin{equation}\label{eq:W_p}
 W_p = \left\{\left(\sum_{i=2}^{2n} D_i^\varphi\varphi(w)x_i,x_2,\ldots,x_{2n}\right):\; (x_2,\ldots,x_{2n})\in \mathbb{R}^{2n-1}\right\}
 \end{equation}
 is a $(2n-1)$-plane perpendicular to
 \begin{displaymath}
\begin{pmatrix}-1\\D_2^{\varphi}\varphi(w)\\\vdots\\ D_{2n}^{\varphi}\varphi(w)\end{pmatrix}
 \end{displaymath}
  in $\mathbb{R}^{2n}$.
 In fact, $L_{p}$ is a bijection $\W \to \W_{p}$, and for all $ (y,t)=(0,x_2,\ldots,x_{2n},t) \in
 \W$,
\begin{equation}\label{form22} \|L_{p}(y,t)\| = \left\|\left(\langle\nabla^{\varphi}\varphi(w),y\rangle,y,t -
\tfrac{1}{2}x_{n+1}\langle\nabla^{\varphi}\varphi(w),y\rangle\right)\right\|
\sim_{L} \|(y,t)\|,
\end{equation} because $|\nabla^{\varphi}\varphi(w)| \leq
\|\nabla^{\varphi}\varphi\|_{L^{\infty}(\W)} \leq L$. The last
observation immediately implies the inequality ``$\lesssim_L$'' in
\eqref{form22}. It also gives the converse inequality, since
\begin{align*}
\|(y,t)\|&\lesssim |y| + |t|^{\frac{1}{2}}\lesssim |y| + \left|t -
\tfrac{1}{2}x_{n+1}\langle\nabla^{\varphi}\varphi(w),y\rangle\right|^{\frac{1}{2}}
+
\left|\tfrac{1}{2}x_{n+1}\langle\nabla^{\varphi}\varphi(w),y\rangle\right|^{\frac{1}{2}}\\&\lesssim_L
|y|+ \left|t -
\tfrac{1}{2}x_{n+1}\langle\nabla^{\varphi}\varphi(w),y\rangle\right|^{\frac{1}{2}}
\lesssim_L  \|L_{p}(y,t)\|.
\end{align*}
Moreover, for arbitrary $w=(y,t)\in \W$, we have $d_{\He}(q,S)
\leq d_{\He}(q,p \cdot \Phi^{(p^{-1})}(y,t))$, where
$\Phi^{(p^{-1})}$ is the graph map of $\varphi^{(p^{-1})}$, simply
because $\Phi^{(p^{-1})}(\W) = p^{-1} \cdot S$. Therefore, by the
left-invariance of $d_{\He}$, one has for $q=p\cdot L_{p}(y,t)$,
\begin{align*} \dist_{\He}(q,S) \leq d_{\He}(p^{-1} \cdot q,\Phi^{(p^{-1})}(y,t))
& = d_{\He}(L_p(y,t),\Phi^{(p^{-1})}(y,t))\\
& = |\varphi^{(p^{-1})}(y,t) - \langle\nabla^{\varphi}\varphi(w),y\rangle| \stackrel{\eqref{CFOEstimate}}{\lesssim}_{H,L} \|(y,t)\|^{1 + \alpha} \notag\\
& \stackrel{\eqref{form22}}{\sim}_{H,L} \|L_{p}(y,t)\|^{1 +
\alpha} = d_{\He}(p,q)^{1 + \alpha}. \notag\end{align*} This
proves \eqref{form20}, and hence the proposition. \end{proof}

\subsubsection{Reduction to unit scale}\label{ss:red_unit_scale}
The rest of Section \ref{s:C1alpha} is devoted to the proof of
Theorem \ref{mainGraphs} in the case $n>1$. With the earlier
preparations in place, a proof for the case $n=1$ could be
obtained along the same lines, but some steps would require a
separate discussion. Instead, we will deduce the case $n=1$ later
from a more general result, see Theorem \ref{mainGraphsn_1}. So we
fix $n>1$ for the rest of this section, and constants will be
allowed to depend on $n$ without special mentioning.

Theorem \ref{mainGraphs} for $n>1$ is essentially a corollary of
Theorem \ref{main} applied to
\begin{equation}
\label{eq:G_M} (G,d_{G},\mu) =(\mathbb{H}^{n-1}\times
\mathbb{R},d_{\He^{n-1}\times \mathbb{R}},\mathcal{L}^{2n})\quad
\text{and} \quad (M,d_{M}) = (S,d_{\He}), \end{equation} where
$(G,d_G)$ is defined as in \eqref{eq:HeisxR} and the line below
it.

Once the hypotheses of Theorem \ref{main} have been verified -- a
task occupying the next section -- the theorem will yield the
existence of $2L'$-bilipschitz maps $f \colon K \to S \cap
B(p,1)$, $p \in S$, where $K \subset G$ with $\calH^{2n+1}(K) \geq
\delta
>0$. The constant $\delta > 0$ will only depend on $\alpha$, the H\"older
constant $H$ in \eqref{CFOdef}, the constant $L$ that bounds the
intrinsic Lipschitz constant of $\varphi$ and the
$L^{\infty}$-norm of $\nabla^{\varphi}\varphi$, whereas the
constant $L'$ will depend only on $L$. This is saying, in particular, that the BP$G$BI condition holds at
unit scale. How about other scales? The following easy lemma shows
that property \eqref{CFOdef} improves under ``zooming in'':
\begin{lemma}\label{blowUpLemma} Let $\varphi:  \W \to \V$ be
intrinsically differentiable with continuous intrinsic gradient
$\nabla^{\varphi}\varphi$ that satisfies \eqref{CFOdef} with
constants  $\alpha > 0$ and $H \geq 1$. For $r>0$, let
\begin{displaymath} \varphi_{r}(w) := \tfrac{1}{r}[\varphi \circ \delta_{r}],\quad w\in \W. \end{displaymath}
Then, $\psi:=\varphi_{r}$ is an intrinsically differentiable
function with intrinsic graph $\delta_{1/r}(\Phi(\W))$, its
intrinsic gradient is continuous, and satisfies \eqref{CFOdef}
with constants $\alpha$ and $r^{\alpha}H$, that is
\begin{displaymath}
|\nabla^{\psi^{(p^{-1})}}\psi^{(p^{-1})}(w)-\nabla^{\psi^{(p^{-1})}}\psi^{(p^{-1})}(0)|\leq
r^{\alpha} H \|w\|^{\alpha},\quad w\in \W,p\in
\delta_{1/r}(\Phi(\W)).
\end{displaymath}
\end{lemma}

\begin{proof}
Fix  $0 < r \leq 1$ and let $w$ be an arbitrary point in $\W$.
Then
\begin{displaymath}
\delta_{\frac{1}{r}}[w\cdot \varphi(w)] =
\delta_{\frac{1}{r}}(w)\cdot \delta_{\frac{1}{r}}(\varphi(w)) =
\delta_{\frac{1}{r}}(w)\cdot
\delta_{\frac{1}{r}}(\varphi(\delta_r(\delta_{\frac{1}{r}}(w)))),
\end{displaymath}
which shows that $\delta_{\frac{1}{r}}(\Phi(\W))$ is the intrinsic
graph of $\psi=\varphi_r$ as defined in the lemma. Since the
Heisenberg dilations are group isomorphisms which commute with
vertical projections, it is easy to see  that $\psi$ is
intrinsically differentiable with intrinsic gradient
\begin{displaymath}
\nabla^{\psi}\psi = \nabla^{\varphi}\varphi \circ \delta_r.
\end{displaymath}
Moreover, since $\psi^{(p^{-1})} = \frac{1}{r}
\varphi^{(\delta_r(p)^{-1})} \circ \delta_r$ for $p\in
\delta_{1/r}(\Phi(\W))$, we have
\begin{displaymath}
\nabla^{\psi^{(p^{-1})}}\psi^{(p^{-1})}=
\nabla^{\varphi^{(\delta_r(p)^{-1})} }
\varphi^{(\delta_r(p)^{-1})} \circ \delta_r,
\end{displaymath}
which yields the remaining claims in the lemma.
\end{proof}

Returning to the proof of Theorem \ref{mainGraphs}, let $p \in S$
and, first, $0 < r \leq C$, where $C := 2\diam_{\He}(\Phi(\spt
\varphi))$. Using the previous lemma, and also recalling that
dilations have no effect on intrinsic Lipschitz constants,
$S_{1/r} := \delta_{1/r}(S)$ is an intrinsic graph of an intrinsic
Lipschitz function with essentially bounded intrinsic gradient
satisfying \eqref{CFOdef} with constants depending only on the
corresponding constants for $S$, and $C$.
  Therefore,
by the BP$G$BI property at scale $r = 1$, to be established in the
next section, every ball $S_{1/r} \cap B(p,1)$ contains the image
of a $2L'$-bilipschitz map $g$ from a compact set $K \subset G$
with $\calH^{2n+1}(K) \geq \delta = \delta(C) > 0$. Now, one may
simply pre- and post-compose $g$ with the natural dilations in $G$
and $\He^n$ to produce a $2L'$-bilipschitz map $g_{r} \colon
\delta_{r}(K) \to S \cap B(\delta_{r}(p),r)$ (note also that
$\calH^{2n+1}(\delta_{r}(K)) = r^{2n+1}\calH^{2n+1}(K) \geq \delta
r^{2n+1}$).

Next, consider the case $r > C$. Then, if $p \in S$ is arbitrary,
the set $S \cap B(p,r)$ satisfies
\begin{displaymath} \calH^{2n+1}([S \cap B(p,r)] \cap \W) \gtrsim \calH^{2n+1}(S \cap B(p,r)). \end{displaymath}
Thus, the restriction of $\mathrm{Id}$ to $[S \cap B(p,r)] \cap
\W$ (composed with an isometry $G \cong \W$) yields the desired
bilipschitz map in this case.

\subsection{Proof for graphs with H\"older continuous normals}\label{verification}
In this section we complete the proof of Theorem \ref{mainGraphs}.
After the reduction to unit scale in Section
\ref{ss:red_unit_scale}, it remains to verify the hypotheses of
Theorem \ref{main} for $n>1$, $(G,d_{G}):=(\mathbb{H}^{n-1}\times
\R,d_{\mathbb{H}^{n-1}\times \R})$, and $(M,d_{M}) :=
(S,d_{\He})$, where $S = \Phi(\W)$, as in Theorem
\ref{mainGraphs}. To be accurate, also take $x_{0} = 0 \in G$, and
fix $p_{0} \in M$ arbitrary. We start by defining the maps $i_{w
\to p} \colon G \to S$. They will not depend on the scale index $k
\geq 0$, that is, $i_{w \to p}^{k} = i_{w \to p}$ for all $k \geq
0$, and they can also be defined for all points $w \in G$, $p \in
M$ (and not only those close to $x_{0}$ and $p_{0}$).

 To construct the  maps
$i_{w \to p} \colon G \to S$, we will first define certain
bilipschitz maps $\Psi_p:\W \to \W_p$. We know that $\W$ is
isometric to $\W_p$, so without further restrictions, this would
be an easy task, but keeping in mind \eqref{comp}, we want to make
sure that the mappings $\Psi_p$ change in a controlled way as we
let $p$ vary in $S$. It would be possible to arrange this even for
isometric $\Psi_p$, but the construction is simpler if we allow
for bilipschitz distortion, and the main ideas are contained in
the following lemma.

\begin{lemma}\label{l:Psi_A}
For $n\geq 2$ and $D:=(a_2,\ldots,a_n,c,b_2,\ldots,b_n)\in
\R^{2n-1}$, define
\begin{displaymath}
\psi_D(x_2,\ldots,x_{2n}):= cx_{n+1}+\sum_{i=2}^n (a_i x_i + b_i
x_{n+i}),
\end{displaymath}
and consider the  vertical subgroups
\begin{displaymath} \W :=
\{(x_1,\ldots,x_{2n},t)\in \He^n:\;x_1=0\}\text{ and }\W ':=
\left\{(x_1,\ldots,x_{2n},t)\in
\He^n:\;x_1=\psi_D(x_2,\ldots,x_{2n})\right\}.
\end{displaymath}
Then the map $\Psi_D(0,x_2,\ldots,x_{2n},t):=$
\begin{displaymath}
\left(\psi_D(x_2,\ldots,x_{2n}),b_2 x_{n+1}+x_2,\ldots,b_n
x_{n+1}+x_n,x_{n+1},-a_2 x_{n+1}+x_{n+2},\ldots,-a_n
x_{n+1}+x_{2n},t \right)
\end{displaymath}
has the following properties:
\begin{enumerate}
\item $\Psi_D:(\W,\cdot)\to (\W',\cdot)$ is a group isomorphism,
\item $\Psi_D:(\W,d_{\He})\to (\W',d_{\He})$ is $L_D$-bilipschitz
with $L_D$ depending continuously on $D$, \item
$d_{\He}(\Psi_D(w),\Psi_{D'}(w))\lesssim
\max\{|D-D'|,|D-D'|^{1/2}\}\,\|w\|$, for all $w\in \W$.
\end{enumerate}
\end{lemma}

\begin{proof}
We start by noting that
\begin{align*}
\psi_D(x_2,&\ldots,x_{2n})= c x_{n+1} +\sum_{i=2}^n (a_i x_i + b_i
x_{n+i} ) \\& = \psi_D(b_2x_{n+1}+x_2,\ldots,b_n
x_{n+1}+x_{n},x_{n+1},-a_2x_{n+1}+x_{n+2},\ldots,-a_n
x_{n+1}+x_{2n}),
\end{align*}
which can be used to show that $\Psi_D(\W)=\W'$. It further
follows directly from the definition that $\Psi_D$ is injective,
$\Psi_D(0)=0$, and $\Psi_D(w^{-1})=(\Psi_D(w))^{-1}$. In order to
see  that
\begin{equation}\label{eq:GroupHomo}
\Psi_D(w\cdot w') =\Psi_D(w)\cdot \Psi_D(w'),\quad w,w'\in \W,
\end{equation}
it suffices to verify the identity for the last components of the
points, as the first components agree obviously by linearity of
$\Psi_D$. For $w=(0,x_2,\ldots,x_{2n},t)$,
$w'=(0,x_2',\ldots,x_{2n}',t')$, we find that
\begin{align*}
[\Psi_D&(w\cdot w')]_{2n+1}= t+t'+\tfrac{1}{2}\sum_{i=2}^n (x_i
x_{n+i}'-x_{n+i}x_i'  ) \\
=&t+t'+ \tfrac{1}{2}\left(c x_{n+1} +\sum_{i=2}^n (a_i x_i + b_i
x_{n+i} ) \right)x_{n+1}'- \tfrac{1}{2}\left(c x'_{n+1} +\sum_{i=2}^n
(a_i x_i' +
b_i x_{n+i}' ) \right)x_{n+1}\\
&+\tfrac{1}{2} \sum_{i=2}^n \left((b_i x_{n+1}+x_i)(-a_i
x_{n+1}'+x_{n+i}')- (b_i x'_{n+1}+x_i')(-a_i x_{n+1}+x_{n+i}) \right)\\
=& [\Psi_D(w)\cdot \Psi_D(w')]_{2n+1},
\end{align*}
which shows \eqref{eq:GroupHomo} and thus completes the proof of
the first claim in the lemma. Using the group isomorphism property
and the fact that $\Psi_D$ commutes with Heisenberg dilations, we
next observe that
\begin{align*}
d_{\He}(\Psi_D(w),\Psi_D(w'))&= \|\Psi_D(w')^{-1}\cdot
\Psi_D(w)\|= \|\Psi_D(w'^{-1}\cdot w)\| \in [c_D d_{\He}(w,w'),
C_D d_{\He}(w,w')],
\end{align*}
where \begin{displaymath} c_D:= \min\{\|\Psi_D(v)\|:\;
\|v\|=1\}\quad \text{and}\quad C_D:= \max\{\|\Psi_D(v)\|:\;
\|v\|=1\}.\end{displaymath} This concludes  the second part of the
lemma, up to the continuity of $D\mapsto L_D$, which will follow
from the third part. To verify the third part, let us fix $D,D'\in
\R^{2n-1}$, and an arbitrary point $w=(0,x_2,\ldots,x_{2n},t)$ in
$\W$, and compute $\Psi_{D'}(w)^{-1}\cdot \Psi_D(w)=$
\begin{align*}
(\psi_{(D-D')}(w),(b_2-b_2')x_{n+1},\ldots,(b_n-b_{n}')x_{n+1},0,(a_2'-a_2)x_{n+1},\ldots,(a_n'-a_n)x_{n+1},
\tau),
\end{align*}
where
\begin{displaymath}
\tau:= \tfrac{1}{2}x_{n+1}\left[\psi_{(D-D')}(w)+\sum_{i=2}^n
\left((b_i-b_i')x_{n+i}+(a_i-a_i')x_i  \right) \right].
\end{displaymath}
This shows that
\begin{displaymath}
d_{\He}(\Psi_{D}(w),\Psi_{D'}(w))\lesssim
\max\{|D-D'|,|D-D'|^{1/2}\} \|w\|,
\end{displaymath}
as claimed.
\end{proof}
The mappings defined in Lemma \ref{l:Psi_A} will be used later in
the case where the components of $D$ are the entries of an
intrinsic gradient $\nabla^{\varphi}\varphi(w)$. For $p=\Phi(w)$,
we then denote
\begin{equation}\label{eq:Psi_p}
\Psi_p:=
\Psi_{(D_2^{\varphi}\varphi(w),\ldots,D_{2n}^{\varphi}\varphi(w))},
\end{equation}
so that $\Psi_p(\W)=\W_p$ is the vertical plane appearing in
Proposition \ref{approxProp}.

\begin{remark}
It is important to note that $\Psi_p$ is different from the
obvious parametrization $L_p:\W \to \W_p$ used in the proof of
Proposition \ref{approxProp}. While $L_p$ is intrinsic Lipschitz,
the map $\Psi_p$ is metrically Lipschitz and it is obtained by
precomposing $L_p$ with a map that serves as ``characteristic
straightening map'' in this setting.
\end{remark}

\begin{remark}
The proof of Theorem  \ref{mainGraphs} can be modified so that it
yields  $2$-bilipschitz maps instead of $2L'$-bilipschitz maps for
a constant $L'=L'(L)>1$. In the case $n=1$ this is due to the
third author in an earlier version of this paper, and it is based
on replacing the bilipschitz map $\Psi_p:\W\to\W_p$ in
\eqref{eq:Psi_p} by the isometry $(0,y,t)\mapsto (y e_p, t)$,
where $e_p$ represents a horizontal unit vector vector (in the
$\{X_1,X_2\}$-frame) perpendicular to the horizontal normal
$\nu_{\He}(p)$.
\end{remark}

Returning to the proof of Theorem \ref{mainGraphs}, we proceed to
construct the mappings $i_{w\to p}:G\to \Phi(\W)$ for $w\in G$,
$p\in S=\Phi(\W)$. Since $G=\He^{n-1}\times \R$ (as in
\eqref{eq:HeisxR}) is isometrically isomorphic to $\W$ via the map
\begin{displaymath}
F:((z_1,\ldots,z_{2n-2},t),s)\mapsto
(0,z_1,\ldots,z_{n-1},s,z_n,\ldots,z_{2n-2},t)
\end{displaymath}
the  idea for the construction of  $i_{w\to p}(v)$ is informally
the following: identify $w^{-1}\cdot_G v \in G$ with the point
$F(w^{-1}\cdot_G v)\in \W$, then map this point to the vertical
plane $\W_p$ by means of the bilipschitz map $\Psi_p$ from
\eqref{eq:Psi_p}, left translate by the point $p$, and finally let
$i_{w\to p}(v)$  be a point in $S$ of minimal distance from
$p\cdot \Psi_p(F(w^{-1}\cdot_G v))\in p\cdot \W_p$, keeping in
mind Proposition \ref{approxProp}. Such a point may not be unique,
but this does not matter as long as the choice is made depending
only on the product $w^{-1}\cdot_G v$, and not on the points $v$
and $w$ individually.

We now explain the construction in detail. First, if $u\in G$, let
$$q := q[p,u] \in S$$ be any point satisfying
\begin{equation}\label{form31} d_{\He}\left(p \cdot \Psi_p(F(u)),q\right)
 = \dist_{\He}\left(p \cdot \Psi_p(F(u)),S\right). \end{equation}
Then, if $v,w \in {G}$ and $p \in S$, let
\begin{equation}\label{eq:i_1}i_{w \to p}(v) :=
q[p,w^{-1}\cdot_G v].\end{equation} The definition implies that if
$w,w',v,v' \in {G}$ with $w^{-1}\cdot_G v = (w')^{-1}\cdot_G v'$,
then
\begin{equation}\label{form30} i_{w \to p}(v) = i_{w' \to p}(v'). \end{equation}
To simplify notation in the sequel, we define $\Tan^{w}_{p} \colon
({G},d_{{G}}) \to (\W_{p},d_{\mathbb{H}})$ to be the map given by
\begin{equation}\label{eq:unitary_Tan}
\Tan^w_p(v) = \Psi_p (F(w^{-1}\cdot_G v)),\quad v\in G.
\end{equation}
It follows from Lemma \ref{l:Psi_A} that $\Tan^{w}_{p} \colon
({G},d_{{G}}) \to (\W_{p},d_{\mathbb{H}})$ is a bilipschitz map
with bilipschitz constant bounded in terms of
$\|\nabla^{\varphi}\varphi\|_{L^{\infty}(\W)}$. Evidently $i_{w
\to p}(w) = p \cdot \Tan_{p}^{w}(w) = p$. Also note that the
isomorphism property of $\Psi_p\circ F$ implies the following
"chain rule":
\begin{equation}\label{chain} \Tan^{w_{1}}_{p}(w_{3}) = \Tan^{w_{1}}_{p}(w_{2}) \cdot \Tan^{w_{2}}_{p}(w_{3}), \qquad w_{1},w_{2},w_{3} \in
G, \: p \in S. \end{equation} Since $p \cdot \Tan_{p}^w(v) \in p
\cdot \W_{p}$, one infers from \eqref{form20} and the definition
of $i_{w \to p}(v)$ that
\begin{align} d_{\He}(p \cdot \Tan^{w}_{p}(v),i_{w \to p}(v)) & = \dist_{\He}(p \cdot \Tan^{w}_{p}(v),S) \notag\\
&\label{form18} \leq Ad_{\He}(p \cdot \Tan_{p}^{w}(v),p)^{1 +
\alpha} \lesssim_L Ad_{G}(w,v)^{1 + \alpha}. \end{align} Using
this estimate,
 one has
\begin{align} |d_{\He} & (i_{w \to p}(v),i_{w \to p}(v')) - d_{\He}(p\cdot\Tan^{w}_{p}(v),p\cdot\Tan^{w}_{p}(v'))| \notag\\
& \leq d_{\He}(p \cdot \Tan^{w}_{p}(v),i_{w \to p}(v)) + d_{\He}(p \cdot \Tan^{w}_{p}(v'),i_{w \to p}(v')) \notag\\
&\label{form25} \lesssim_L A\max\{d_{G}(w,v)^{1 +
\alpha},d_{G}(w,v')^{1 + \alpha}\}, \qquad v,v' \in G.
\end{align}
Moreover
\begin{displaymath}
d_{\He}(p\cdot\Tan^{w}_{p}(v),p\cdot\Tan^{w}_{p}(v')) =
d_{\He}(\Psi_p(F(w^{-1}\cdot_G v)),\Psi_p(F(w^{-1}\cdot_G v'))),
\end{displaymath}
and since $\Psi_p \circ F$ is bilipschitz with a constant that
depends only on $\|\nabla^{\varphi}\varphi\|_{L^{\infty}(\W)}$,
condition \eqref{ISO} follows with the help of \eqref{form25}.

It remains to check condition \eqref{comp}. Using the homogeneity
of $G$  (see the discussion around \eqref{form27} for further
details), it suffices to verify the case "$x = 0$" of condition
\eqref{comp}: if $w \in G$ with $\|w\| \leq 2$, $p \in S$, and
$i_{0 \to p}(w) = q \in S$, then
\begin{equation}\label{form24} d_{\He}(i_{0 \to p}(v),i_{w \to q}(v)) \lesssim_{L} AH \max\{\|w\|^{1 + \alpha/2},\|v\|^{1 + \alpha/2},
d_{G}(v,w)^{1 + \alpha/2}\} \end{equation} for all $v \in G$ with
$\|v\|\leq 1$. (In particular, it may be interesting to note that the
$C^{1,\alpha}_{\He}$-hypothesis only gives the condition
\eqref{comp} with exponent $\alpha/2$.) To estimate the left hand
side of \eqref{form24}, the strategy will be to first obtain a
corresponding estimate for
\begin{displaymath} d_{\He}(p \cdot \Tan^{0}_{p}(v), q \cdot \Tan^{w}_{q}(v)), \end{displaymath}
and eventually conclude \eqref{form24} from this bound combined
with \eqref{form18}. Consider $v,w \in G$ with $\|w\| \leq 2$.
Start by applying the "chain rule" \eqref{chain}, and the triangle
inequality, as follows:
\begin{align*} d_{\He} (p \cdot \Tan^{0}_{p}(v), q \cdot \Tan^{w}_{q}(v))
& = d_{\He}([p \cdot \Tan^{0}_{p}(w)] \cdot \Tan^{w}_{p}(v), q \cdot \Tan^{w}_{q}(v))\\
& \leq d_{\He}([p \cdot \Tan^{0}_{p}(w)] \cdot \Tan^{w}_{p}(v), [p \cdot \Tan^{0}_{p}(w)] \cdot \Tan^{w}_{q}(v))\\
& \quad + d_{\He}([p \cdot \Tan^{0}_{p}(w)] \cdot
\Tan^{w}_{q}(v),q \cdot \Tan^{w}_{q}(v)) =: I_{1} + I_{2}.
\end{align*} To estimate $I_{1}$, note that by left-invariance
\begin{displaymath} I_{1} = d_{\He}(\Tan^{w}_{p}(v),\Tan^{w}_{q}(v)) =
d_{\mathbb{H}}(\Psi_p(F(w^{-1}\cdot_G v)),\Psi_q(F(w^{-1}\cdot_G
v))).
\end{displaymath}
Then \eqref{eq:Psi_p} and the third part of Lemma \ref{l:Psi_A}
imply  that
\begin{displaymath}
I_1=d_{\mathbb{H}}(\Psi_p(F(w^{-1}\cdot_G
v)),\Psi_q(F(w^{-1}\cdot_G v)))\lesssim_L
|\nabla^{\varphi}\varphi(\pi_{\W}(p))-\nabla^{\varphi}\varphi(\pi_{\W}(q))|^{1/2}
d_{G}(v,w).
\end{displaymath}
To proceed, we note that $p= a \cdot \varphi(a)$ and $q=b\cdot
\varphi(b)$ satisfy
\begin{align*}
|\nabla^{\varphi}\varphi(b)-\nabla^{\varphi}\varphi(a)|&=|\nabla^{\varphi}\varphi(\pi_{\W}(p\cdot
\varphi(a)^{-1}\cdot a^{-1}
\cdot b \cdot\varphi(a)))-\nabla^{\varphi}\varphi(\pi_{\W}(p))|\\
&=|\nabla^{\varphi^{(p^{-1})}}\varphi^{(p^{-1})}(\varphi(a)^{-1}\cdot a^{-1}\cdot b \cdot\varphi(a))) - \nabla^{\varphi^{(p^{-1})}}\varphi^{(p^{-1})}(0)|\\
&\overset{\eqref{CFOdef}}{\leq} H \|\varphi(a)^{-1}\cdot a^{-1}\cdot b \cdot\varphi(a)\|^{\alpha}\\
&=  H \|\varphi(a)^{-1}\cdot a^{-1}\cdot b \cdot\varphi(b) \cdot\varphi(b)^{-1} \cdot\varphi(a)\|^{\alpha}\\
&\leq 2 H d_{\mathbb{H}}(p,q)^{\alpha},
\end{align*}
using Lemma \ref{invintrgrad} when passing to the second line. It
follows for $p\in S$ and $q=i_{0\to p}(w)$ with $\|w\|\leq 2$ that
\begin{align}\label{form39}  |\nabla^{\varphi}\varphi(\pi_{\W}(p))-\nabla^{\varphi}\varphi(\pi_{\W}(q))| \leq
2 H d_{\mathbb{H}}(p,q)^{\alpha} & = 2 Hd_{\He}(i_{0 \to p}(0),i_{0 \to p}(w))^{\alpha}\\
& \stackrel{\eqref{form25}}{\lesssim}_{L} H[\|w\|^{\alpha} +
A\|w\|^{\alpha}] \lesssim_L AH\|w\|^{\alpha}. \notag\end{align}
One  needed here the assumption $\|w\| \leq 2$ since
\eqref{form25} initially gives a term of the form $\|w\|^{1 +
\alpha}$. The estimate above implies that
\begin{displaymath}
I_1 \lesssim_L
|\nabla^{\varphi}\varphi(\pi_{\W}(p))-\nabla^{\varphi}\varphi(\pi_{\W}(q))|^{1/2}
d_{G}(v,w) \lesssim_{L} A H\|w\|^{\alpha/2} d_G(v,w).
\end{displaymath}
Thus, the term $I_{1}$ is bounded from above by the right hand
side of \eqref{form24}.

The term $I_{2}$ has the form $I_{2} = \|\mathfrak{b}^{-1}
\cdot\mathfrak{ a} \cdot \mathfrak{b}\|$ with
\begin{displaymath} \mathfrak{a} = q^{-1} \cdot p \cdot \Tan^{0}_{p}(w) \quad \text{and} \quad \mathfrak{b} = \Tan^{w}_{q}(v). \end{displaymath}
Note that $\|\mathfrak{a}\| = d_{\He}(q,p \cdot \Tan^{0}_{p}(w))
\lesssim_L A\|w\|^{1 + \alpha}$ by \eqref{form18} (this is the
place where the relation $q = i_{0 \to p}(w)$ is used), whereas
$\|\mathfrak{b}\| \sim_L d_{G}(w,v)$. Now, writing $\mathfrak{a} =
(x_{a},t_{a})$ and $\mathfrak{b} = (x_{b},t_{b})$, one can easily
compute that
\begin{equation}\label{eq:FundCommRel} \mathfrak{b}^{-1} \cdot \mathfrak{a} \cdot \mathfrak{b}
 =\mathfrak{ a} \cdot \left(0,0,\sum_{i=1}^n \left( x_{a,i}x_{b,n+i}-x_{b,i}x_{a,n+i} \right) \right)=:\mathfrak{a}\cdot (0,0,\omega(x_a,x_b)).
 \end{equation}
(This is just the fundamental "commutator relation" in $\He^n$.)
Consequently,
\begin{align*} I_{2} & \leq \|\mathfrak{a}\| + \sqrt{|{\omega(x_a,x_b)}|} \lesssim_L A\|w\|^{1 + \alpha} + \sqrt{\|\mathfrak{a}\|\|\mathfrak{b}\|}\\
& \lesssim_L A\|w\|^{1 + \alpha} + A^{1/2}\|w\|^{1/2 + \alpha/2}d_{G}(w,v)^{1/2}\\
& \lesssim_L A\max\{\|w\|^{1 + \alpha/2},d_{G}(w,v)^{1 +
\alpha/2}\}.
\end{align*} This shows that also $I_{2}$ is bounded by the right
hand side of \eqref{form24}. Glancing again at the estimates for
$I_{1}$ and $I_{2}$, one sees that
\begin{equation}\label{form26} d_{\He}(p \cdot \Tan^{0}_{p}(v), q \cdot \Tan^{w}_{q}(v))
\lesssim_{L} AH\max\{\|w\|^{1 + \alpha/2},d_{G}(v,w)^{1 +
\alpha/2}\}, \end{equation} which is even a bit better than
\eqref{form24}. The estimate \eqref{form24} now follows from the
triangle inequality:
\begin{align*} d_{\He}(i_{0 \to p}(v),i_{w \to q}(v)) & \leq d_{\He}(i_{0 \to p}(v),p \cdot \Tan^{0}_{p}(v))\\
&\quad + d_{\He}(p \cdot \Tan^{0}_{p}(v), q \cdot \Tan^{w}_{q}(v))\\
&\quad + d_{\He}(i_{w \to q}(v),q \cdot \Tan^{w}_{q}(v)).
\end{align*} The middle term here is controlled by \eqref{form26},
and the first and third terms are controlled by \eqref{form18},
recalling the bounds for $\|v\|\leq 1$ and $\|w\|\leq 2$, which
ensure that we can replace "$\alpha$" by "$\alpha/2$" in
\eqref{form18}. This concludes the proof of \eqref{form24}.

Finally, we address the point left open above, that \eqref{form24}
looks slightly less general than \eqref{comp}. To check
\eqref{comp} properly, we need to fix $w_{1},w_{2} \in B_{G}(0,1)$
and $p,q \in S$ with $i_{w_{1} \to p}(w_{2}) = q$, and verify that
\begin{equation}\label{form27} d_{\He}(i_{w_{1} \to p}(w_{3}),i_{w_{2} \to q}(w_{3})) \lesssim
\max\{d_{G}(w_{1},w_{2}),d_{G}(w_{1},w_{3})),d_{G}(w_{2},w_{3})\}^{1
+ \alpha/2} \end{equation} for all $w_{3} \in G$ with
$d_G(w_1,w_3)\leq 1$. However, set $w := w_1^{-1}\cdot_G w_{2}$
and $v := w_1^{-1}\cdot_G w_{3}$, and observe that
\begin{displaymath}w_{2}^{-1}\cdot_G w_3 = w^{-1}\cdot_G v\quad \text{and} \quad w_{1}^{-1}\cdot_G w_3 =0^{-1}\cdot_G v. \end{displaymath}
These relations, and \eqref{form30}, show that
\begin{displaymath} i_{w_{1} \to p}(w_{3}) = i_{0 \to p}(v) \quad \text{and} \quad i_{w_{2} \to q}(w_{3}) = i_{w \to q}(v). \end{displaymath}
Thus, \eqref{form27} follows from \eqref{form24} applied with $w$
and $v$, as above.

This completes the verification of the hypotheses of Theorem
\ref{main}, and hence the proof of Theorem \ref{mainGraphs}.


\section{Lipschitz flags and extra vertical H\"older
regularity}\label{s:LipFlag}

This section is devoted to the first Heisenberg group, $\He^1$.
For convenience we use coordinates $(x,y,t)$, instead of
$(x_1,x_2,t)$,
 to denote points in $\He^1$. As
usual, we may identify $\mathbb{W}$ with $\mathbb{R}^2$ by mapping
$(0,y,t)$ to $(y,t)$, and we identify $(x,0,0)\in\mathbb{V}$ with
$x\in \mathbb{R}$.

 \begin{definition}\label{d:ExtraVerticalHolder}
We say that an intrinsic Lipschitz function $\varphi:\W \to \V$
has \emph{extra vertical H\"older regularity with constants
 $0<\alpha\leq 1$ and $H>0$} if
\begin{equation}\label{eq:VertHolDef}
|\varphi(y,t)-\varphi(y,t')|\leq H |t-t'|^{\frac{1+\alpha}{2}},
\end{equation}
for all $y,t,t'\in \mathbb{R}$.
 \end{definition}

Intrinsic Lipschitz functions are always $1/2$-H\"older continuous
with respect to the Euclidean metric along vertical lines.
Condition \eqref{eq:VertHolDef} constitutes an amount of extra
regularity at small scales which is not implied by the intrinsic
Lipschitz condition alone, see for instance \cite[Example
1.3]{MR3400438}.

 \begin{remark}\label{r:VertHolDilateTranslate}
The definition is left invariant: if $\varphi$ has extra vertical
H\"older regularity with constants $\alpha$ and $H$, then for
every $p\in \He^1$, the function $\varphi^{(p^{-1})}$ whose
intrinsic graph is $p^{-1}\cdot \Phi(\W)$ also has extra vertical
H\"older regularity with the same constants. Moreover, for $r>0$,
the function $\delta_\frac{1}{r} \circ \varphi \circ \delta_{r}$,
whose intrinsic graph is $\delta_{\frac{1}{r}}(\Phi(\W))$, has
extra vertical H\"older regularity with constants $\alpha$ and
$Hr^{\alpha}$. So condition \eqref{eq:VertHolDef} improves under
``zooming in''.
\end{remark}

\begin{remark}\label{r:small_scale} An intrinsic Lipschitz
function \textbf{with compact support} has extra vertical H\"older
regularity with constants $0<\alpha\leq 1$ and $H>0$ if and only
if there is $H'>0$ such that
\begin{equation}\label{eq:all_Scale}
|\varphi(y,t)-\varphi(y,t')|\leq \left\{\begin{array}{ll}H'
|t-t'|^{\frac{1+\alpha}{2}},&\text{if }|t-t'|\leq 1,\\H'
|t-t'|^{\frac{1-\alpha}{2}},&\text{if }|t-t'|\geq
1,\end{array}\right.,\quad y,t,t'\in \R,
\end{equation}
that is, $\varphi$ has extra vertical H\"older regularity in the
sense of \cite[Theorem 5.1]{fssler2018riesz}.
\end{remark}

Before studying in more detail the intrinsic graphs of functions
that satisfy the conditions in Definition
\ref{d:ExtraVerticalHolder}, we give two examples of such functions.

\begin{ex}\label{ex:EuclC1} Under the identification  $\mathbb{W}\triangleq \mathbb{R}^2$ and $\mathbb{V}\triangleq \mathbb{R}$ described
before Definition \ref{d:ExtraVerticalHolder}, every compactly
supported Euclidean Lipschitz function $\varphi:\mathbb{R}^2
\to \mathbb{R}$ is an intrinsic Lipschitz function that satisfies
the extra vertical H\"older regularity condition in Definition
\ref{d:ExtraVerticalHolder} with $\alpha=1$.
\end{ex}

\begin{ex}\label{ex:intrC1alpha}
Let $0<\alpha\leq 1$, $\W,\V\subset \He^1$ as above, and let
$\varphi:\W \to \V$ be a compactly supported
$C_{\He}^{1,\alpha}(\W)$ function. Since $\spt \varphi$ is
compact, $\nabla^{\varphi}\varphi$ is continuous and compactly
supported, hence $L :=
\|\nabla^{\varphi}\varphi\|_{L^{\infty}(\W)} < \infty$. According
to \cite[Lemma 2.22]{CFO2}, this implies  that $\varphi$ is
intrinsic Lipschitz. The extra vertical H\"older regularity
condition  follows from \cite[Proposition 4.2]{CFO2}, keeping in
mind the characterization of compactly supported
$C_{\He}^{1,\alpha}(\W)$ functions stated in Proposition
\ref{equivProp}.
\end{ex}

Intrinsic graphs of intrinsic Lipschitz functions with extra
vertical H\"older regularity turn out to be well approximable at
all points and small scales by intrinsic Lipschitz graphs of a
special form, the \emph{Lipschitz flags}; see Proposition
\ref{p:contenutesSpPsi} for the precise statement. In the proof of
Theorem \ref{mainIntroVertical}, Lipschitz flags will play an
analogous role as vertical tangent planes did in the proof of
Theorem \ref{mainIntro}, so we start with some observations of
general nature that serve as a counterpart for Section
\ref{s:DefPrelim}.

\subsection{Approximation by Lipschitz flags}

\begin{definition}\label{d:FlagSurface}
We say that $F\subset \He^1$ is a \emph{Lipschitz flag} if there
exists a Euclidean Lipschitz map $\psi :\R \to \R$ such that $F$
is the intrinsic graph $\Phi(\W)$ of the map $\varphi : \W \to \V$
defined by
\begin{equation}\label{linkvarphipsi}
\varphi (y,t ) = \psi (y).
\end{equation}
\end{definition}

Lipschitz flags are bilipschitz equivalent to the parabolic plane.
This observation appeared already in \cite[Lemma
7.5]{fssler2019singular}, but we include the proof for completeness.

\begin{lemma}\label{flagimpliesbilip} If $F\subset \He^1$ is a
Lipschitz flag given by an $L$-Lipschitz function $\psi :\R\to
\R$, then there is a $\sim (1+L)$-bilipschitz map
\begin{equation*}
\Psi _F: (\W, d_\He) \to (F, d_\He).
\end{equation*}
\end{lemma}

\begin{proof}
Let $F$ be a Lipschitz flag, so $F$ is the intrinsic graph
$\Phi(\W)$ of the map $\varphi : \W \to \V$ defined as in
\eqref{linkvarphipsi} for the Euclidean L-Lipschitz function $\psi
:\R \to \R$. We will show that the map $\Psi _F: (\W, d_\He) \to
(F, d_\He)$ given by
\begin{equation}\label{eq:psi_F}
\begin{aligned}
\Psi _F (y,t) : = \left(\psi (y), y, t- \frac 1 2 y \psi (y) + \int_0^y \psi (\eta ) \, d \eta \right)\\
\end{aligned}
\end{equation}
is the $\sim (1+L)$-bilipschitz map which we are looking for.
Firstly, we observe that
\begin{equation*}
\Psi _F (y,t)  = \Phi \left(y, t + \int_0^y \psi (\eta ) \, d \eta \right),\\
\end{equation*}
where $\Phi$ is the graph map\footnote{The map $\Psi_F$ is the
composition of the graph map $\Phi$ with the ``characteristic
straightening map'' mentioned in the introduction. First, the
characteristic straightening map sends horizontal lines in
$\mathbb{W}$ to integral curves of $\nabla^{\varphi}= \partial_y +
\psi(y)\partial_t$, and these integral curves are then mapped by
$\Phi$ to the horizontal curves that foliate the flag $F$.} of
$\varphi$, hence $\Psi_F(\W)=\Phi(\W)$. Moreover, since $\psi$ is
an $L$-Lipschitz function, we get
\begin{equation}\label{bilipproof}
\begin{aligned}
d_\He & (\Psi _F (y,t), \Psi _F (y',t') )\\& = \left\| \left(\psi (y')-\psi (y),y'- y, t'- t+ \tfrac{ 1}{ 2} (y-y' )(\psi (y')+\psi (y)) + \int_{y}^{y'} \psi (\eta ) \, d \eta \right) \right\|\\
& = \left\| \left(\psi (y')-\psi (y),y'- y, t'- t + \int_{y}^{y'} \left(\tfrac{\psi (\eta ) -\psi (y)}{2} \right)+ \left(\tfrac{\psi (\eta ) -\psi (y')}{2} \right) \, d \eta \right) \right\|\\
& \lesssim (1+L) \left(|y'-y | + \sqrt{ |t'-t|}\right),
\end{aligned}
\end{equation}
for all $(y,t), (y',t') \in \W$. On the other hand, since
\begin{equation*}
\begin{aligned}
|t'-t| \leq \left| t'- t + \int_{y}^{y'} \left(\tfrac{\psi (\eta )
-\psi (y)}{2} \right)+ \left(\tfrac{\psi (\eta ) -\psi (y')}{2}
\right) \, d \eta \right| + \left| \int_{y}^{y'} \left(\tfrac{\psi
(\eta ) -\psi (y)}{2} \right)+ \left(\tfrac{\psi (\eta ) -\psi
(y')}{2} \right) \, d \eta \right|,
\end{aligned}
\end{equation*}
it follows
\begin{equation*}
\begin{aligned}
d_\He ( (y,t), (y',t'))  \lesssim |y-y'| + |t-t'|^{1/2} \lesssim
(1+L) d_{\He}(\Psi _F (y,t), \Psi _F (y',t') ),\end{aligned}
\end{equation*}
for all $(y,t), (y',t') \in \W$. Hence, putting together
\eqref{bilipproof} and the last inequality, we get that $\Psi _F$
is a $\sim(1+L)$-bilipschitz map, as desired.
\end{proof}

Given an intrinsic Lipschitz graph $S\subset \He^1$, we will
define for each $p\in S$ a Lipschitz flag that intersects
$S$ in a Lipschitz curve passing through $p$. We start with some
general definitions that will be used throughout this section. We
state them in terms of the intrinsic Lipschitz function
$\varphi^{(p^{-1})}$, whose intrinsic graph is $p^{-1}\cdot S$,
recall \eqref{form34} .

\begin{definition}\label{d:mappsi_p} Let $S=\Phi(\W)=\{ w\cdot \varphi (w)\,:\, w\in \W \}$ be the intrinsic graph of an intrinsic $L$-Lipschitz function $\varphi : \W \to \V$. To each point $p\in S$, we associate the function
\begin{equation}\label{defipsi_P}
\begin{aligned}
\psi _p :\R & \longrightarrow \R\\
 s & \longmapsto \varphi ^{(p^{-1})} (s, \tau _p(s)),
\end{aligned}
\end{equation}
where $\tau_p :\R \to \R $ is some solution of the Cauchy problem
 \begin{equation}\label{Cauchy problem}
\left\{
\begin{array}{lcl}
\dot{\tau}_p(s)= \varphi ^{(p^{-1})} ( s, \tau _p(s)), \quad \mbox{for } s\in \R  \\
\tau _p (0)=0. \\
\end{array}
\right.
\end{equation}
\end{definition}

Note that $\varphi  ^{(p^{-1})} $ is intrinsic Lipschitz with the
same constant as $\varphi$, and then the global existence of
$\tau_p $ follows as in \cite[(6.27)]{antonelli2019pauls}.
Proposition 6.10 in \cite{antonelli2019pauls} is only for higher
dimensions, but this part of the argument works also for our
setting $\He^1$. Moreover, using the same proof as for
\cite[(6.30)]{antonelli2019pauls}, it follows that $\tau _p$
satisfies the inequality
\begin{equation}\label{(6.31)new}
|\tau _p (s) | \lesssim _L |s|^2, \qquad s\in \R.
\end{equation}
  Notice further that $\gamma_p : s \mapsto (0,s,\tau_p(s))$ is a $C^1$ curve with
 \begin{equation*}
\dot \gamma_p (s)= \begin{pmatrix} 0\\ 1\\ \varphi
^{(p^{-1})}(\gamma_p (s))
\end{pmatrix}= D_2^{\varphi^{(p^{-1})}}|_{\gamma_p(s)}, \end{equation*}
where we recall that $D^{\varphi}_{2}$ is the vector field $D^{\varphi}_{2} = \partial_{y} + \varphi \partial_{t}$. As a consequence, from  \cite[Proposition 6.6]{antonelli2019pauls} it also follows that
 $ s  \mapsto \psi_p(s)=\varphi ^{(p^{-1})} (\gamma_p (s))$ is Euclidean Lipschitz with Lipschitz constant depending only
on the intrinsic Lipschitz constant of $\varphi$. The solution $\tau_p$ may not be unique, but
we never need other properties of it than the ones described
above, so any choice will do. For completeness, we also mention
that $\Phi^{(p^{-1})}\circ \gamma_p$ is a Lipschitz curve in
$(\He^1,d_{\He})$ by \cite[Theorem
4.2.16]{kozhevnikov:tel-01178864}.

\begin{definition}\label{d:flagsurfaceF_p}
Let $S= \Phi(\W)$ be the intrinsic graph of an intrinsic Lipschitz function $\varphi : \W \to \V$. For each point $p\in S$, we define
\begin{equation}\label{flagsurfaceF_p}
F_p:= \left\{ \left(0,y,t \right) \cdot \left(\psi_p(y) ,0,0
\right)\,:\, (y,t) \in \R^2 \right\}.
\end{equation}
\end{definition}

We note the following properties of $F_p$ defined as in
\eqref{flagsurfaceF_p}:
\begin{enumerate}
\item $F_p$ is a Lipschitz flag, \item $0\in F_p$, \item $p\cdot
F_p$ is also a Lipschitz flag, \item $F_p \cap (p^{-1}\cdot
S)\supseteq \Phi^{(p^{-1})}( \gamma_p(\R))$.
\end{enumerate}

Item (1) follows immediately from the fact that
$\psi_p$ is a Euclidean Lipschitz function. The item (2) follows
since $p\in S$, and so $\psi_p(0)=\varphi ^{(p^{-1})}(0,0)=0$.
Next, (3) follows by computing explicitly that
\begin{displaymath} p \cdot F_{p} = \{(0,y,t) \cdot (\psi(y),0,0) : (y,t) \in \R^{2}\}, \end{displaymath}
 where $\psi(y) := \psi_{p}(y - y_{p}) + x_{p}$ and $p = (x_{p},y_{p},t_{p})$. Since $\psi_p:\mathbb{R}\to\mathbb{R}$ is Lipschitz, so is $\psi$. Finally, (4) is clear from the
definitions since $\psi_p(y)=\varphi^{(p^{-1})}(y,\tau_p(y))$ and
$F_p$ is of the form \eqref{flagsurfaceF_p}.

\begin{definition}\label{d:Psi_p}
Let $S= \Phi(\W) $ be the intrinsic graph of an intrinsic
Lipschitz function $\varphi : \W \to \V$. For each point $p\in S$,
we let $\Psi_p:=\Psi_{F_p}:\W \to F_p$ be the map given by the
formula \eqref{eq:psi_F} applied to the Lipschitz flag $F=F_p$
from \eqref{flagsurfaceF_p}, that is,
\begin{equation}\label{defiPsi_p}
\begin{aligned}
\Psi_p(y,t)=  \left(\varphi ^{(p^{-1})} ( y, \tau _p(y)) , y, t-
\frac 1 2 y \varphi ^{(p^{-1})} ( y, \tau _p(y)) + \int_0^y
\varphi ^{(p^{-1})} ( \eta , \tau _p(\eta )) \, d \eta \right).
\end{aligned}
\end{equation}
\end{definition}

The reader may think of $\Psi_p$ as a surrogate
 for the bilipschitz map defined in
 \eqref{eq:Psi_p}, which sends $\mathbb{W}$ to the vertical plane
 $\mathbb{W}_p$. In this analogy,
the Lipschitz flag $F_p$ plays the role of $\mathbb{W}_p$. To
emphasise this conceptual similarity, we decided to use again the
symbol ``$\Psi_p$'' in Definition \ref{d:Psi_p}. The analogy is
however not perfect: if $\varphi^{(p^{-1})}$ is not intrinsic
linear, the map $\Psi_p$ from Definition \ref{d:Psi_p} is not a
group homomorphism and hence we lack a counterpart for the chain
rule \eqref{chain}, which we proved for the map defined in
\eqref{eq:unitary_Tan}.

The following properties of $\Psi_{F_p}=\Psi_p$ defined as in
\eqref{defiPsi_p} will be used:
\begin{enumerate}
\item $\Psi _p(0)=0$ since $\varphi ^{(p^{-1})} (0,0)=0$, \item
$\Psi_p (y,t)= \Phi ^{(p^{-1})} (y,\tau_p(y)) \cdot (0,0,t)$, as
$\int_0^y \varphi ^{(p^{-1})} ( \eta , \tau _p(\eta )) \, d \eta=
\int_0^y  \dot{\tau}_p(\eta)\,d\eta = \tau_p(y)$, \item $\Psi_p$
is bilipschitz with a constant that depends only on the intrinsic
Lipschitz constant of $\varphi$ (by Lemma \ref{flagimpliesbilip}
and the paragraph before Definition \ref{d:Psi_p}).
\end{enumerate}

\begin{remark}\label{r:WhatItMeansForFlags}
If $S= \Phi(\W) $ is itself a Lipschitz flag, that is, the
intrinsic graph of an intrinsic Lipschitz
function $\varphi : \W \to \V$ that does not depend on the
$t$-variable, then
\begin{equation}\label{eq:FlagParam}
\Psi_p(y,t)= \Phi^{(p^{-1})}\left(y,t+ \int_0^y
\varphi^{(p^{-1})}(\eta,\tau_p(\eta))\,d\eta\right),
\end{equation}
and in particular, $\Psi_p(\W)= \Phi^{(p^{-1})}(\W)$ and  hence
$S= p \cdot \Psi_p(\W)$ in that case. Also note that here
$\varphi^{(p^{-1})}$ does not depend on the $t$-variable, and so
the integral in \eqref{eq:FlagParam} can be written without the
dependence on $\tau_p$.
\end{remark}

As we noted below Definition \ref{d:flagsurfaceF_p}, the surfaces
$F_p=\Psi_p(\W)$ and $p^{-1}\cdot S = \Phi^{(p^{-1})}(\W)$
intersect at least along a curve. The next lemma shows that they
approximate each other well also in a neighborhood of that curve
if $\varphi$ has extra vertical H\"older regularity.

\begin{lemma}\label{betterProp3.36}
Let $\varphi : \W \to \V $ be an intrinsic Lipschitz function with
extra vertical H\"older regularity with constants $0<\alpha\leq 1$
and $H>0$. Then
\begin{displaymath}
d_\He ( \Psi_p (y,t), \Phi ^{(p^{-1})} (y,\tau_p(y) +t ) ) \leq H
|t|^{(1+\alpha)/2} \leq H\|(y,t)\|^{1 + \alpha}, \qquad \, (y,t) \in \R^2.
\end{displaymath}

\end{lemma}

\begin{remark}\label{r:largeScale}From the proof of Lemma \ref{betterProp3.36} one can
 infer that if $\varphi$ has extra vertical H\"older regularity in the stronger sense of \eqref{eq:all_Scale}, then
\begin{equation*}
d_\He ( \Psi_p (y,t), \Phi ^{(p^{-1})} (y,\tau_p(y) +t ) ) \leq
H'\min\{ |t|^{(1+\alpha)/2}, |t|^{(1-\alpha)/2}\}, \quad \mbox{
for }\, (y,t) \in \R^2.
\end{equation*}
\end{remark}

\begin{proof}[Proof of Lemma \ref{betterProp3.36}]
Recall that $\tau_p$ is a solution of the Cauchy problem
\eqref{Cauchy problem},  and we have that
\begin{equation*}
\begin{aligned}
\Psi_p (y,t) =\Phi ^{(p^{-1})} (y,\tau_p(y)) \cdot (0,0,t)
\end{aligned}
\end{equation*}
for all  $(y,t) \in \R^2$. As a consequence,
\begin{equation*}
\begin{aligned}
& d_\He ( \Psi_p (y,t), \Phi ^{(p^{-1})} (y,\tau_p(y) +t ) ) = d_\He ( \Phi ^{(p^{-1})} (y,\tau_p(y)) \cdot (0,0,t) , \Phi ^{(p^{-1})} (y,\tau_p(y) +t ) ) \\
& = \|\varphi ^{(p^{-1})} ( y, \tau _p(y) +t )^{-1}\cdot  (0, y, \tau _p(y) +t )^{-1} \cdot  (0, y, \tau _p(y)) \cdot  \varphi ^{(p^{-1})} ( y, \tau _p(y)) \cdot (0,0,t)  \|\\
& = | \varphi ^{(p^{-1})} ( y, \tau _p(y) + t ) -  \varphi
^{(p^{-1})} ( y, \tau _p(y))| \leq H |t|^{\frac{1+\alpha}{2}},
\end{aligned}
\end{equation*}
as claimed.
\end{proof}

In the following we denote by $[A]_{\delta}^{\He^1}$ the
$\delta$-neighbourhood of a set $A\subset \He^1$ with
respect to $d_{\He}$, and $[B]^{\mathbb{R}^2}_{\delta}$ stands for
the Euclidean $\delta$-neighbourhood of a set $B \subset
\mathbb{R}^2$. A direct corollary of Lemma \ref{betterProp3.36} is
the following result:

\begin{proposition}\label{p:contenutesSpPsi}
 Let $\varphi : \W \to \V $ be an intrinsic $L$-Lipschitz function with extra vertical H\"older regularity with constants $0<\alpha\leq 1$
 and $H>0$.
Then, there is $c=c(L)\geq 1$ such that for all $r>0$ and all
$p\in S=\Phi(\W)$ it follows
 \begin{enumerate}
\item $S \cap  B(p,r) \subset [p\cdot \Psi
_p(\W)]_{c\delta}^{\He^1}$, \item $\left(p\cdot \Psi _p(\W)
\right)\cap B(p,r) \subset [S]_{c\delta}^{\He^1}$,
\end{enumerate}
where $\delta := \delta(H,r) := Hr^{1+\alpha}$.
\end{proposition}

\begin{remark}
Continuing Remark \ref{r:largeScale}, we note that if an intrinsic
$L$-Lipschitz function $\varphi$ has extra vertical H\"older
regularity with constants $0<\alpha\leq 1$
and $H'>0$ in the stronger sense \eqref{eq:all_Scale}, then the conclusion of
Proposition \eqref{p:contenutesSpPsi} can be improved by replacing
"$\delta$" with $H'\min\{r^{1+\alpha},r^{1-\alpha}\}$. In other
words, the intrinsic graph of $\varphi$ is well approximated by
Lipschitz flags also at large scales.
\end{remark}

\begin{proof}[Proof of Proposition \ref{p:contenutesSpPsi}]
Fix $r>0$ and $p\in S$ as in the assumptions of the proposition.
Since the metric $d_\He$ is left invariant, it  is sufficient to
show that there is $c=c(L)\geq 1$ such that
\begin{enumerate}
\item[(i)]  $\left( p^{-1} \cdot S\right) \cap  B(0,r) \subset [
\Psi _p(\W)]_{c\delta }^{\He^1}$, \item[(ii)]   $\Psi _p(\W) \cap
B(0,r) \subset [p^{-1} \cdot S]_{c\delta }^{\He^1}$.
\end{enumerate}

We consider (i). Let $q\in (p^{-1} \cdot S) \cap  B(0,r)$. We will
prove that $q \in [ \Psi _p(\W)]_{c\delta }^{\He^1}$ for a
constant $c$ depending only on $L$. Firstly, since $q\in \Phi
^{(p^{-1})} (\W) \cap B(0,r)$, we have that $q= \Phi ^{(p^{-1})}
(0,y, \tau _p(y) + t )$ for some $(y,t) \in \R^2$ and $\| q\| = \|
\Phi ^{(p^{-1})} (y, \tau _p(y) + t ) \| \leq r$. More precisely,
by the definition of $\Phi ^{(p^{-1})}$, we find
\begin{equation*}
\begin{aligned}
|\varphi ^{(p^{-1})} ( y, \tau _p(y) +t )| & \leq r,\quad |y| \leq
r,\quad\text{and}\quad
 \left|\tau_p(y)+t - \tfrac{ 1}{ 2 }y \varphi ^{(p^{-1})} ( y, \tau _p(y) +t ) \right|   \leq \tfrac{r^2}{4},\\
\end{aligned}
\end{equation*}
\begin{equation}\label{stimatsquare}
\begin{aligned}
|t|\leq   \left|\tau_p(y)+t - \tfrac {1 }{2} y \varphi ^{(p^{-1})}
( y, \tau _p(y) +t ) \right|+|\tau_p(y)|
+  \left| \tfrac {1}{ 2} y \varphi ^{(p^{-1})} ( y, \tau _p(y) +t )\right|  & \stackrel{\eqref{(6.31)new}}{\leq}C_Lr^2+ \tfrac{3r^2}{4}.\\
\end{aligned}
\end{equation}
Now applying Lemma \ref{betterProp3.36} to the point $(y,t)$, we
obtain
 \begin{equation*}
 \begin{aligned}
d_\He ( \Psi_p (y,t), q ) & = d_\He ( \Psi_p (y,t), \Phi
^{(p^{-1})} (y,\tau_p(y) +t ) )\leq H|t|^{(1+\alpha)/2}
\stackrel{\eqref{stimatsquare}}{\lesssim}_L \delta,
\end{aligned}
\end{equation*}
so $q \in [ \Psi _p(\W)]_{c\delta }^{\He^1}$, as desired.

Next we consider (ii). Let $q\in \Psi _p(\W) \cap  B(0,r)$. We
want to prove that $q \in [p^{-1} \cdot S]_{c \delta }^{\He^1}$
for a constant $c$ that depends only on $L$. Since $q= \Psi_p (y,
t ) = \Phi ^{(p^{-1})} (y, \tau _p(y) ) \cdot (0,0,t)$ for some
$(y,t) \in \R^2$ and $\|q\| \leq r$, we have that
\begin{equation*}
\begin{aligned}
|\varphi ^{(p^{-1})} ( y, \tau _p(y) )|  \leq r,\quad |y|  \leq
r,\quad\text{and}\quad
 \left|t+ \tau _p(y) - \tfrac{1}{ 2} y \varphi ^{(p^{-1})} ( y, \tau _p(y) ) \right| \leq \tfrac{r^2}{4}.
\end{aligned}
\end{equation*}
and so
\begin{equation}\label{stimatsquare2}
\begin{aligned}
|t| \leq  \left|t +\tau _p(y) - \tfrac{y}{ 2}  \varphi ^{(p^{-1})}
( y, \tau _p(y) ) \right|  + |\tau _p(y)| + \left| \tfrac{y}{ 2}
\varphi ^{(p^{-1})} ( y, \tau _p(y) ) \right|
\stackrel{\eqref{(6.31)new}}{\leq}
 \tfrac{3r^2}{4} + C_L r^2.\\
\end{aligned}
\end{equation}
Now we apply Lemma \ref{betterProp3.36} to the point $(y,t)$,
hence
 \begin{equation*}
 \begin{aligned}
d_\He ( q , \Phi ^{(p^{-1})} (y,\tau_p(y) +t ))  = d_\He ( \Psi_p
(y,t), \Phi ^{(p^{-1})} (y,\tau_p(y) +t ) )
 \leq H|t|^{(1+\alpha)/2} \stackrel{\eqref{stimatsquare2}}{
\lesssim _{L}}  \delta,
\end{aligned}
\end{equation*}
so $q \in [p^{-1} \cdot S]_{c \delta }^{\He^1}$ for $c$ depending
only on $L$, as desired. This completes the proof.
\end{proof}

\medskip

Let $\pi \colon \He^{1} \to \R^{2}$ be the projection $\pi(x,y,t)=(x,y)$. Then $\pi$ is $1$-Lipschitz $(\He^{1},d_{\He}) \to (\R^{2},|\cdot|)$, which easily implies the following statement:

\begin{lemma}\label{l:Notes3}
Assume that $A_1,A_2\subset \He^1$, $p\in \He^1$, $r>0$, and
$\delta>0$ are such that \begin{displaymath} A_1 \cap B(p,r)
\subset [A_2]_{\delta}^{\He^1},\end{displaymath} then
\begin{displaymath}
\pi(A_1 \cap B(p,r))\subset [\pi(A_2 \cap
B(p,r+\delta))]^{\mathbb{R}^2}_{\delta}.
\end{displaymath}
\end{lemma}

The lemma will applied with $A_{2} = F$, a Lipschitz flag. Then $\pi(F) \subset \R^{2}$ is a Lipschitz graph, and the lemma says that $\pi(A_1\cap B(p,r))$ is contained in the Euclidean $\delta$-neighbourhood of the Lipschitz graph $\pi(F)$ whenever if $A_1\cap B(p,r) \subset [F]_{\delta}^{\He^{1}}$.

\subsection{Proof for graphs with extra vertical H\"older regularity}

In this section, we prove Theorem \ref{mainIntroVertical} from the
introduction, which we restate here for the reader's convenience.

\begin{thm}\label{mainVertical} Let  $S \subset \He^1$
be the intrinsic graph of a globally defined but compactly
supported intrinsic Lipschitz function with extra vertical
regularity. Then $S$ has big pieces of bilipschitz images of the
parabolic plane $(\Pi,d_{\Pi})$. In particular, $S$ is LI
rectifiable.
\end{thm}

The theorem will be proven as an application of Theorem \ref{main}
and a reduction to unit scale analogous to the one after Lemma
\ref{blowUpLemma}. We will verify the hypotheses of Theorem
\ref{main} for $(G,d_G)=(\Pi,d_{\Pi})$ and $(M,d_M)=(S,d_{\He})$,
with $x_0=0\in G$ and an arbitrary point $p_0\in S$. Since the map
$(y,t)\mapsto (0,y,t)$ is an isometric isomorphism between
$(\Pi,+,d_{\Pi})$ and $(\W,\cdot,d_{\He})$, it suffices to
construct maps $i_{w\to p}:\W \to S$ with the desired properties.
As in Section \ref{verification}, the maps $i_{w\to p}$ will be independent of the "scale" parameter $k \in \N$, and can be defined for all $p\in S$ and all $w\in
\W$.

\begin{definition}\label{d:iotanew}
Let $S= \Phi(\W)$ be the intrinsic graph of an intrinsic Lipschitz
function $\varphi : \W \to \V$. For each point $p\in S$ and $u \in
\mathbb{W}$, let $q := q[p,u] \in S$ be any point satisfying
\begin{equation}\label{form31new} d_{\He}\left(p \cdot \Psi _p( u) , q    \right) =
\dist_{\He}\left(p \cdot \Psi _p( u) ,S\right). \end{equation}
Then, define $i_{w\to p} :\W \to S$  as
\begin{equation}\label{eq:i_2new}i_{w \to p}(v) := q[p,w^{-1}\cdot v].\end{equation}
\end{definition}
\begin{remark}\label{r:iForFlag}
Remark \ref{r:WhatItMeansForFlags} implies that if $S=\Phi(\W)$ is
itself a Lipschitz flag, then $i_{w\to p}(v)= p \cdot
\Psi_p(w^{-1}\cdot v)$ for all $v,w\in \W$. In general, the extra vertical H\"older regularity allows
to control the distance between $i_{w\to p}(v)$ and $p \cdot
\Psi_p(w^{-1}\cdot v)$.
\end{remark}
If $\varphi$ has extra vertical H\"older regularity with constants
$0<\alpha\leq 1$ and $H>0$, then Lemma \ref{betterProp3.36}
immediately implies that
\begin{align}\label{eq:Distance estimate} d_{\He}(p\cdot\Psi_p(w^{-1}\cdot w'),i_{w \to p}(w')) &= \dist_{\He}\left(\Psi_{p}(w^{-1} \cdot w'),p^{-1} \cdot S\right) \notag\\
& =
\dist_{\He}
\left(\Psi_p\left(w^{-1}\cdot
w'\right),\Phi^{(p^{-1})}(\W)\right) \notag\\& \leq Hd_{\He}(w,w')^{1 + \alpha},
\end{align}
for all $p\in S=\Phi(\W)$ and $w,w'\in
\W$.

Once again, $i_{w \to p}(v)$ does not depend on the points $v$ and
$w$ individually, but only on the product $w^{-1}\cdot v$, and by
definition $i_{w\to p}(w)=p$. To apply Theorem \ref{main}, we need to verify the two hypotheses \eqref{ISO} and \eqref{comp}. We start by showing that \eqref{ISO} holds for a
constant which depends only on the bilipschitz constant of
$\Psi_p$, which in turn depends only on the intrinsic Lipschitz
constant of $\varphi$, recall the comment below Definition
\ref{d:Psi_p}. Now \eqref{ISO} follows immediately
from \eqref{eq:Distance estimate} and the triangle inequality:
\begin{equation}\label{eq:FirstCondVertical}
|d_{\He}(i_{w\to p}(w'),i_{w\to p}(w''))-
d_{\He}(\Psi_p(w^{-1}\cdot w'),\Psi_p(w^{-1}\cdot w'))|\lesssim H
\max\{d_{\He}(w,w'),d_{\He}(w,w'')\}^{1+\alpha},
\end{equation}
for all $p\in S$ and $w,w',w''\in \W$. We proceed to verify
condition \eqref{comp} in our situation:

\begin{proposition}\label{p:NotesCor}
Let $\varphi:\W \to \V$ be an intrinsic $L$-Lipschitz function
that has extra vertical H\"older regularity with constants
$0<\alpha\leq 1$ and $H>0$. If $w_1,w_2 \in \W$ satisfy
$\|w_1\|,\|w_2\|\leq 1$, and $p,q\in \Phi(\W)$ satisfy $i_{w_1\to
p}(w_2)=q$, then
\begin{equation}\label{form29}
d_{\He}\left(i_{w_1 \to p}(w_3),i_{w_2 \to
q}(w_3)\right)\lesssim_{H,L}
\max\left\{d_{\He}(w_1,w_2),d_{\He}(w_1,w_3),d_{\He}(w_2,w_3)
\right\}^{1+\frac{\alpha}{2}}
\end{equation}
for all $w_3\in \W$
 with $d_{\He}(w_1,w_3)\leq 1$.
\end{proposition}

\begin{remark}
If $\varphi$ does not depend on the $t$-variable, that is,
$\Phi(\W)$ is itself a Lipschitz flag and the extra H\"older
regularity holds with constant $H=0$, then the left hand side of
\eqref{form29} vanishes for all
$w_1,w_2,w_3\in \W$ with $i_{w_1\to p}(w_2)=q$. Indeed, Remark
\ref{r:iForFlag} implies in this case that
\begin{displaymath}
i_{w_1\to p}(w_3)=p\cdot \Psi_p(w_1^{-1}\cdot w_3),\quad i_{w_2\to
q}(w_3)=q\cdot \Psi_q(w_2^{-1}\cdot w_3),
\end{displaymath}
and
\begin{equation}\label{eq:qForm}
q=i_{w_1\to p}(w_2)=p\cdot \Psi_p(w_1^{-1}\cdot w_2).
\end{equation}
Hence, the left hand side of \eqref{form29} can be written as
\begin{displaymath}
d_{\He}(i_{w_1\to p}(w_3),i_{w_2\to q}(w_3))=
d_{\He}(\Psi_p(w_1^{-1}\cdot w_3),\Psi_p(w_1^{-1}\cdot w_2)\cdot
\Psi_q(w_2^{-1}\cdot w_3)).
\end{displaymath}
We recall from Remark \ref{r:WhatItMeansForFlags} that
\begin{equation}\label{form37} \Psi_q(y,t)= \Phi^{(q^{-1})}\left(y,t+ \int_0^y
\varphi^{(q^{-1})}(\eta,\tau_q(\eta))\,d\eta\right). \end{equation}
Let us spell out the formula for $\Phi^{(q^{-1})}$:
\begin{equation}\label{eq:chain_flag}
\Phi^{(q^{-1})}\overset{\eqref{eq:qForm}}{=}
[\Phi^{(p^{-1})}]^{(\Psi_p(w_1^{-1}\cdot w_2)^{-1})}=
\Psi_p(w_1^{-1}\cdot w_2)^{-1}\cdot
\Phi^{(p^{-1})}(\pi_{\W}(\Psi_p(w_1^{-1}\cdot w_2)\cdot [\cdot])),
\end{equation}
where we have applied the formula $\Phi(\pi_{\W}(p \cdot v)) = p \cdot \Phi^{(p^{-1})}(v)$ from Lemma \ref{l:translated} in the last
equality. Using \eqref{form37}-\eqref{eq:chain_flag}, and writing
$w_i=(0,y_i,t_i)$, it is not difficult to show that
\begin{align*}
&\Psi_p(w_1^{-1}\cdot w_2)\cdot \Psi_q(w_2^{-1}\cdot
w_3)\\&=\Psi_p(w_1^{-1}\cdot w_2)\cdot
\Phi^{(q^{-1})}\left(y_3-y_2,t_3-t_2+\int_0^{y_3-y_2}\varphi^{(q^{-1})}(\eta,\tau_q(\eta))\,d\eta\right)
\overset{\eqref{eq:chain_flag}}{=}\Psi_p(w_1^{-1}\cdot w_3).
\end{align*}
We omit some computations, as the remark only serves to motivate Proposition \ref{p:NotesCor}.
\end{remark}

\begin{proof}[Proof of Proposition \ref{p:NotesCor}]
We first apply the triangle inequality:
\begin{align*}
d_{\He}\left(i_{w_1 \to p}(w_3),i_{w_2 \to q}(w_3)\right) & \leq
d_{\He}\left(p \cdot \Psi_p(w_1^{-1}\cdot
w_3),i_{w_1 \to p}(w_3)\right)\\&\quad + d_{\He}(p \cdot \Psi_p(w_1^{-1}\cdot w_3),q \cdot
\Psi_q(w_2^{-1}\cdot w_3)) \\&\quad + d_{\He}\left(q \cdot
\Psi_q(w_2^{-1}\cdot w_3),i_{w_2 \to q}(w_3)\right).
\end{align*}
The estimate \eqref{eq:Distance estimate} shows that the first and third terms are bounded from above by $H d_{\He}(w_1,w_3)^{1+\alpha}$ and $H d_{\He}(w_2,w_3)^{1+\alpha}$, respectively, which is better than claimed. So, the heart of the matter is to prove an upper bound for the second term. This is the content of the next lemma.
\end{proof}

\begin{lemma}\label{l:NotesLem5}
Under the same assumptions as in Proposition \ref{p:NotesCor}, we
have
\begin{displaymath}
d_{\He}(p \cdot \Psi_p(w_1^{-1}\cdot w_3),q \cdot
\Psi_q(w_2^{-1}\cdot w_3)) \lesssim_{H,L}
\max\left\{d_{\He}(w_1,w_2),d_{\He}(w_1,w_3),d_{\He}(w_2,w_3)
\right\}^{1+\frac{\alpha}{2}}.
\end{displaymath}
\end{lemma}

\begin{proof} We fix points
$ p$ and  $q=i_{w_1\to p}(w_2)$ as in the assumptions of
Proposition \ref{p:NotesCor}. We may assume with no loss of generality that $w_{1} \neq w_{2}$, since otherwise $q = p$ and the claimed estimate is clear. By left invariance of $d_{\He}$, we have
\begin{equation}\label{eq:TermToBeBounded}
d_{\He}(p \cdot \Psi_p(w_1^{-1}\cdot w_3),q \cdot
\Psi_q(w_2^{-1}\cdot w_3))= d_{\He}(\Psi_p(w_1^{-1}\cdot
w_3),p^{-1}\cdot q \cdot \Psi_q(w_2^{-1}\cdot w_3)).
\end{equation}
Let us denote $p^{-1}\cdot q =: (x,y,t)$. We then define the
curves
\begin{displaymath}
\gamma_p:\R \to \R^2,\quad \gamma_p(s):=
(\varphi^{(p^{-1})}(s,\tau_p(s)),s)
\end{displaymath}
and
\begin{displaymath}
\gamma_q:\R \to \R^2,\quad \gamma_q(s):=
(\varphi^{(q^{-1})}(s,\tau_q(s))+ x,s+y),
\end{displaymath}
whose traces are the $\pi$-projections of the
corresponding Lipschitz flags:
\begin{displaymath}
  \gamma_p(\R)= \pi\left(\Psi_p(\W)\right)\quad \text{and}\quad
   \gamma_q(\R)= \pi\left(p^{-1}\cdot q \cdot \Psi_q(\W)\right)
\end{displaymath}
We observe that the curves $\gamma_{p}$ and $\gamma_{q}$ come close in at least one point. Writing
\begin{equation}\label{eq:w1w2w3}
w_1 = (0,y_1,t_1),\quad w_2=(0,y_2,t_2),\quad w_3=(0,y_3,t_3),
\end{equation}
and recalling that $q = i_{w_{1} \to p}(w_{2})$ by the assumption in the lemma, we have
\begin{align}
\notag|\gamma_p(y_2-y_1)-\gamma_q(0)|&=
|\gamma_p(y_2-y_1)-(x,y)| \leq d_{\He}(\Psi_p(w_1^{-1}\cdot w_2),p^{-1}\cdot q)\\
\notag &= d_{\He} (p\cdot \Psi_p(w_1^{-1}\cdot w_2), i_{w_1\to p}(w_2))\\
&\overset{\eqref{eq:Distance estimate}}{\leq} H
d_{\He}(w_1,w_2)^{1+\alpha}.\label{eq:close_to_meeting}
\end{align}
We next show, \emph{a fortiori}, that the curves $\gamma_p$ and $\gamma_q$ also stay close
to each other for some time. Precisely, we claim that
\begin{equation}\label{eq:curve_closeness}
\dist_{\R^{2}}(\gamma_p(s+y_2-y_1),\gamma_{q}(\R)) \lesssim_{H,L}
\max\left\{|s+y_2-y_1|,d_{\He}(w_1,w_2) \right\}^{1+\alpha}, \quad s \in \R.
\end{equation}
Note that for $s=0$, the right hand side of
\eqref{eq:curve_closeness} equals $d_{\He}(w_1,w_2)^{1+\alpha}$, as
expected. To prove the claim for arbitrary $s\in \R$, we first
observe that
\begin{displaymath}
\gamma_p(s+y_2-y_1)= \pi(\Psi_p(s+y_2-y_1,0)). \end{displaymath}
Since $\Psi_{p}$ is Lipschitz according to the remark below Definition \ref{d:Psi_p},
\begin{displaymath}
\|\Psi_p(s+y_2-y_1,0)\|\lesssim_L \|(s+y_2-y_1,0)\|,
\end{displaymath}
and we find
\begin{displaymath}
\gamma_p(s+y_2-y_1) \in \pi\left(\Psi_p(\W) \cap B(0,r)\right)
\end{displaymath}
for some $0 < r\lesssim_L  \max\left\{|s+y_2-y_1|,d_{\He}(w_1,w_2) \right\}$. It follows from
Proposition \ref{p:contenutesSpPsi} and Lemma \ref{l:Notes3} that
\begin{equation}\label{eq:Incl1}
\gamma_p(s+y_2-y_1) \in \left[\pi\left(\left(p^{-1}\cdot \Phi(\W)
\right)\cap B(0,r+c\delta)\right)\right]^{\mathbb{R}^2}_{c
\delta},
\end{equation}
where the constant $c$ depends only on $L$, and
$\delta:=Hr^{1+\alpha}$. Since
\begin{displaymath}
d_{\He}(p,q) \leq d_{\He}(p^{-1}\cdot q,\Psi_p(w_1^{-1}\cdot
w_2))+ \|\Psi_p(w_1^{-1}\cdot w_2)\|\overset{\eqref{eq:Distance
estimate}}{\lesssim_{H,L}} d_{\He}(w_1,w_2)^{1+\alpha}+
d_{\He}(w_1,w_2),
\end{displaymath}
and $d_{\He}(w_1,w_2)\leq 1$, we have
\begin{equation}\label{eq:Incl2}
B\left(0,r+c \delta\right)\subset B(p^{-1}\cdot q,R)
\end{equation}
for some $R\geq r$ with $R\lesssim_{H,L}
\max\{|s+y_2-y_1|,d_{\He}(w_1,w_2)\}$. By another instance of
Proposition \ref{p:contenutesSpPsi} (applied to the point $q$),
Lemma \ref{l:Notes3}, and left translation by $p^{-1}$, we find
that
\begin{equation}\label{eq:Incl3}
\left[\pi\left(\left(p^{-1}\cdot \Phi(\W) \right)\cap
B(p^{-1}\cdot q,R)\right)\right]^{\mathbb{R}^2}_{c\delta}\subseteq
\left[\pi\left(p^{-1}\cdot q \cdot
\Psi_q(\W)\right)\right]^{\mathbb{R}^2}_{2cHR^{1+\alpha}}=
\left[\gamma_q(\mathbb{R})\right]^{\mathbb{R}^2}_{c_{L,H}R^{1+\alpha}}.
\end{equation}
for a constant $0<c_{L,H}<\infty$ that depends only on $L$ and
$H$. Then the claim \eqref{eq:curve_closeness} follows  by
combining the inclusions \eqref{eq:Incl1}, \eqref{eq:Incl2}, and
\eqref{eq:Incl3}.

We now fix $s \in \R$, and let $s'\in
\mathbb{R}$ be any point such that $|\gamma_p(s+y_2-y_1) - \gamma_{q}(s')| \lesssim_{H,L} \max\{|s + y_{2} - y_{1}|,d_{\He}(w_{1},w_{2})\}^{1 + \alpha}$. The existence of $s'$ is guaranteed by \eqref{eq:curve_closeness}. We next show that $s'$ cannot be too far from
$s$. Indeed, considering the first component of
$\gamma_p(s+y_2-y_1)-\gamma_q(s')$, we see immediately from \eqref{eq:curve_closeness} that
\begin{equation}\label{eq:p_q_comparison}
|\varphi^{(p^{-1})}(s+y_2-y_1,\tau_p(s+y_2-y_1))-\varphi^{(q^{-1})}(s',\tau_q(s'))-x|\lesssim_{H,L}
\max\{|s+y_2-y_1|,d_{\He}(w_1,w_2)\}^{1+\alpha}.
\end{equation}
Considering the second component, we find the estimate
\begin{equation}\label{eq:first_s_s'_estimate}
|s+y_2-y_1-s'-y|\lesssim_{H,L}
\max\{|s+y_2-y_1|,d_{\He}(w_1,w_2)\}^{1+\alpha}.
\end{equation}
By the initial estimate \eqref{eq:close_to_meeting}, we know that
\begin{equation}\label{eq:w1_w2_p_q}
|(y_2-y_1)-y|\leq H d_{\He}(w_1,w_2)^{1+\alpha},
\end{equation}
so \eqref{eq:first_s_s'_estimate} yields
\begin{equation}\label{eq:s_s'_comp}
|s-s'|\lesssim_{H,L}
\max\{|s+y_2-y_1|,d_{\He}(w_1,w_2)\}^{1+\alpha}.
\end{equation}
This last estimate allows us to deduce a version of
\eqref{eq:p_q_comparison} with "$s'$" replaced by "$s$". Indeed,
recalling that $s\mapsto\psi_q(s)=
\varphi^{(q^{-1})}(s,\tau_q(s))$ is a Lipschitz function
$\mathbb{R} \to \mathbb{R}$ whose Lipschitz constant depends only
on $L$ (see below Definition \ref{d:mappsi_p}), we find
\begin{align}
\notag|\varphi^{(p^{-1})}&(s+y_2-y_1,\tau_p(s+y_2-y_1))-\varphi^{(q^{-1})}(s,\tau_q(s))-x|\\&\leq
|\varphi^{(q^{-1})}(s,\tau_q(s))-\varphi^{(q^{-1})}(s',\tau_q(s'))|\\\notag&+
|\varphi^{(p^{-1})}(s+y_2-y_1,\tau_p(s+y_2-y_1))-\varphi^{(q^{-1})}(s',\tau_q(s'))-x|\\
\notag&\overset{\eqref{eq:p_q_comparison}}{\lesssim}_{L,H}|s-s'| +
\max\{|s+y_2-y_1|,d_{\He}(w_1,w_2)\}^{1+\alpha}\\
&\overset{\eqref{eq:s_s'_comp}}{\lesssim}_{L,H}
\max\{|s+y_2-y_1|,d_{\He}(w_1,w_2)\}^{1+\alpha}.\label{eq:bound_two_s}
\end{align}
After these preparations, we are ready to deduce the desired upper
bound for \eqref{eq:TermToBeBounded} by considering $s:=y_3-y_2$.
We will show that
\begin{equation}\label{eq:CompacitbilityGoal}
d_{\He}(\Psi_p(w_1^{-1}\cdot w_3),p^{-1}\cdot q\cdot
\Psi_q(w_2^{-1}\cdot
w_3))\lesssim_{L,H}\max\{d_{\He}(w_1,w_2),d_{\He}(w_1,w_3),d_{\He}(w_2,w_3)\}^{1+\frac{\alpha}{2}}.
\end{equation}

It is convenient to estimate the expression on the left hand side
as follows
\begin{align}
\notag d_{\He}(\Psi_p(w_1^{-1}\cdot w_3)&, p^{-1}\cdot q\cdot
\Psi_q(w_2^{-1}\cdot w_3))\\
\label{i:I1} &\leq
 d_{\He}(\Psi_p(w_1^{-1}\cdot w_3),\Psi_p(w_1^{-1}\cdot
w_2)\cdot
\Psi_q(w_2^{-1}\cdot w_3))\\
\label{i:I2} &+ d_{\He}(\Psi_q(w_2^{-1}\cdot w_3),
\Psi_p(w_1^{-1}\cdot w_2)^{-1}\cdot p^{-1}\cdot q\cdot
\Psi_q(w_2^{-1}\cdot w_3))
\end{align}

First, the term \eqref{i:I2} can be bounded using the fundamental
commutator relation as in \eqref{eq:FundCommRel} with
\begin{displaymath}
\mathfrak{a}:=[p\cdot \Psi_p(w_1^{-1}\cdot w_2)]^{-1}\cdot q
\quad\text{and}\quad \mathfrak{b}:=\Psi_q(w_2^{-1}\cdot w_3).
\end{displaymath}
This yields
\begin{align*}
\eqref{i:I2}&\lesssim \|[p\cdot \Psi_p(w_1^{-1}\cdot
w_2)]^{-1}\cdot q\|+ \| [p\cdot \Psi_p(w_1^{-1}\cdot
w_2)]^{-1}\cdot q\|^{\frac{1}{2}}\|\Psi_q(w_2^{-1}\cdot
w_3)\|^{\frac{1}{2}}
\\&\lesssim_{L,H}d_{\He}(w_1,w_2)^{1+\alpha}+
d_{\He}(w_1,w_2)^{\frac{1+\alpha}{2}}d_{\He}(w_2,w_3)^{\frac{1}{2}},
\end{align*}
where the last inequality follows from $q=i_{w_1\to p}(w_2)$, the
estimate \eqref{eq:Distance estimate}, and the Lipschitz
continuity of $\Psi_p$. Hence, recalling that
$d_{\He}(w_1,w_2)\leq 2$, the expression \eqref{i:I2} can be
bounded from above by the right hand side of
\eqref{eq:CompacitbilityGoal}.

Next, we handle the term \eqref{i:I1}. Since points on the
$t$-axis commute with all other elements in $\He^1$, it follows
from the definition of
 $\Psi_p$ and $\Psi_q$ that \eqref{i:I1} is independent of the vertical components
of $w_1,w_2,w_3$. Writing these points in coordinates, as in
\eqref{eq:w1w2w3},  and recalling that $s=y_3-y_2$, we thus find
\begin{equation}\label{eq:ProdFormula}
\eqref{i:I1}=\|\Psi_p(s+y_2-y_1,0)^{-1}\cdot
\Psi_p(y_2-y_1,0)\cdot \Psi_q(s,0) \|.
\end{equation}
While $\Psi_p$ and $\Psi_q$ are in general not group
homomorphisms, their second components are linear:
\begin{displaymath}
[\Psi_p]_2(y,t)=[\Psi_q]_2(y,t)=y,\quad (y,t)\in \W.
\end{displaymath}
Thus, the second coordinate of the product in
\eqref{eq:ProdFormula} vanishes by linearity, and it suffices to
consider the first and and third coordinate, which we denote by
$I_1$ and $I_2$, respectively, so that
\begin{equation}\label{eq:finalStepI1}
\eqref{i:I1} \lesssim |I_1|+ |I_2|^{\frac{1}{2}}.
\end{equation}
Using that $\varphi^{(q^{-1})}(0,0)=0$, we may write
\begin{align*}
I_1=\varphi^{(p^{-1})}(y_2-y_1,\tau_p(y_2-y_1&))-\varphi^{(q^{-1})}(0,0)-x\\&+x+\varphi^{(q^{-1})}(s,\tau_q(s))-\varphi^{(p^{-1})}(s+y_2-y_1,\tau_p(s+y_2-y_1)).
\end{align*}
The term $I_1$ is the sum of two expressions of the same form as
in the estimate \eqref{eq:bound_two_s}, and we thus deduce that
\begin{displaymath}
|I_1|\lesssim
 \max\{|y_2-y_1|,d_{\He}(w_1,w_2)\}^{1+\alpha}+
\max\{|s+y_2-y_1|,d_{\He}(w_1,w_2)\}^{1+\alpha}.
\end{displaymath}
Clearly, $|y_2-y_1|\leq d_{\He}(w_1,w_2)$ and by the choice of
$s=y_3-y_2$, we have
\begin{displaymath}
|s+y_2-y_1|=|y_3-y_1|\leq d_{\He}(w_1,w_3).
\end{displaymath}
Thus we see that $|I_1|$ is bounded from above by the right hand
side of \eqref{eq:CompacitbilityGoal}, using again that
$d_{\He}(w_1,w_3),d_{\He}(w_1,w_2)\leq 2$. It remains to bound
$|I_2|$, where $I_2$ denotes the third component of the product in
\eqref{eq:ProdFormula}. A direct computation yields
\begin{displaymath}
I_2=
-\tau_p(s+y_2-y_1)+\tau_p(y_2-y_1)+\tau_q(s)+s\varphi^{(p^{-1})}(y_2-y_1,\tau_p(y_2-y_1)),
\end{displaymath}
and we continue as follows:
\begin{align*}
|I_2|&=|\tau_p(s+y_2-y_1)-\tau_p(y_2-y_1)-\tau_q(s)-s\varphi^{(p^{-1})}(y_2-y_1,\tau_p(y_2-y_1))|\\&=
\left|\int_{y_2-y_1}^{s+y_2-y_1} \dot{\tau}_p(\sigma)\,d\sigma-
\int_0^s \dot{\tau}_q(\sigma)+\varphi^{(p^{-1})}(y_2-y_1,\tau_p(y_2-y_1))\,d\sigma\right|\\
&=\left|\int_0^s
\dot{\tau}_p(\sigma+y_2-y_1)-\dot{\tau}_q(\sigma)-
\varphi^{(p^{-1})}(y_2-y_1,\tau_p(y_2-y_1))\,d\sigma\right|\\
&= \bigg| \int_0^s
\left[ \varphi^{(p^{-1})}(\sigma+y_2-y_1,\tau_p(\sigma+y_2-y_1))-\varphi^{(q^{-1})}(\sigma,\tau_{q}(\sigma))- x \right]\\
& \qquad + \left[x + \varphi^{(q^{-1})}(0,\tau_{q}(0)) - \varphi^{(p^{-1})}(y_2-y_1,\tau_p(y_2-y_1)) \right] \, d\sigma\bigg| \\
&\overset{\eqref{eq:bound_two_s}}{\lesssim}_{L,H} \int_{J_s}
\max\{|\sigma+y_2-y_1|,d_{\He}(w_1,w_2)\}^{1+\alpha}\,d\sigma\\
&\lesssim_{H,L} |s|
\max\{d_{\He}(w_1,w_2),d_{\He}(w_1,w_3)\}^{1+\alpha}\\&\lesssim_{H,L}
\max\{d_{\He}(w_1,w_2),d_{\He}(w_1,w_3),d_{\He}(w_2,w_3)\}^{2+\alpha},
\end{align*}
where $J_s:=[s,0]$ if $s\leq 0$ and $J_s:=[0,s]$ if $s\geq 0$. To
justify the application of \eqref{eq:bound_two_s} above, we have
applied inside the integral an analogous argument as we did to
bound the term $I_1$.

Finally inserting the bounds for $|I_1|$ and $|I_2|$ in
\eqref{eq:finalStepI1}, we conclude that \eqref{i:I1} is bounded
from above by the right hand side of
\eqref{eq:CompacitbilityGoal}. Combined with the bound for
\eqref{i:I2}, this concludes the proof of the lemma.
\end{proof}

\begin{proof}[Proof of Theorem \ref{mainVertical}]
The BP$G$BI condition "at unit scale" follows from Theorem
\ref{main}, whose hypotheses \eqref{ISO} and \eqref{comp} we have
verified in \eqref{eq:FirstCondVertical} and Proposition
\ref{p:NotesCor}, respectively. Here
\begin{equation*}\label{eq:G_Mnew} (G,d_{G},\mu) = (\mathbb{R}^2,d_{\Pi},\mathcal{L}^2),  \quad \text{and} \quad (M,d_{M}) = (S,d_{\He}), \end{equation*}
with $x_{0} = 0 \in G$, and $p_{0} \in S$ arbitrary, and we recall
that $(G,d_G)$ is isometric to $(\W,d_{\He})$. More precisely,
Theorem \ref{main} yields the existence of $2L$-bilipschitz maps
$f \colon K \to S \cap B(p,1)$, $p \in S$, where $K \subset G$
with $\calH^{3}(K) \geq \delta >0$. The constant $L$ only depends
on the intrinsic Lipschitz constant of $\varphi$, and $\delta
> 0$ depends in addition on $\alpha$ and the constant $H$ in
\eqref{eq:VertHolDef}. Since  property \eqref{eq:VertHolDef}
improves under ``zooming in'', see Remark
\ref{r:VertHolDilateTranslate}, we can argue analogously as in
Section \ref{ss:red_unit_scale}. Let $p \in S$ and, first, $0 < r
\leq C$, where $C := 2\diam_{\He}(\Phi(\spt \varphi))$. Using
Remark \ref{r:VertHolDilateTranslate} and the support assumption
on $\varphi$, we see that $S_{1/r} := \delta_{1/r}(S)$ is an
intrinsic graph of an intrinsic Lipschitz function (with the same
constant) satisfying \eqref{eq:VertHolDef} with constants $\alpha$
and $H'=H'(H,C)$.

Therefore, by the BP$G$BI property at scale $r = 1$, every ball
$S_{1/r} \cap B(p,1)$ contains the image of a $2L$-bilipschitz map
$g$ from a compact set $K \subset G$ with $\calH^{3}(K) \geq
\delta = \delta(C) > 0$. Now, one may simply pre- and post-compose
$g$ with the natural dilations in $G$ and $\He^1$ to produce a
$2L$-bilipschitz map $g_{r} \colon \delta_{r}(K) \to S \cap
B(\delta_{r}(p),r)$ (note also that $\calH^{3}(\delta_{r}(K)) =
r^{3}\calH^{3}(K) \geq \delta r^{3}$).

Next, consider the case $r > C$. Then, if $p \in S$ is arbitrary,
the set $S \cap B(p,r)$ satisfies
\begin{displaymath} \calH^{3}([S \cap B(p,r)] \cap \W) \gtrsim \calH^{3}(S \cap B(p,r)). \end{displaymath}
Thus, the restriction of $\mathrm{Id}$ to $[S \cap B(p,r)] \cap
\W$ yields the desired bilipschitz map. The proof of Theorem
 \ref{mainVertical} is thus complete.
\end{proof}

\subsection{Application to $C^1$ and intrinsic $C^{1,\alpha}$ surfaces}

As a first application of Theorem \ref{mainVertical}, we deduce
the case $n=1$ of Theorem \ref{mainGraphs}, recalling from Example
\ref{ex:intrC1alpha} that a compactly supported
$C^{1,\alpha}_{\He}(\W)$ function is intrinsic Lipschitz and
satisfies the extra vertical H\"older regularity condition.

\begin{thm}\label{mainGraphsn_1} Let $S = \Phi(\W) \subset \He^1$, where $\varphi \in C^{1,\alpha}_{\He}(\W)$ is compactly supported.
Then $S$ has big pieces of bilipschitz images of the parabolic
plane $(\Pi,d_{\Pi})$.
\end{thm}

\begin{remark}
As a corollary of Theorem \ref{mainVertical}, we also obtain that
every Euclidean $C^1$ surface in $\He^1$ is rectifiable by
bilipschitz images of subsets of the parabolic plane. As written
in the introduction, this was known before by the work of
Cole-Pauls and Bigolin-Vittone, cf.\ Theorem \ref{CPBV}, but we
briefly explain how to deduce it from Theorem \ref{mainVertical}.
The reduction uses again the result by Balogh \cite{MR2021034},
which says that the set $\Sigma(S)$ of \emph{characteristic
points} of a Euclidean $C^1$ surface in $\He^1$ has vanishing
$3$-dimensional Hausdorff measure with respect to $d_{\He}$.

We will argue that outside $\Sigma(s)$, the surface $S$ can be
written locally as intrinsic graph of a compactly supported
Euclidean $C^1$ function, and hence as intrinsic Lipschitz graph
with extra vertical H\"older regularity. This will show that $S$
is rectifiable by bilipschitz images of subsets of the parabolic
plane.

We now turn to the details. Let $p \in S\setminus \Sigma(S)$. For
$r>0$ small enough, $S\cap B(p,r)$ is contained in the level set
$\{f=0\}$ of a Euclidean $C^1$ function
$f:\mathbb{R}^3\to\mathbb{R}$ with non-vanishing gradient in
$B(p,r)$. Without loss of generality, we may assume that $f(0)=0$,
$Xf(0)>0$ and
\begin{displaymath}
S\cap B(p,r)=\{q\in B(p,r):\; f(q)=0\}
\end{displaymath}
for $r>0$ with the property that $Xf(q)>0$ for all $q\in B(p,r)$.
Since $Xf(0)=\partial_x f(0)$, we may further assume, by making
$r$ smaller if necessary, that $\partial_x f(q)>0$ for all $q\in
B(p,r)$. In order to write $S\cap B(p,r)$, for small enough $r$,
as  intrinsic graph of a Euclidean $C^1$ function, we first
consider the diffeomorphism
\begin{displaymath}
F:\mathbb{R}^3 \to \mathbb{R}^3,\quad
F(x,y,t)=\left(x,y,t+\tfrac{xy}{2}\right).
\end{displaymath}
Then $F(S\cap B(p,r))$ is contained in the level set of  $f\circ
F^{-1}$, and hence it is again a Euclidean $C^1$ surface.  Since
the derivative of $F$ at the origin is the identity, and
$\partial_x f(q)>0$ for all $q\in B(p,r)$, we can apply the usual
implicit function theorem to deduce that, if $r>0$ is small
enough,  there is an open set $U\subset \mathbb{R}^2$, and a
Euclidean $C^1$ function $\psi: U  \to \mathbb{R}$ such that
$F(S\cap B(p,r))$ is the Euclidean graph of $\psi$ over the set
$U$ in the $yt$-plane:
\begin{displaymath}
F(S\cap B(p,r))= \{(\psi(y,t),y,t):\; (y,t)\in U\}.
\end{displaymath}
It is easy to see that the preimage of this set under $F$ is then
given by the \textbf{intrinsic} graph of $\psi$,
\begin{displaymath}
S\cap B(p,r)=\left\{
(\psi(y,t),y,t-\tfrac{1}{2}y\psi(y,t)):\,(y,t)\in U\right\}.
\end{displaymath}
We will next modify $\psi$ to obtain a Euclidean $C^1$ function
$\varphi$ that is defined on the entire plane, but compactly
supported. To this end, let $B',B \subset U$ be concentric balls,
relatively open in the $yt$-plane $\W$ (identified with
$\mathbb{R}^2$), such that
\begin{displaymath}
[B(p,r') \cap S] \subseteq \{w\cdot \psi(w):\; w\in B'\} \subseteq
\{w\cdot \psi(w):\, w\in B\} \subseteq [B(p,r)\cap S]
\end{displaymath}
for some $0<r'<r$. We define
\begin{equation*}
\varphi (w) := \left\{\begin{array}{ll} \psi (w),&\text{if } w\in B',\\
\xi (w),&\text{if } w \in  B\setminus B',
\\0,&\text{otherwise},\end{array}\right.
\end{equation*}
with a suitable $C^1$ function $\xi$ in order that $\varphi $ is
also $C^1$. More precisely, $\varphi $ is a compactly supported
$C^1$ function defined in $\W$ such that  $S \cap B(p,r') =
\Phi(\W)\cap B(p,r')$. By Example \ref{ex:EuclC1}, we know that
$\varphi$ is also an intrinsic Lipschitz function with extra
vertical regularity. Finally, it follows from Theorem
\ref{mainVertical} that $\Phi(\W)$ is rectifiable by bilipschitz
images, and hence the same holds for  $S \cap B(p,r')$. Repeating
the argument for every noncharacteristic point in $S$ proves that
$S$ is rectifiable by bilipschitz images of subsets of the
parabolic plane, and in particular LI rectifiable.\end{remark}

\appendix

\section{Fat Cantor sets in metric measure spaces}\label{FCSPropProof}

Here is again the statement of Proposition \ref{FCSProp}:
\begin{proposition} Every doubling and complete metric measure space $(X,d,\mu)$ of diameter $\geq 1$ admits fat Cantor sets. In other words, for every $\epsilon > 0$ and $n_{0} \geq 0$, the constants $\delta(n_{0}) > 0$ and $\tau(\epsilon) > 0$ can be found as in Definition \ref{FCS}. They are also allowed to depend on the doubling constant of $(X,d,\mu)$.
\end{proposition}

\begin{proof} Let $\calQ = \cup \{\calQ_{n} : z \in \Z\}$ be a family of \textbf{closed} (and hence compact) Christ cubes on $(X,d,\mu)$, see \cite[Theorem 11]{MR1096400}. Thus, the cubes here are closures of the cubes defined in \cite[Theorem 11]{MR1096400}. By changing the indexing of the families $\calQ_{n}$ slightly, one may assume that $2^{-n} \lesssim_{X} \diam_{X}(Q) < 2^{-n}$ for all $Q \in \calQ_{n}$. According to \cite[(3.6)]{MR1096400}, the cubes in $\calQ$ can be chosen so that they have \emph{small boundary regions} in the following sense: there are constants $C \geq 1$ and $\eta > 0$ such that $\mu(\partial_{\rho} Q) \leq C\rho^{\eta}\mu(Q)$ for all $Q \in \calQ$, where
\begin{equation}\label{boundaries} \partial_{\rho} Q := \{x \in Q : \dist(x,Q^{c}) < \rho 2^{-n}\}, \qquad Q \in \calQ_{n}. \end{equation}
Fix $x \in X$. To begin the construction of a fat Cantor set inside $B(x,1)$, fix also the parameters $\epsilon > 0$ and $n_{0} \in \N$, and let $Q_{0}$ be a cube in $\calQ_{0}$ containing $x$ (there may be several options, since the cubes in $\mathcal{Q}_{0}$ are closures of "dyadic" cubes, but any choice will do). Since $\diam_{X}(Q_{0}) < 2^{-n_{0}} \leq 1$, one has $Q_{0} \subset B(x,1)$. Set $\calD_{n_{0}} := \{Q_{0}\}$. Now, if one simply declared that $\calD_{n} := \{Q \in \calQ_{n} : Q \subset Q_{0}\}$, then one would already have the conditions \nref{i}-\nref{iii} listed at the beginning of Section \ref{s:construction}. Then, the Cantor set $K$ defined by
\begin{displaymath} K := \bigcap_{n \geq n_{0}} \bigcup_{Q \in \calD_{n}} Q \end{displaymath}
would satisfy $\mu(K) = \mu(Q_{0}) \sim_{n_{0}} \mu(B(x,1))$ by the doubling hypothesis.

To secure, in addition, the separation condition \nref{iv}, one need to remove some boundary regions, and apply \eqref{boundaries}. Namely, fix a constant $\tau = \tau(\epsilon) > 0$, to be determined a little later, and define
\begin{displaymath} Q_{0}' := Q_{0} \, \setminus \, \partial_{\tau 2^{-n_{0}\epsilon}} Q_{0} \quad \text{and} \quad \calD_{n_{0}}' := \{Q_{0}'\}. \end{displaymath}
Then, assume that $\calD_{n}'$ has already been defined for some $n \geq n_{0}$. Assume also that the sets in $\calD_{n}'$ are obtained as compact subsets of sets in $\calD_{n}$: for every $Q \in \calD_{n}$, there corresponds a compact set $Q' \in \calD_{n}'$ with $Q' \subset Q$ (but it may, and will, sometimes happen that $Q' = \emptyset$). To define $\calD_{n + 1}'$, fix $Q \in \calD_{n + 1}$, and let $\widehat{Q}' \in \calD_{n}'$ be the compact set contained in the $\calD_{n}$-parent $\widehat{Q} \supset Q$. Define
\begin{displaymath} Q' := [Q \, \setminus \, \partial_{\tau 2^{-n\epsilon}} Q] \cap \widehat{Q}'.   \end{displaymath}
Then evidently $Q' \subset \widehat{Q}'$, and the conditions \nref{i}-\nref{iii} from the beginning of Section \ref{s:construction} remain valid for the modified collections $\calD_{n}'$, $n \geq n_{0}$. But now also condition \nref{iv} is valid. Indeed, if $Q_{1}',Q_{2}' \in \calD_{n}$ are distinct, and there still existed a point $x \in Q_{1}' \subset Q_{1}$ with $\dist(x,Q_{2}') < \tau 2^{-(1 + \epsilon)n}$, then clearly $x \in \partial_{\tau 2^{-n\epsilon}} Q_{1}$, and hence in fact $x \notin Q_{1}'$.

So, the only remaining concern is the $\mu$-measure of the new Cantor set
\begin{displaymath} K' := \bigcap_{n \geq n_{0}} \bigcup_{Q' \in \calD_{n}'} Q'. \end{displaymath}
Evidently, if $x \in Q_{0} \, \setminus \, K'$, then $x \in \partial_{\tau 2^{-n\epsilon}} Q$ for some $Q \in \calD_{n}$ with $Q \subset Q_{0}$ (hence $n \geq n_{0}$). Recalling the estimate for the $\mu$-measure of boundary regions above \eqref{boundaries}, one infers that
\begin{align*} \mu(Q_{0} \, \setminus \, K') & \leq \sum_{n \geq n_{0}} \mathop{\sum_{Q \in \calD_{n}}}_{Q \subset Q_{0}} \mu(\partial_{\tau 2^{-\epsilon n}} Q)\\
& \leq C\sum_{n \geq n_{0}} (\tau 2^{-n\epsilon})^{\eta} \mathop{\sum_{Q \in \calD_{n}}}_{Q \subset Q_{0}} \mu(Q)\\
& = C\tau^{\eta} \sum_{n \geq 0} 2^{-n\epsilon \eta} \mu(Q_{0}) \lesssim_{\epsilon,\eta} C\tau^{\eta} \mu(Q_{0}). \end{align*}
Therefore, choosing $\tau = \tau(C,\epsilon,\eta) > 0$ sufficiently small, one has $\mu(Q_{0} \, \setminus \, K') \leq \mu(Q_{0})/2$, hence $\mu(K') \gtrsim \mu(Q_{0}) \gtrsim_{n_{0}} \mu(B(x,1))$. The proof is complete. \end{proof}

\section{Containing pieces of $C_{\He}^{1,\alpha}$-surfaces on intrinsic graphs}\label{extensionAppendix}

This appendix contains the proof of the following proposition
which was needed in the proof of Theorem \ref{mainQualitative} (or
Theorem \ref{surfaces}).

\begin{proposition}\label{appProp} Let $S \subset \He ^n$ be a $C^{1,\alpha}_{\He}$-surface, $0 < \alpha \leq 1$. Then, for every $p_{0} \in S$, there exists $r_{0} > 0$, a vertical subgroup $\W \subset \He ^n$ with complementary horizontal subgroup $\V$, and a compactly supported intrinsic $C^{1,\alpha/3}$-function $\varphi \colon \W \to \V$ such that $S \cap B(p_{0},r_{0})$ is contained on the intrinsic graph of $\varphi$. \end{proposition}

\begin{proof} Fix $p_{0} \in S$, and let $r_{0} > 0$ and $B:=B(p_{0},r_{0})$ first be so small that $S \cap B$ can be written as
\begin{displaymath} S \cap \bar B = \{p \in  B : f(p) = 0\} \end{displaymath}
for some $f \in C^{1}_{\He}(B(p_{0},10r_{0}))$ satisfying
$\nabla_{\He}f(p_{0}) \neq 0$, and
\begin{equation}\label{appForm1} |\nabla_{\He}f(p_{1}) - \nabla_{\He} f(p_{2})| \leq Hd_{\mathbb{H}}(p_{1},p_{2})^{\alpha}, \qquad p_{1},p_{2} \in B(p_{0},10r_{0}). \end{equation}
It follows from \eqref{appForm1} that $f,\nabla_{\He} f \in
L^{\infty}(B(p_{0},10r_{0}))$. By making $r_{0}$ smaller, one may
further improve \eqref{appForm1} to
\begin{equation}\label{appForm1++} |\nabla_{\He} f(p_{1}) - \nabla_{\He} f(p_{2})| \leq \min\{H d_{\mathbb{H}}(p_{1},p_{2})^{\alpha},\epsilon\}, \qquad p_{1},p_{2} \in B(p_0,10r_{0}), \end{equation}
where $\epsilon > 0$ is a small absolute constant to be chosen
later. For notational convenience, we will also assume that
$\nabla_{\He}f(p_{0}) = (X_1f(p_{0}),\dots, X_{2n}f(p_{0})) =
(1,0,\dots, 0)$, but any other non-zero constant vector would work
equally well: it is only crucial to choose $\W$ (as in the
statement of the proposition) so that $\nabla_{\He}f(p_{0})$ is
the horizontal normal of $\W$. Under the present assumption, set
$\W := \{(0,y,t) : y\in \R^{2n-1},t \in \R\}$ and $\V :=
\{(x_1,0,0) : x_1 \in \R\}$.

Let $C > 20$ be another constant to be determined later, which may
depend on the \emph{data}
\begin{equation}\label{appData} \|f\|_{L^{\infty}(B(p_0,r_{0}))}, \, \|\nabla_{\He}f\|_{L^{\infty}( B)}, \, \|p_{0}\|, \, r_{0}, \, H \quad \text{and} \quad \epsilon. \end{equation}
Then, initially extend $f$ by setting $f(x_1,y,t) := x_1$ for
$(x_1,y,t) \in \He ^n \, \setminus \, B(p_0,Cr_{0})$. The main
task is now to extend $f$ to a function $f_{1} \in
C^{1,\alpha/3}_{\He}(\He ^n)$ in such a manner that $X_1f_{1} \geq
\tfrac{1}{2}$; then $\{f_{1} = 0\}$ will be an intrinsic
$C^{1,\alpha/3}$-graph containing $S \cap B$. The extension of $f$
to $f_{1}$ can be accomplished, up to a few additional details, by
using the standard proof of the Whitney extension theorem
\cite[Theorem 6.8]{FSSC} in $\He ^n$. An underlying observation is
that $\nabla_{\He} f \approx (1,0,\ldots, 0)$ on $B \cup [\He ^n
\, \setminus \, B(p_{0},Cr_{0})]$, and if $C$ is chosen large
enough, depending on the data in \eqref{appData}, the extension
$f_{1}$ can be arranged to have the same property.

Define
\begin{displaymath} k(p) := \nabla_{\He}f(p), \qquad p \in \bar{B}, \end{displaymath}
recalling that $f$ was initially defined on $B(p_{0},10r_{0})$.
Also define $k(p) := (1,0,\ldots,0) \equiv \nabla_{\He}f(p)$ for
$p \in \He ^n\, \setminus \, B(p_{0},Cr_{0})$, so both $k$ and $f$
are now defined on the closed set
\begin{displaymath} F := \bar{B} \cup [\He ^n\, \setminus \, B(p_{0},Cr_{0})]. \end{displaymath}
Recalling \eqref{appForm1++}, and that $\nabla_{\He} f(p) =
(1,0,\ldots, 0)$, we note that
\begin{equation}\label{appForm1+} |k(p_{1}) - k(p_{2})| \leq \min\{H d_{\mathbb{H}}(p_{1},p_{2})^{\alpha},\epsilon\}, \qquad p_{1},p_{2} \in F. \end{equation}
Towards applying the Whitney extension theorem \cite[Theorem
6.8]{FSSC}, consider the following quantity $R(q,p)$ appearing in
its statement:
\begin{displaymath} R(q,p) := \frac{f(q) - f(p) - \langle k(p), \pi(p^{-1} \cdot q) \rangle}{d_{\mathbb{H}}(p,q)}, \qquad p,q \in F. \end{displaymath}
Here $\pi$ is the projection $\pi(x_1,\ldots,x_{2n},t) =
(x_1,\ldots,x_{2n})$, and $\langle \cdot , \cdot \rangle$ stands
for the usual inner product in $\R^{2n}$. Recall that $\pi$ is a
Lipschitz map $\He ^n \to \R^{2n}$, and also a group homomorphism,
that is, $\pi(p \cdot q) = \pi(p) + \pi(q)$ for $p,q \in \He ^n$.
The following estimate for $|R(p,q)|$ will be needed, and next
verified:
\begin{equation}\label{appForm2} |R(q,p)| \lesssim \min \{H d_{\mathbb{H}}(p,q)^{\alpha},\epsilon\}, \qquad p,q \in F \cap B(p_{0},2Cr_{0}). \end{equation}
For $p,q \in \bar{B}$, the estimate follows immediately from
\eqref{appForm1++} and \cite[Lemma 4.2]{MR2223801}, which further
cites \cite[Theorem 2.3.3]{MontiThesis}. The case $p,q \in \He
^n\, \setminus \, B(p_{0},Cr_{0})$ is clear, as $R(p,q) = 0$
(recalling that $k(p) = (1,0,\ldots,0)$ and
$f(x_1,\ldots,x_{2n},t) = x_1$). Finally, consider points $q \in
B(p_{0},2Cr_{0}) \, \setminus \, B(p_{0},Cr_{0})$ and $p \in
\bar{B}$ (the case where the roles of $p$ and $q$ are reversed is
similar, and even slightly easier). Then $|k(p) - (1,0,\ldots,0)|
\leq \epsilon$, $f(q) = (1,0,\ldots,0) \cdot \pi(q)$, and
$d_{\mathbb{H}}(p,q) \gtrsim Cr_{0}$. Consequently,
\begin{align*} |R(q,p)| & = \left|\frac{f(q) - f(p) - \langle k(p), \pi(p^{-1} \cdot q) \rangle}{d_{\mathbb{H}}(p,q)} \right|\\
& \lesssim \frac{|\langle ((1,0,\ldots,0) - k(p)), \pi(q) \rangle|}{Cr_{0}} + \frac{|f(p) - \langle k(p), \pi(p) \rangle|}{Cr_{0}}\\
& \lesssim \epsilon \leq
\min\{d_{\mathbb{H}}(p,q)^{\alpha},\epsilon\},
\end{align*} noting that $|\pi(q)| \lesssim \|p_{0}\| + Cr_{0}$,
and choosing $C \geq 1$ eventually so large that $(\|p_{0}\| +
Cr_{0})/(Cr_{0}) \leq 2$ and $(Cr_{0})^{\alpha} \geq \epsilon$,
and
\begin{displaymath} |f(p) - \langle k(p), \pi(p) \rangle| \lesssim \|f\|_{L^{\infty}(\bar{B})} + \|\nabla_{\He}f\|_{L^{\infty}(\bar{B})}(\|p_{0}\| + r_{0}) \leq \epsilon C r_{0}. \end{displaymath}
This completes the proof of \eqref{appForm2}.

Next, we claim that $f$ can be extended to a function $f_{1} \in
C^{1,\alpha/3}_{\He}(\He ^n)$ with the additional property that
\begin{equation}\label{appForm3} |\nabla_{\He}f_{1}(p) - (1,0,\ldots,0)| \leq \tfrac{1}{2}, \qquad p \in \He ^n. \end{equation}
The proof follows the usual argument for the Whitney extension
theorem, see \cite[Theorem 6.8]{FSSC} or \cite[\S 6.5]{EG}, and
one just needs to check that the resulting extension is in
$C^{1,\alpha/3}_{\He}(\He ^n)$, and that \eqref{appForm3} is
satisfied. We start by setting up some notation. For any $p \in
\He ^n$, let
\begin{displaymath} r(p) := \dist_{\mathbb{H}}(p,F)/20. \end{displaymath}
Since $U := F^{c} = B(p_{0},Cr_{0}) \setminus \bar{B}$ is bounded
in our scenario, the numbers $r(p)$ above are uniformly bounded to
begin with (in \cite{FSSC} and \cite{EG}, one needs to take
instead $r(p) = \min\{1,\dist_{\mathbb{H}}(p,F)\}/20$ to fix
this). Thus, by the $5r$ covering theorem, there exists a
countable set $S \subset \He ^n\, \setminus \, F$ such that
\begin{displaymath} U = \bigcup_{s \in S} B(s,5r(s)), \end{displaymath}
and the balls $B(s,r(s))$, $s \in S$, are disjoint. One may then
proceed to define the (smooth) partition of unity $\{\nu_{s}\}_{s
\in S}$ of $U$, subordinate to the cover $\{B(s,10r(s))\}_{s \in
S}$, as in the proof of either \cite[\S 6.5]{EG} or \cite[Theorem
6.8]{FSSC}. The key properties are that
\begin{equation}\label{appForm8} \sum_{s \in S} \nu_{s} = \mathbf{1}_{U} \quad \text{and} \quad \sum_{s \in S} \nabla_{\He}\nu_{s}(p) \equiv 0, \end{equation}
and
\begin{equation}\label{appForm9} |\nabla_{\He}^{j}\nu_{s}(p)| \lesssim \frac{1}{r(p)^{j}}, \qquad p \in U, \, s \in S, \, j \in \{1,2\}. \end{equation}
Here $\nabla_{\He}^{2}$ simply refers to any second order
horizontal derivative. Moreover, the supports of the functions
$\nu_{s}$ have bounded overlap, that is, for $p \in U$ fixed,
there are only $\lesssim 1$ indices $s \in S$ with $v_{s}(p) \neq
0$ or $\nabla_{\He} v_{s}(p) \neq 0$. To be precise, in the proof
of \cite[Theorem 6.8]{FSSC}, condition \eqref{appForm9} is only
stated for first-order horizontal derivatives, that is, for $j=1$.
However, the bound for the second order derivatives easily follows
from formula  \cite[(57)]{FSSC}, observing that the functions
defined there are all obtained from a fixed smooth function by
rescaling with a factor proportional to $1/r(p)$ and using
properties of $q\mapsto d_{\He}(p,q)$.

Now, the extension $f_{1}$ is defined as follows:
\begin{displaymath} f_{1}(p) := \begin{cases} f(p), & \text{if } p \in F, \\ \sum_{s \in S} \nu_{s}(p)[f(\hat{s}) + \langle k(\hat{s}), \pi(\hat{s}^{-1} \cdot p) \rangle], & \text{if } p \in U. \end{cases} \end{displaymath}
Here, for $p \in U$ given, $\hat{p} \in F$ is any point satisfying
$\dist_{\mathbb{H}}(p,F) = d_{\mathbb{H}}(p,\hat{p})$. Since the
definition is precisely the same as the one in \cite[Theorem
6.8]{FSSC}, the function $f_{1}$ is readily a $C^{1}_{\He}(\He
^n)$-extension of $f$, and moreover
\begin{equation}\label{appForm4} \nabla_{\He}f_{1}(q) = k(q), \qquad q \in F. \end{equation}
To prove, further, that $f_{1} \in C^{1,\alpha/3}_{\He}(\He ^n)$,
and that \eqref{appForm3} holds, one needs to look closer at the
differences $|\nabla_{\He}f_{1}(p) - \nabla_{\He}f_{1}(q)|$. The
following estimates are copied from \cite[p. 250]{EG} (and
completely omitted in \cite{FSSC}, as there is virtually no
difference between $\He^{n}$ and $\R^{n}$ in this argument).
First, the horizontal gradient of $f_{1}$ on $U$ is evidently
\begin{equation}\label{appForm7} \nabla_{\He}f_{1}(p) = \sum_{s \in S} \{[f(\hat{s}) + \langle k(\hat{s}), \pi(\hat{s}^{-1} \cdot p) \rangle]\nabla_{\He}\nu_{s}(p) + \nu_{s}(p)k(\hat{s})\}, \qquad p \in U. \end{equation}
By \eqref{appForm4}, $\nabla_{\He}f_{1}$ and $\nabla_{\He}f$
coincide on $F$, hence satisfy the same estimates, and in
particular \eqref{appForm1+}. To understand the behaviour of
$\nabla_{\He}f_{1}$ outside $F$, consider first the case $p \in U$
and $q \in F$. First,
\begin{equation}\label{appForm6} |\nabla_{\He}f_{1}(p) - \nabla_{\He}f_{1}(q)|
\leq |\nabla_{\He}f_{1}(p) - k(\hat{p})| + |k(\hat{p}) - k(q)|.
\end{equation} Since $d_{\mathbb{H}}(\hat{p} , q) \leq d_{\mathbb{H}}(\hat{p} , p) + d_{\mathbb{H}}(p ,
q) \leq 2d_{\mathbb{H}}(p, q)$, the second term in
\eqref{appForm6} can be estimated by
\begin{equation}\label{appForm5} |k(\hat{p}) - k(q)| \stackrel{\eqref{appForm1+}}{\lesssim}
\min\{H d_{\mathbb{H}}(p,q)^{\alpha},\epsilon\}. \end{equation}
The first term in \eqref{appForm6} is estimated as follows,
recalling \eqref{appForm8} and \eqref{appForm7}:
\begin{align} |\nabla_{\He}f_{1}(p) - k(\hat{p})| & \stackrel{\eqref{appForm8} \& \eqref{appForm7}}{=} \left|\sum_{s \in S} [f(\hat{s}) + \langle k(\hat{s}), \pi(\hat{s}^{-1} \cdot p)\rangle]\nabla_{\He}\nu_{s}(p) + \nu_{s}(p)[k(\hat{s}) - k(\hat{p})] \right| \notag\\
& \stackrel{\eqref{appForm8}}{\leq} \left| \sum_{s \in S} [f(\hat{s}) - f(\hat{p}) + \langle k(\hat{s}), \pi(\hat{s}^{-1} \cdot \hat{p}) \rangle]
\nabla_{\He}\nu_{s}(p) \right| \notag\\
& \qquad + \left| \sum_{s \in S} [\langle (k(\hat{s}) - k(\hat{p})), \pi(\hat{p}^{-1} \cdot p) \rangle]\nabla_{\He}\nu_{s}(p) \right| \notag\\
&\label{appForm11} \qquad + \left| \sum_{s \in S}
\nu_{s}(p)[k(\hat{s}) - k(\hat{p})] \right| =: \Sigma_{1} +
\Sigma_{2} + \Sigma_{3}. \end{align} To arrive at the expression
for $\Sigma_2$, we added here the term
\begin{displaymath}
\sum_{s \in S} \langle k(\hat{p})), \pi(\hat{p}^{-1} \cdot p)
\rangle \nabla_{\He}\nu_{s}(p) = \langle k(\hat{p})),
\pi(\hat{p}^{-1} \cdot p) \rangle \sum_{s \in S}
\nabla_{\He}\nu_{s}(p) \overset{\eqref{appForm8}}{=} 0.
\end{displaymath}
With the expressions for $\Sigma_1,\Sigma_2$, and $\Sigma_3$ in
hand, we can now continue to bound the right hand side of
\eqref{appForm6}. First, the term $\Sigma_{1}$ essentially
contains $R(\hat{p},\hat{s})$, and can be bounded using
\eqref{appForm2} and \eqref{appForm9}, and the bounded overlap of
the supports of the functions $\nu_{s}$ (in applying
\eqref{appForm2}, note that easily $\hat{p},\hat{s} \in F \cap
B(p_{0},2Cr_{0})$, as $\hat{p},\hat{s}$ are among the points in
$F$ closest to $p,s \in B(p_{0},Cr_{0})$):
\begin{displaymath} \Sigma_{1} \lesssim \frac{d_{\mathbb{H}}(\hat{s},\hat{p})}{r(p)} \cdot \min\{H d_{\mathbb{H}}(\hat{s},\hat{p})^{\alpha},\epsilon\}.   \end{displaymath}
Repeating verbatim the estimate on \cite[p. 251]{EG}, we moreover
find that $d_{\mathbb{H}}(\hat{s},\hat{p}) \lesssim
d_{\mathbb{H}}(p,\hat{p}) = \dist_{\mathbb{H}}(p,F) = r(p)$ for
all $s \in S$ relevant in the summation above, that is, for those
$s \in S$ where $\nu_{s}(p) \neq 0$ or $\nabla_{\He} \nu_{s}(p)
\neq 0$. Consequently,
\begin{displaymath} \Sigma_{1} \lesssim \min\{H d_{\mathbb{H}}(p,\hat{p})^{\alpha},\epsilon\} \leq \min\{H d_{\mathbb{H}}(p,q)^{\alpha},\epsilon\}. \end{displaymath}
Next, to estimate $\Sigma_{2}$, one uses the same ingredients as
above, except that the appeal to \eqref{appForm2} is replaced by
\eqref{appForm1+}:
\begin{displaymath} \Sigma_{2} \lesssim \frac{d_{\mathbb{H}}(\hat{p},p)}{r(p)} \cdot \min\{H d_{\mathbb{H}}(\hat{s},\hat{p})^{\alpha},\epsilon\}
 \lesssim \min\{H d_{\mathbb{H}}(p,q)^{\alpha},\epsilon\}. \end{displaymath}
Virtually the same argument gives the same upper bound for
$\Sigma_{3}$. Starting from \eqref{appForm6}, and recalling
\eqref{appForm5}, one finally infers that
\begin{equation}\label{appForm10} |\nabla_{\He}f_{1}(p) - \nabla_{\He}f_{1}(q)| \lesssim
\min\{H d_{\mathbb{H}}(p,q)^{\alpha},\epsilon\}, \qquad p \in U, \, q \in F. \end{equation}
By symmetry, \eqref{appForm10} also holds if $p \in F$ and $q \in
U$. Combining this with \eqref{appForm1+}, one concludes that
\eqref{appForm10} holds for all pairs $p,q \in \He ^n$ with (a)
both $p,q \in F$, or (b) one point in $F$ and the other one in
$U$. How about the the case (c) $p,q \in U$? The estimate
\begin{equation}\label{appForm13} |\nabla_{\He}f_{1}(p) - \nabla_{\He}f_{1}(q)| \lesssim \epsilon \end{equation}
follows by recalling that $|k(\hat{p}) - (1,0,\ldots,0)| \leq
\epsilon$ and $|k(\hat{q}) - (1,0,\ldots,0)| \leq \epsilon$ by
\eqref{appForm1+}, then repeating the estimate from
\eqref{appForm11} and using the triangle inequality. So, it
remains to show that $|\nabla_{\He}f_{1}(p) -
\nabla_{\He}f_{1}(q)| \lesssim d_{\mathbb{H}}(p,q)^{\alpha/3}$.
The implicit constants here may depend on all the data in
\eqref{appData}. One may assume that $d_{\mathbb{H}}(p,q) \leq 1$,
since otherwise this is implied by \eqref{appForm13}. Consider
first the case where
\begin{equation}\label{appForm12} d_{\mathbb{H}}(p,q) \leq r(p)^{3}. \end{equation}
Recall, once again, the formulae for
$\nabla_{\He}f_{1}(p),\nabla_{\He}f_{1}(q)$ from \eqref{appForm7};
the plan is to make crude term-by-term estimates. Note that if $s
\in S$ is fixed, then
\begin{displaymath} |\nu_{s}(p)k(\hat{s}) - \nu_{s}(q)k(\hat{s})| \lesssim \|\nabla_{\He}\nu_{s}\|_{L^{\infty}}d_{\mathbb{H}}(p,q)
 \lesssim \frac{d_{\mathbb{H}}(p,q)}{r(p)} \leq d_{\mathbb{H}}(p,q)^{2/3}, \end{displaymath}
using \eqref{appForm9} for $j = 1$, and the assumption
\eqref{appForm12}. Similarly, using \eqref{appForm9} for $j = 2$,
\begin{align*} |f(\hat{s})\nabla_{\He}\nu_{s}(p) - f(\hat{s})\nabla_{\He}\nu_{s}(q)| & \lesssim
\|f\|_{L^{\infty}(B(p_{0},2Cr_{0}))}\|\nabla_{\He}^{2}\nu_{s}\|_{L^{\infty}}d(p,q)\\
& \lesssim
\frac{\|f\|_{L^{\infty}(B(p_{0},2Cr_{0}))}}{r(p)^{2}}d_{\mathbb{H}}(p,q)
\lesssim
\|f\|_{L^{\infty}(B(p_{0},2Cr_{0}))}d_{\mathbb{H}}(p,q)^{1/3}.
\end{align*}
Finally, to deal with the last term
\begin{equation}\label{eq:triangle}
\triangle:=|\langle k(\hat{s}), \pi(\hat{s}^{-1} \cdot p)
 \rangle\nabla_{\He}\nu_{s}(p) - \langle k(\hat{s}), \pi(\hat{s}^{-1} \cdot q) \rangle\nabla_{\He}\nu_{s}(q)|
\end{equation}
that arises from $|\nabla_{\He}f_1(p)-\nabla_{\He}f_1(q)|$, we
assume without loss of generality that $\nabla_{\He}\nu_s(p)\neq
0$; if $\nabla_{\He}\nu_s(p)= \nabla_{\He}\nu_s(q)=0$, the
estimate is trivial. As explained between \eqref{appForm11}  and
\eqref{appForm10}, the assumption $\nabla_{\He}\nu_s(p)\neq 0$
ensures that
$
d_{\He}(\hat s,\hat p)\lesssim r(p).
$
Hence we have
\begin{equation}\label{eq:k_est}
|\langle k(\hat s),\pi(\hat s^{-1}\cdot p)\rangle|\lesssim
d_{\He}(\hat s,p)\lesssim d_{\He}(\hat s,\hat p)+ d_{\He}(\hat
p,p)\lesssim r(p).
\end{equation}
This allows us to bound the term $\triangle$ in
\eqref{eq:triangle} as follows
\begin{align*} \triangle
 &\lesssim  |\langle k(\hat s),\pi(\hat s^{-1}\cdot p)\rangle| |\nabla_{\He}\nu_s(p)-\nabla_{\He}\nu_s(q)|
 + |\langle k(\hat s),\pi(\hat s^{-1}\cdot p)-\pi(\hat s^{-1}\cdot q)|
 |\nabla_{\He}\nu_s(q)|\\
 &\overset{\eqref{eq:k_est},\eqref{appForm9}}{\lesssim} r(p)
 \frac{1}{r(p)^2} d_{\He}(p,q)+ d_{\He}(p,q)\frac
 {1}{r(p)} \stackrel{\eqref{appForm12}}{\lesssim} d_{\He}(p,q)^{2/3},
\end{align*}
recalling also that the implicit constants in ``$\lesssim$'' are
allowed to depend on the data in \eqref{appData}.
%
These bounds combined with the bounded overlap of the supports of
the functions $\nu_{s}$ show that
\begin{displaymath} |\nabla_{\He}f_{1}(p) - \nabla_{\He}f_{1}(q)| \overset{\eqref{appForm7}}{\lesssim} d_{\mathbb{H}}(p,q)^{1/3} \end{displaymath}
under the assumption \eqref{appForm12}. Finally, assume that
\begin{equation}\label{appForm14} d_{\mathbb{H}}(p,q) \geq r(p)^{3}. \end{equation}
In this remaining case, one may apply \eqref{appForm10} as
follows:
\begin{align*} |\nabla_{\He}f_{1}(p) - \nabla_{\He}f_{1}(q)| & \leq |\nabla_{\He}f_{1}(p) - k(\hat{p})| + |\nabla_{\He}f_{1}(q) - k(\hat{p})|\\
& \lesssim d_{\mathbb{H}}(p,\hat{p})^{\alpha} +
d_{\mathbb{H}}(q,\hat{p})^{\alpha} \lesssim r(p)^{\alpha} +
d_{\mathbb{H}}(p,q)^{\alpha} \lesssim
d_{\mathbb{H}}(p,q)^{\alpha/3},
\end{align*} using the assumption \eqref{appForm14} in the final
estimate. Recalling also the cases (a)-(b) discussed after
\eqref{appForm10}, it has now been established that $f_{1} \in
C^{1,\alpha/3}_{\He}(\He ^n)$, and \eqref{appForm13} holds for all
$p,q \in \He ^n$. Consequently, \eqref{appForm3} holds if
$\epsilon > 0$ was chosen small enough to  begin with, and then
\begin{equation}\label{appForm3+} X_1f_{1}(p) \geq \tfrac{1}{2}, \qquad p \in \He ^n. \end{equation}
It follows from \eqref{appForm3+} that for every $p \in \He ^n$,
the map $s \mapsto f_{1}(p \cdot (s,0,\ldots,0))$ is strictly
increasing with derivative $\partial_{s}[s \mapsto f_{1}(p \cdot
(s,0,\ldots,0))] = X_1f_{1}(p \cdot (s,0,\ldots,0)) \geq
\tfrac{1}{2}$. Consequently, for every $p \in \He ^n$, the line $p
\cdot \V = \{p \cdot (s,0,\ldots,0) : s \in \R\}$ intersects
$\{f_{1} = 0\}$ in exactly one point, so the set $\{f_{1} = 0\}$
is the intrinsic graph of a certain function $\varphi \colon \W
\to \V$. Recalling that $f_{1} \in C^{1,\alpha/3}_{\He}(\He ^n)$,
and noting \eqref{appForm3+}, the conclusion is that $\{f_{1} =
0\}$ is an intrinsic $C^{1,\alpha/3}$-graph. Moreover, since
$f_{1}(p) = f(p)$ for all $p \in \bar{B}$, the set $S \cap \bar{B}
\subset \{f = 0\} \cap \bar{B}$ is contained on the graph.
Finally, the function $\varphi$ is compactly supported, because
$f_{1}(x_1,\ldots,x_{2n},t) = x_1$ for all $(x_1,\ldots,x_{2n},t)
\in \He ^n \, \setminus \, B(p_{0},Cr_{0})$. Consequently,
\begin{displaymath} \{f_{1} = 0\} \cap [\He ^n\, \setminus \, B(p_{0},Cr_{0})] \subset \{(x_1,\ldots,x_{2n},t) : x_1 = 0\} = \W. \end{displaymath}
This implies that $\varphi \equiv 0$ outside a sufficiently large
ball centred at the origin. The proof of the proposition is
complete.  \end{proof}

\bibliographystyle{plain}
\bibliography{references}

\def\cprime{$'$}
\begin{thebibliography}{10}

\bibitem{MR2223801}
Luigi Ambrosio, Francesco Serra~Cassano, and Davide Vittone.
\newblock Intrinsic regular hypersurfaces in {H}eisenberg groups.
\newblock {\em J. Geom. Anal.}, 16(2):187--232, 2006.

\bibitem{antonelli2019pauls}
Gioacchino Antonelli and Enrico Le~Donne.
\newblock Pauls rectifiable and purely {P}auls unrectifiable smooth
  hypersurfaces.
\newblock {\em Nonlinear Anal.}, 200:111983, 30pp, 2020.

\bibitem{MR2496655}
Gabriella Arena and Raul Serapioni.
\newblock Intrinsic regular submanifolds in {H}eisenberg groups are
  differentiable graphs.
\newblock {\em Calc. Var. Partial Differential Equations}, 35(4):517--536,
  2009.

\bibitem{2017arXiv171103088A}
Jonas {Azzam}.
\newblock {Semi-uniform domains and the $A_{\infty}$ property for harmonic
  measure}.
\newblock {\em Int. Math. Res. Notices \textup{(to appear)}}, page
  arXiv:1711.03088, Nov 2017.

\bibitem{MR2021034}
Zolt\'{a}n~M. Balogh.
\newblock Size of characteristic sets and functions with prescribed gradient.
\newblock {\em J. Reine Angew. Math.}, 564:63--83, 2003.

\bibitem{MR3400438}
F.~Bigolin, L.~Caravenna, and F.~Serra~Cassano.
\newblock Intrinsic {L}ipschitz graphs in {H}eisenberg groups and continuous
  solutions of a balance equation.
\newblock {\em Ann. Inst. H. Poincar\'{e} Anal. Non Lin\'{e}aire},
  32(5):925--963, 2015.

\bibitem{MR2603594}
Francesco Bigolin and Davide Vittone.
\newblock Some remarks about parametrizations of intrinsic regular surfaces in
  the {H}eisenberg group.
\newblock {\em Publ. Mat.}, 54(1):159--172, 2010.

\bibitem{CheegerColding}
Jeff Cheeger and Tobias~H. Colding.
\newblock On the structure of spaces with {R}icci curvature bounded below. {I}.
\newblock {\em J. Differential Geom.}, 46(3):406--480, 1997.

\bibitem{CFO}
V.~{Chousionis}, K.~{F{\"a}ssler}, and T.~{Orponen}.
\newblock {Intrinsic Lipschitz graphs and vertical $\beta$-numbers in the
  Heisenberg group}.
\newblock {\em Amer. J. Math.}, 141(4):1087--1147, Aug 2019.

\bibitem{CFO2}
Vasileios Chousionis, Katrin F{\"a}ssler, and Tuomas Orponen.
\newblock Boundedness of singular integrals on ${C}^{1,\alpha}$ intrinsic
  graphs in the {H}eisenberg group.
\newblock {\em Adv. Math.}, 354:106745, 2019.

\bibitem{MR1096400}
Michael Christ.
\newblock A {$T(b)$} theorem with remarks on analytic capacity and the {C}auchy
  integral.
\newblock {\em Colloq. Math.}, 60/61(2):601--628, 1990.

\bibitem{MR2247905}
Daniel~R. Cole and Scott~D. Pauls.
\newblock {$C^1$} hypersurfaces of the {H}eisenberg group are
  {$N$}-rectifiable.
\newblock {\em Houston J. Math.}, 32(3):713--724, 2006.

\bibitem{DS1}
G.~David and S.~Semmes.
\newblock Singular integrals and rectifiable sets in {${\bf R}^n$}: Au-del\`a
  des graphes lipschitziens.
\newblock {\em Ast\'erisque}, (193):152, 1991.

\bibitem{MR2907827}
Guy David and Tatiana Toro.
\newblock Reifenberg parameterizations for sets with holes.
\newblock {\em Mem. Amer. Math. Soc.}, 215(1012):vi+102, 2012.

\bibitem{EG}
Lawrence~C. Evans and Ronald~F. Gariepy.
\newblock {\em Measure theory and fine properties of functions}.
\newblock Studies in Advanced Mathematics. CRC Press, Boca Raton, FL, 1992.

\bibitem{fssler2018riesz}
Katrin {F{\"a}ssler} and Tuomas {Orponen}.
\newblock {Riesz transform and vertical oscillation in the Heisenberg group}.
\newblock {\em Anal. PDE (to appear)}, arXiv:1810.13122.

\bibitem{fssler2019singular}
Katrin {F{\"a}ssler} and Tuomas {Orponen}.
\newblock {Singular integrals on regular curves in the Heisenberg group}.
\newblock {\em J. Math. Pures Appl. (to appear)}, arXiv:1911.03223.

\bibitem{FSSC}
B.~Franchi, R.~Serapioni, and F.~Serra~Cassano.
\newblock Rectifiability and perimeter in the {H}eisenberg group.
\newblock {\em Math. Ann.}, 321(3):479--531, 2001.

\bibitem{FSS}
B.~Franchi, R.~Serapioni, and F.~Serra~Cassano.
\newblock Intrinsic {L}ipschitz graphs in {H}eisenberg groups.
\newblock {\em J. Nonlinear Convex Anal.}, 7(3):423--441, 2006.

\bibitem{MR2836591}
Bruno Franchi, Raul Serapioni, and Francesco Serra~Cassano.
\newblock Differentiability of intrinsic {L}ipschitz functions within
  {H}eisenberg groups.
\newblock {\em J. Geom. Anal.}, 21(4):1044--1084, 2011.

\bibitem{kozhevnikov:tel-01178864}
Artem Kozhevnikov.
\newblock {\em {Metric properties of level sets of differentiable maps on
  Carnot groups}}.
\newblock Theses, {Universit{\'e} Paris Sud - Paris XI}, May 2015.

\bibitem{LDY}
Enrico {Le Donne} and Robert {Young}.
\newblock {Carnot rectifiability of sub-Riemannian manifolds with constant
  tangent}.
\newblock {\em arXiv e-prints}, page arXiv:1901.11227, Jan 2019.

\bibitem{MontiThesis}
R.~Monti.
\newblock Distances, boundaries, and surface measures in
  {C}arnot-{C}arath\'eodory spaces.
\newblock {\em PhD Thesis, Universit\'a degli Studi di Trento}, 2001.

\bibitem{MR2048183}
Scott~D. Pauls.
\newblock A notion of rectifiability modeled on {C}arnot groups.
\newblock {\em Indiana Univ. Math. J.}, 53(1):49--81, 2004.

\bibitem{MR3587666}
F.~Serra~Cassano.
\newblock Some topics of geometric measure theory in {C}arnot groups.
\newblock In {\em Geometry, analysis and dynamics on sub-{R}iemannian
  manifolds. {V}ol. 1}, EMS Ser. Lect. Math., pages 1--121. Eur. Math. Soc.,
  Z\"urich, 2016.

\end{thebibliography}

\end{document}